\documentclass[final]{siamltex}

\usepackage{amsmath,amssymb,mathrsfs}
\usepackage{enumerate}
\usepackage{epsfig,psfrag}
\usepackage{pifont}
\usepackage{subfigure}

\allowdisplaybreaks

\newcommand{\done}{\ding{182}}
\newcommand{\dtwo}{\ding{183}}
\newcommand{\dthree}{\ding{184}}
\newcommand{\dfour}{\ding{185}}
\newcommand{\dfive}{\ding{186}}

\newtheorem{assumption}{Assumption}% [theorem]

\DeclareMathOperator*{\argmax}{arg\,max}
\DeclareMathOperator*{\stat}{stat}
\DeclareMathOperator*{\argstat}{arg\,stat}

\newcommand{\ba}			{\begin{array}}
\newcommand{\ea}			{\end{array}}
\newcommand{\nn}			{\nonumber}

\newcommand{\dom}		{{\textsf{dom\,}}}
\newcommand{\ran}			{{\textsf{ran\,}}}

\newcommand{\kernel}		{{\textsf{ker\,}}}

\newcommand{\bo}			{{\mathcal{L}}}
\newcommand{\ddtone}[2]	{{\frac{d {#1}}{d {#2}}}}
\newcommand{\ddttwo}[2]	{{\frac{d^2 {#1}}{d {#2}^2}}}
\newcommand{\er}[1]		{{(\ref{#1})}}
\newcommand{\funspace}[1]	{{\mathscr{#1}}}
\newcommand{\ol}[1]		{{\overline{#1}}}
\newcommand{\ul}[1]		{{\underline{#1}}}
\newcommand{\op}[1]		{{\mathcal{#1}}}

\newcommand{\opbreve}[1]	{{\breve{\op{#1}}}}
\newcommand{\optilde}[1]	{{\tilde{\op{#1}}}}
\newcommand{\pdtone}[2]	{{\frac{\partial {#1}}{\partial {#2}}}}
\newcommand{\pdttwo}[2]		{{\frac{\partial^2 {#1}}{\partial {#2}^2}}}
\newcommand{\ts}[1]		{{\textstyle{#1}}}

\newcommand{\ooer}[1]			{\normalfont\tag{\ref{#1}}}

\newcommand{\cB}			{{\funspace{B}}}

\newcommand{\cL}			{{\mathcal{L}}}

\newcommand{\cX}			{{\funspace{X}}}

\newcommand{\half}			{{\frac{1}{2}}}
\newcommand{\minushalf}	{{-\frac{1}{2}}}
\newcommand{\cV}			{{\funspace{V}}}
\newcommand{\cW}			{{\funspace{W}}}
\newcommand{\cY}			{{\funspace{Y}}}
\newcommand{\cZ}			{{\funspace{Z}}}

\newcommand{\demi}		{{\ts{\frac{1}{2}}}}
\newcommand{\eps}			{{\epsilon}}
\newcommand{\ggrad}		{{\nabla}}
\newcommand{\Lone}		{{\funspace{L}_1}}
\newcommand{\Ltwo}		{{\funspace{L}_2}}
\newcommand{\ltwo}		{{$\Ltwo$}}

\newcommand{\mapsinto}	{{\rightarrow}}

\newcommand{\R}			{{\mathbb{R}}}
\newcommand{\N}			{{\mathbb{N}}}

\newcommand{\TPBVP}		{{\textsf{TPBVP}}}

\newcommand{\Frechet}			{{Fr\'{e}chet}}

\newcommand{\opAsqrt}		{{\op{A}^\frac{1}{2}}}

\newcommand{\Wbreve}		{{\breve W}}

\newcommand{\What}		{{\widehat W}}

\newcommand{\opdot}[1]		{{\dot{\op{#1}}}}
\newcommand{\ophat}[1]		{{\widehat{\op{#1}}}}

\begin{document}

%%%%%%%%%%%%%%%%%%%%%%%%%%%%%%%%%%%%%%%%%%%%%%%%%%%%%%%%%%%%%%%%%%%%%%%%%%
%%
%%		Title
%%

\title{% Stationary action and a fundamental solution for two-point boundary value problem
Solving two-point boundary value problems for a wave equation via the principle of stationary action and optimal control
% A fundamental solution for an infinite dimensional two-point boundary value problem via the principle of stationary action%
%
\thanks{This research was partially supported by AFOSR, NSF, and the Australian Research Council. Preliminary results contributing to this paper appear in \cite{DM2:13}.}}

\author{Peter M. Dower\thanks{Department of Electrical \& Electronic Engineering, University of Melbourne, Parkville, Victoria 3010, Australia. Email: {\tt pdower@unimelb.edu.au}}
\and
William M. McEneaney\thanks{Department of Mechanical \& Aerospace Engineering, University of California at San Diego, 9500 Gilman Dr., La Jolla, CA 92093-04111, USA. Email: {\tt wmceneaney@eng.ucsd.edu.au}}}
\date{}

\maketitle

%%%%%%%%%%%%%%%%%%%%%%%%%%%%%%%%%%%%%%%%%%%%%%%%%%%%%%%%%%%%%%%%%%%%%%%%%%
%%
%%		Abstract
%%

\begin{abstract}  \small\baselineskip=9pt 
A new approach to solving two-point boundary value problems for a wave equation is developed. This new approach exploits the principle of stationary action to reformulate and solve such problems in the framework of optimal control. In particular, an infinite dimensional optimal control problem is posed so that the wave equation dynamics and temporal boundary data of interest are captured via the characteristics of the associated Hamiltonian and choice of terminal payoff respectively. In order to solve this optimal control problem for any such terminal payoff, and hence solve any two-point boundary value problem corresponding to the boundary data encapsulated by that terminal payoff, a fundamental solution to the optimal control problem is constructed. Specifically, the optimal control problem corresponding to any given terminal payoff can be solved via a max-plus convolution of this fundamental solution with the specified terminal payoff. Crucially, the fundamental solution is shown to be a quadratic functional that is defined with respect to the unique solution of a set of operator differential equations, and computable using spectral methods. An example is presented in which this fundamental solution is computed and applied to solve a two-point boundary value problem for the wave equation of interest.
\end{abstract}



%%%%%%%%%%%%%%%%%%%%%%%%%%%%%%%%%%%%%%%%%%%%%%%%%%%%%%%%%%%%%%%%%%%%%%%%%%
%%
%%		Introduction
%%

\section{Introduction}
The principle of stationary action, or {\em action principle}, states that any trajectory generated by a conservative system must render the corresponding {\em action functional} stationary in the calculus of variations sense, see for example \cite{F:48,F:64,GT:07}. As this action functional is defined as the time integral of the associated Lagrangian, it may be regarded as the payoff due to a unique trajectory generated by some generalized system dynamics, and corresponding to a specified initial system state. By regarding the velocity of these generalized dynamics as an input, the action principle may be expressed as an optimal control problem. Recent work by the authors has exploited this correspondence with optimal control to develop a fundamental solution to the classical gravitational $N$-body problem, see \cite{MD1:13,MD2:13}. This fundamental solution is a special case of a more general notion of a fundamental solution semigroup developed for optimal control problems, see for example \cite{M:08,DM1:14,ZD1:15,DMZ1:15}. In the specific case of the gravitational $N$-body problem, by constructing a fundamental solution to the optimal control problem corresponding to stationary action, a fundamental solution to a class of $N$-body two-point boundary value problems (TPBVPs) may also be constructed. % This fundamental solution admits the efficient computation of solutions to the $N$-body problem subject to initial and final time position and velocity constraints. 
In this paper, the corresponding fundamental solution construction for a class of TPBVPs is extended via infinite dimensional systems theory to a wave equation \cite{P:83,CZ:95,DM:11,DM:12,DM2:13,DM1:14}. The specific wave equation considered is expressed via the partial differential equation (PDE) and boundary data
\begin{equation}
	\begin{aligned}
		& \pdttwo{u}{s}
		= \left( \frac{\kappa}{m} \right) \pdttwo{u}{\lambda}\,.
		\qquad u(\cdot,0) = 0 = u(\cdot,L)\,, \ L\in\R_{>0}\,.
	\end{aligned}
	\label{eq:wave}
\end{equation}
In a mechanical setting (for example), $u(s,\lambda)$ may be interpreted as the displacement of a vibrating string at time $s\in[0,\bar t]$, $\bar t\in\R_{>0}$, and location $\lambda\in\overline{\Lambda}$, $\Lambda\doteq(0,L)$. Here, constants $\kappa, m\in\R_{>0}$ model the distributed elastic spring constant and mass respectively (with SI units of $\mathsf{N}$ and $\mathsf{kg\, m^{-1}}$). An example pair of initial and terminal conditions defining a TPBVP of interest is
\begin{equation}
	u(0,\cdot) = x(\cdot)\,, \
	u(t,\cdot) = z(\cdot)\,,
	\label{eq:IC-FC}
\end{equation}
in which $x$ and $z$ denote the initial and terminal displacements. The problem to solve is then
\begin{align}
	\TPBVP(t, x, z)
	\doteq \left\{ \ba{c}
	\text{Find the initial velocity $\pdtone{u}{s}(0,\cdot)$}
	\\
	\text{(if it exists) such that \er{eq:wave} and \er{eq:IC-FC}}
	\\
	\text{hold with functions $x$ and $z$ given.}
	\ea \right.
	\label{eq:TPBVP}
\end{align}
(Another example of a TPBVP of interest is to determine the initial velocity such that a desired terminal velocity is attained.) In order to formulate the action principle for system \er{eq:wave}, note that the potential and kinetic energies associated with a solution $u(s,\cdot)$ of \er{eq:wave} at time $s\in[0,t]$ are respectively denoted by $V(u(s,\cdot))$ and $T(\ts{\pdtone{u}{s}}(s,\cdot))$, where $V$ and $T$ are the functionals defined by
\begin{align}
	V(u(s,\cdot))
	& = \ts{\frac{\kappa}{2}} \, \int_\Lambda \left| \pdtone{u}{\lambda}(s,\lambda) \right|^2\, d\lambda\,,
	\qquad
	T\left(\pdtone{u}{s}(s,\cdot) \right)
	=  \ts{\frac{m}{2}} \, \int_\Lambda \left| \pdtone{u}{s}(s,\lambda) \right|^2\, d\lambda\,.
	\label{eq:V-and-T-wave}
\end{align}
The action principle states that any solution $u$ of \er{eq:wave} must render the action functional 
\begin{align}
	\int_0^t V(u(s,\cdot)) - T\left(\pdtone{u}{s}(s,\cdot) \right) \, ds
	\label{eq:action}
\end{align}
stationary in the sense of the calculus of variations \cite{GT:07}, where $V$ and $T$ denote the energy functionals as per \er{eq:V-and-T-wave}, and the integrand is the Lagrangian or its additive inverse (as selected here). This includes any solution of the TPBVP \er{eq:TPBVP}. Hence, by formulating an appropriate optimal control problem encapsulating this variational problem, solutions of the TPBVP \er{eq:TPBVP} may be investigated. In particular, a fundamental solution to the TPBVP \er{eq:TPBVP} can be constructed within the framework of infinite dimensional optimal control, using a more general notion of fundamental solution semigroup for optimal control \cite{M:08,DM1:14,ZD1:15,DMZ1:15}.  The attendant optimal control problem and subsequent TPBVP fundamental solution is formulated and developed in Sections \ref{sec:optimal} and \ref{sec:fundamental}. Useful auxiliary optimal control problems, and their interrelationships, are employed in this development. The application of this fundamental solution is then considered in the context of an illustrative example in Section \ref{sec:example}. Selected technical details of relevance to the development are included in the appendices.

In terms of the notation, $\R$, $\R_{\ge 0}$, and $\R_{>0}$ denotes the sets of reals, non-negative reals, and positive reals respectively. Given an open subset $\mathbb{D}$ of a Euclidean space and a Banach space $\cZ$, the respective spaces of continuous, continuously differentiable, and Lebesgue square integrable functions mapping $\mathbb{D}$ to $\cZ$ are denoted by $C(\mathbb{D};\, \cZ)$, $C^1(\mathbb{D};\cZ)$, and $\Ltwo(\mathbb{D}; \cZ)$. Symbols $\partial$ and $\partial^2$ denote first and second order differentation for functions defined on $\Lambda$. An operator $\mathcal{O}:\cX\mapsinto\cY$ between Banach spaces $\cX$ and $\cY$ is {\Frechet} differentiable at $x\in\cX$ if there exists a bounded linear operator $d\mathcal{O}(x)\in\cL(\cX;\cY)$ such that the limit 
$
	\lim_{\|h\|_\cX\rightarrow 0} \|\mathcal{O}(x+h) - \mathcal{O}(x) - d\mathcal{O}(x)\, h \|_\cY / \|h\|_\cX
	% \label{eq:Frechet}
$
exists and is zero. % Here, $d\mathcal{O}:\cX\mapsinto\cL(\cX;\cY)$ is the {\Frechet} derivative of $\mathcal{O}$, $d\mathcal{O}(x)\in\cL(\cX;\cY)$ is the {\Frechet} derivative of $\mathcal{O}$ at $x\in\cX$, and $d\mathcal{O}(x)\, h\in\cY$ is the {\Frechet} differential at $x\in\cX$ in direction $h\in\cX$.

%%%%%%%%%%%%%%%%%%%%%%%%%%%%%%%%%%%%%%%%%%%%%%%%%%%%%%%%%%%%%%%%%%%%
%%
%%		Optimal control
%%

\section{Approximating stationary action via optimal control}
\label{sec:optimal}

Where the action functional is concave or convex, the action principle can be formulated as an optimal control problem, see for example \cite{MD1:13,MD2:13}. However, this convexity or concavity, corresponding to that of the payoff or cost functional, is limited to a finite time horizon that is determined by parameters associated with the kinetic and potential energies. In the finite dimensional case, this limited time horizon is strictly positive, so that the conservative dynamics defined by the action principle can be propagated via solution of the optimal control problem up to that time horizon. However, in the infinite dimensional case considered here, this limited time horizon tends to zero, see Theorem \ref{thm:second-difference} and \cite{DM2:13}, thereby complicating the direct application of the approach of \cite{MD1:13,MD2:13}.  In order to overcome this complication, a perturbed optimal control problem is formulated that approximates the stationary action principle on a strictly positive time horizon, thereby allowing the solution of the TPBVP \er{eq:TPBVP} to be approximated on that time horizon. By concatenating such horizons via the dynamic programming principle, solutions on longer horizons can also be approximated. Such approximations are shown (using well-known semigroup approximation results) to be exact in the limit of vanishing perturbations, see Section \ref{sec:approx-solution}.

%%		Preliminaries

\subsection{Preliminaries}
Define an {\ltwo} and Sobolev space by
\begin{align}
	\cX
	& \doteq \Ltwo(\Lambda;\R)\,,
	% \label{eq:cX}
	% \\
	\qquad
	\cX_0
	\doteq 
	\left\{ x\in\cX \, \left| \, \ba{c} 
		x,\, \partial x \text{ absolutely continuous,} \\
		x(0) = 0 = x(L)\,,
		\\
		\partial^2 x\in\cX
	\ea \right. \right\}\,,
	\label{eq:cX-0}
\end{align}
and let $\langle\, , \rangle$ and $\|\cdot\|$ denote the standard {\ltwo} inner product and norm on $\cX$.
A specific unbounded operator $\op{A}$ of interest in considering the wave equation \er{eq:wave} is densely defined on $\cX$ by
\begin{align}
	\op{A}\, x
	= (\op{A}\, x)(\cdot)
	& \doteq - \partial^2 x(\cdot)\,,
	\quad
	\dom(\op{A})
	\doteq \cX_0\subset\cX, \ \ol{\cX_0} = \cX\,.
	\label{eq:op-A}
\end{align}
% As $\ol{\cX_0}=\cX$, $\op{A}$ is densely defined. 
Operator $\op{A}$ is closed, positive, self-adjoint, and boundedly invertible, and has a unique, positive, self-adjoint, boundedly invertible square root, denoted by $\opAsqrt$. The inverse of this square root is denoted by $\op{J}\doteq(\opAsqrt)^{-1}\in\mathcal{L}(\cX)$. 
% The additive inverse operator $-\op{A}$ is the generator of a $C_0$-semigroup (although $\op{A}$ is not). 
See Appendix \ref{app:op-A-properties} (Lemma \ref{lem:op-A-properties}), \cite{CZ:95} (Example 2.2.5, Lemma A.3.73, and Examples A.4.3 and A.4.26), and also \cite{B:68}. These properties admit the definition of Hilbert spaces
\begin{align}
	\cX_\half
	& \doteq \dom(\opAsqrt)\,,
	&& \langle x, \, \xi \rangle_\half \doteq \langle \opAsqrt\, x,\, \opAsqrt\, \xi \rangle\,,
	&& \forall\ x,\xi\in\cX_\half\,,
	\label{eq:cX-half}
	\\
	\cY_\half
	& \doteq \cX_\half\oplus \cX\,,
	&& \langle (x,p) , (\xi,\pi) \rangle_\oplus \doteq m\, \langle x,\, \xi\rangle_\half + \ts{\frac{1}{\kappa}}\, \langle p,\, \pi\rangle\,,
	&& \forall\ x,\xi\in\cX_\half,\, p,\pi\in\cX\,.
	\label{eq:cY-oplus}
\end{align}
The corresponding norms are denoted by $\|\cdot\|_\half$ and $\|\cdot\|_\oplus$. 
% Recalling the definition of $\op{J}$, note that $\dom(\op{J}) = \cX\supset\cX_\half = \ran(\op{J})$, and $\cX_0\subset\cX_\half\subset\cX$. 
Similarly, it is also convenient to define the set
\begin{align}
	\cY_0
	& \doteq \cX_0\oplus \cX_\half \subset\cY_\half\,.
	\label{eq:cY-0}
\end{align}
Operators $\op{A}$ and $\opAsqrt$ are Riesz-spectral operators, see Appendix \ref{app:Riesz} and \cite{CZ:95}. Define orthonormal Riesz bases 
\begin{align}
	\begin{aligned}
	\op{B}
	& \doteq \{ \varphi_n \}_{n=1}^\infty\,,
	&
	\varphi_n(\cdot) & \doteq \ts{\sqrt{\frac{2}{L}}} \, \sin(\ts{\frac{n\, \pi}{L}}\, \cdot)\,,
	\\
	\optilde{B} & \doteq \{ \tilde\varphi_n \}_{n=1}^\infty\subset\cX_\half\,,
	&
	\tilde\varphi_n(\cdot) & \doteq \ts{\frac{\sqrt{2\, L}}{n\, \pi}} \, \sin(\ts{\frac{n\, \pi}{L}}\, \cdot)\,,
	\end{aligned}
	\label{eq:basis-Riesz}
\end{align}
for $\cX$ and $\cX_\half$ respectively (see Lemma \ref{lem:basis-Riesz}). The input space for the optimal control problem of interest is 
\begin{align}
	\cW[r,t]
	& \doteq \Ltwo([r,t];\cX_\half)
	\label{eq:cW}
\end{align}
for all $t\in\R_{\ge 0}$, $r\in[0,t]$. The corresponding norm is defined by $\|w\|_{\cW[r,t]}^2 \doteq \int_r^t \|w(s)\|_\half^2 \, ds$.

%%
%%		Optimal control problem.
%%

\subsection{Approximating optimal control problem}
\label{ssec:approx}
In order to formulate the action principle for the conservative infinite dimensional dynamics of \er{eq:wave}, define the abstract Cauchy problem \cite{P:83,CZ:95} by
\begin{align}
	\dot \xi(s)
	& = w(s)\,,
	\quad \xi(0) = x\in\cX_\half\,,
	\label{eq:dynamics}
\end{align}
in which $\xi(s)$ denotes the infinite dimensional state at time $s\in[0,t]$ that has evolved from initial state $x\in\cX_\half$ in the presence of input $w\in\cW[0,s]$. The derivative in \er{eq:dynamics} is of {\Frechet} type, defined with respect to the norm $\|\cdot\|_\half$. The mild solution \cite{P:83,CZ:95} of \er{eq:dynamics} is defined as
\begin{align}
	\xi(s)
	& = x + \int_0^s w(\sigma) \, d\sigma
	\label{eq:mild}
\end{align}
for all $x\in\cX_\half$, $w\in\cW[0,s]$, $s\in[0,t]$. In view of these dynamics, define the payoff (action) functional $J_{m,\psi}^\mu:[0,\bar t)\times\cX_\half\times\cW[0,\bar t)\mapsinto\R$ for some $\bar t\in\R_{>0}$ by
\begin{align}
	J_{m,\psi}^\mu(t,x,w)
	& \doteq \hspace{-1mm}
	\int_0^t \hspace{-1mm} 
					\ts{\frac{\kappa}{2}} \, \|\xi(s)\|_\half^2 - \ts{\frac{m}{2}}\, \|\op{J}^\mu\, w(s)\|_\half^2 
					ds
	+ \psi(\xi(t))\,,
	\label{eq:payoff}
\end{align}
in which $\kappa,\, m\in\R_{>0}$ are physical constants as per \er{eq:wave}, $\mu\in\R_{>0}$ is a real-valued perturbation parameter, $\op{J}^\mu:\cX_\half\mapsinto\cX_0\times\cX_\half$ is a bounded linear operator given by
\begin{align}
	\op{J}^\mu\, w
	& \doteq \left[ \ba{c} \op{J} \\ \mu\, \op{I} \ea \right] \, w\,,
	\label{eq:op-J-mu}
\end{align}
$\op{I}$ is the identity operator on $\cX_\half$, and $\psi:\cX_\half\mapsinto\R$ is any concave terminal payoff. In the integrand in \er{eq:payoff}, note that $\|\xi(s)\|_\half \equiv \|\partial\xi(s)\|$ is well-defined as $\xi(s)\in\cX_\half$ for each $s\in[0,t]$. Note also that $\|\op{J}\, w(s)\|_\half\equiv \|w(s)\|$ is well-defined as $\ran(\op{J}) = \cX_\half$. Consequently, \er{eq:payoff} approximates the action functional \er{eq:action} as
\begin{align}
	&
	\int_0^t V(u(s,\cdot)) - T^\mu\left(\pdtone{u}{s}(s,\cdot) \right) \, ds
	\label{eq:action-mu}
\end{align}
where $T^\mu:\cX_\half\mapsinto\R$ is the approximate kinetic energy functional defined analogously to \er{eq:V-and-T-wave} by 
\begin{align}
	T^\mu\left(\pdtone{u}{s}(s,\cdot) \right)
	& \doteq
	\ts{\frac{m}{2}} \left( \left\| \pdtone{u}{s}(s,\cdot) \right\|^2
	% \ts{\frac{m}{2}} \, \int_\Lambda \left| \pdtone{u}{s}(s,\lambda) \right|^2 \, d\lambda
	+ \mu^2 \left\| \pdtone{u}{s}(s,\cdot) \right\|_\half^2
	\right)
	% + \mu^2\, \left\langle \pdtone{u}{s} (s,\lambda),\, -\pdttwo{}{\lambda}\pdtone{u}{s} (s,\lambda) \right\rangle \, d\lambda\,,
	= \ts{\frac{m}{2}} \, \|\op{J}^\mu\, w(s) \|_\half^2\,.
	\label{eq:T-mu}
\end{align}
Note in particular that the $\|\cdot\|_\half^2$ term introduces a penalty on spatial ripples in $\ts{\pdtone{u}{s}}(s,\cdot)$ for $\mu\ne 0$. This term vanishes for $\mu=0$, so that $T = T^0$ by inspection of \er{eq:V-and-T-wave} and \er{eq:T-mu}.

In order to ensure that the optimal control problem defined via the payoff functional $J_{m,\psi}^\mu$ of \er{eq:payoff} has a finite value, it is critical to establish the existence of a $\bar t^\mu\in\R_{>0}$ in \er{eq:payoff} such that $J_{m,\psi}^\mu(t,x,\cdot)$ is either convex or concave for all $t\in[0,\bar t^\mu)$ and $x\in\cX_\half$. To this end, it may be shown (see Theorem \ref{thm:second-difference} at the end of this section) that the second difference $\Delta J_{m,\psi}^\mu(t,x,w^*\!\!,\delta\, \tilde w) \doteq J_{m,\psi}^\mu(t,x,w^*\!\!+\delta\tilde w) - 2\, J_{m,\psi}^\mu(t,x,w^*) + J_{m,\psi}^\mu(t,x,w^*\!\!-\delta\tilde w)$ of $J_{m,\psi}^\mu(t,x,\cdot)$ for an input $w^*\in\cW[0,t]$ in direction $\tilde w\in\cW[0,t]$, $\delta\in\R$ (with $\delta\, \|\tilde w\|_{\cW[0,t]} \ne 0$) satisfies
\begin{align}
	\Delta J_{m,\psi}^\mu(t,x,w^*\!\!,\delta\, \tilde w)
	& \le -\delta^2 \left [m\, \mu^2 - \kappa\, ( \ts{\frac{t^2}{2}} ) \right] \|\tilde w\|_{\cW[0,t]}^2 < 0
	\label{eq:second-difference-bound}
\end{align}
for all $t\in[0, \bar t^\mu)$, provided the terminal payoff $\psi$ is concave, where 
\begin{align}
	\bar t^\mu
	& \doteq \mu\, \ts{({\ts{\frac{2\, m}{\kappa}}})^\half}\,.
	\label{eq:t-bar-mu}
\end{align}
That is, the payoff functional $J_{m,\psi}^\mu(t,x,\cdot)$ of \er{eq:payoff} is strictly concave under these conditions. Consequently, the approximate action principle (modified to include a concave terminal payoff $\psi$, and perturbed by $\mu\in\R_{>0}$) may be expressed via the value function $W^\mu:\R_{\ge 0}\times\cX_\half\mapsinto\R$,
\begin{align}
	W^\mu(t,x)
	& \doteq \sup_{w\in\cW[0,t]} J_{m,\psi}^\mu(t,x,w)\,.
	\label{eq:W}
\end{align}
By interpreting \er{eq:W} as an optimal control problem, it is shown that the state feedback characterization of the optimal (velocity) input for the approximate action principle is defined via $w^*(s) = k(s,\xi^*(s))$, where $k(s,x) \doteq \ts{\frac{1}{m}}\, \opAsqrt\, \op{I}_\mu\, \opAsqrt\, \ggrad_x W^\mu(t-s,x)$. Here, $\xi^*(\cdot)$ denotes the trajectory \er{eq:dynamics} corresponding to input $w^*$, and $\op{I}_\mu$ is a self-adjoint bounded linear operator that approximates the identity for small $\mu\in\R_{>0}$ (to be defined later). Consequently, by selecting a terminal payoff that forces the terminal displacement $\xi(t)$ to $z$ (fixed apriori as per \er{eq:TPBVP}), the corresponding initial velocity required to achieve this terminal displacement is shown to be
\begin{align}
	w(0) & = \ts{\frac{1}{m}}\, \opAsqrt\, \op{I}_\mu\, \opAsqrt \, \ggrad_x W^\mu(t,\xi^*(0)) 
	= \ts{\frac{1}{m}}\, \opAsqrt\, \op{I}_\mu\, \opAsqrt \, \ggrad_x W^\mu(t,x)\,.
	\nn
\end{align}
The characteristic equations corresponding to the Hamiltonian associated with \er{eq:W} imply that this initial velocity determines the corresponding initial momentum costate. Here, it is convenient to define a scaled costate $\pi(s) \doteq m\, \op{I}_\mu^{-\half}\, w(s)$, so that the initialization $\pi(0) = p\doteq m\, \op{I}_\mu^{-\half}\, w(0)$ ultimately yields the terminal displacement $z$ after evolution of the state and costate dynamics to time $t\in(0,\bar t^\mu)$. This evolution is governed by the abstract Cauchy problem 
\begin{align}
	\left( \ba{c} \dot\xi(s) \\ \dot \pi(s) \ea \right)
	& = \op{A}_\mu^\oplus
	\left( \ba{c} \xi(s) \\ \pi(s) \ea \right)\,,
	\quad 
	\op{A}_\mu^\oplus \doteq
	\left( \ba{cc}
		0 & \ts{\frac{1}{m}} \, \op{I}_\mu^\half
		\\
		-\kappa\, \opAsqrt\, \op{I}_\mu^\half\, \opAsqrt & 0 
	\ea \right)\,,
	\quad
	\dom(\op{A}_\mu^\oplus) \doteq \cY_\half\,,
	\label{eq:Cauchy-mu}
\end{align}
in which $\xi(s)$ and $\pi(s)$ denote the state and costate at time $s\in[0,t]$, evolved from $\xi(0) = x$ and $\pi(0) = p$. The uniformly continuous semigroup of bounded linear operators $\op{T}_\mu^\oplus(s)\in\bo(\cY_\half)$ generated by $\op{A}_\mu^\oplus\in\bo(\cY_\half)$ yields solutions of \er{eq:Cauchy-mu} of the form
\begin{align}
	\left( \ba{c}
		\xi(s) \\ \pi(s)
	\ea \right)
	& = \op{T}_\mu^\oplus(s) \left( \ba{c}
			\xi \\ \pi 
		\ea \right),
	\label{eq:wave-solution-mu}
\end{align}
for all $s\in[0,t]$, with $\xi$ solving an approximation of the wave equation \er{eq:wave} given by
\begin{align}
	& \ddot\xi(s) = -(\ts{\frac{\kappa}{m}}) \, \opAsqrt\, \op{I}_\mu\, \opAsqrt\, \xi(s)
	\label{eq:wave-mu}
\end{align}
for $s\in\R_{\ge 0}$. Furthermore, $\op{A}_\mu^\oplus$ is shown to converge strongly (as $\mu\rightarrow 0$) to an unbounded, closed, and densely defined operator $\op{A}^\oplus$ on $\cY_0\doteq\dom(\op{A}^\oplus) \doteq \cX_0\times\cX_\half$ . This operator defines the related abstract Cauchy problem
\begin{align}	
	\left( \ba{c} \dot x(s) \\ \dot p(s) \ea \right)
	& = \op{A}^\oplus
	\left( \ba{c} x(s) \\ p(s) \ea \right)\,,
	\quad 
	\op{A}^\oplus 
	\doteq
	\left( \ba{cc}
		0 & \ts{\frac{1}{m}} \, \op{I}
		\\
		-\kappa\, \op{A} & 0 
	\ea \right),
	\quad
	\dom(\op{A}^\oplus)\doteq \cY_0 \doteq \cX_0\oplus\cX_\half\,,
	\label{eq:Cauchy}
\end{align}
and is the generator of the $C_0$-semigroup of bounded linear operators $\op{T}^\oplus(t)\in\bo(\cY_\half)$, $t\in\R_{\ge 0}$, yielding all solutions of \er{eq:Cauchy} of the form 
\begin{align}
	\left( \ba{c} x(s) \\ p(s) \ea \right)
	& = \op{T}^\oplus(s) \left( \ba{c} x \\ p \ea \right),
	\label{eq:wave-solution}
\end{align}
in which $x(s)$ and $p(s)$ denote the state and costate analogously to \er{eq:wave-solution-mu}. Crucially, the state $x$ is the solution is the wave equation \er{eq:wave} itself. As the first Trotter-Kato theorem (e.g. \cite{EN:00}) implies that the semigroup $\op{T}_\mu^\oplus(t)$ converges strongly to $\op{T}^\oplus(t)$ for $t\in\R_{\ge 0}$ on bounded intervals, solutions \er{eq:wave-solution-mu} of \er{eq:wave-mu} converge to solutions \er{eq:wave-solution} of \er{eq:wave} as $\mu\rightarrow 0$. In this sense, solutions of the TPBVP \er{eq:TPBVP} defined with respect to the wave equation \er{eq:wave} may be approximated via the optimal control problem defined by \er{eq:W}. 

Where the terminal velocity is specified (rather than the terminal position as in \er{eq:TPBVP}), this same approach may be applied by employing the terminal payoff
\begin{align}
	\psi(x)
	= \psi_{v}(x)
	& \doteq m\, \langle \op{J}\, \op{J}\, v, \, x \rangle_\half\,.
	\label{eq:stationary-payoff}
\end{align}
In that case, the terminal momentum costate is given by 
$$
	\pi(t) = \opAsqrt\, \op{I}_\mu^\half\, \opAsqrt\,\ggrad_x W^\mu(0,\xi^*(t)) 
	= \opAsqrt\, \op{I}_\mu^\half\, \opAsqrt\, \ggrad_x \psi_v(\xi^*(t)) 
	= m\, \op{I}_\mu^\half\, v\,.
$$
% which converges to $m\, v$ as $\mu\rightarrow 0$ as required. 
Hence, by solving the optimal control problem \er{eq:W} defined with respect to terminal payoff $\psi_{v}$ of \er{eq:stationary-payoff}, the infinite dimensional dynamics of \er{eq:wave-mu} can be propagated forward from a known initial position $x\in\cX_0\subset\cX_\half$ to a known terminal velocity $\dot\xi(t) =  \op{I}_\mu\, v\in\cX_0$. As in the terminal position case, this approximation converges to the actual wave equation dynamics  satisfying $\dot\xi(t) = v\in\cX_\half$ as $\mu\rightarrow 0$. The rigorous development yielding this conclusion commences with a theorem concerning the concavity of the payoff functional $J_{m,\psi}^\mu$ of \er{eq:payoff}.

\begin{theorem}
\label{thm:second-difference}
Given $t\in\R_{>0}$, $x\in\cX_\half$, and concave terminal payoff $\psi:\cX_\half\mapsinto\R$,
the payoff functional $J_{m,\psi}^\mu(t,x,\cdot)$ of \er{eq:payoff} is strictly concave. In particular, the second difference 
$\Delta J_{m,\psi}^\mu(t,x,w^*\!\!,\delta\, \tilde w) \doteq J_{m,\psi}^\mu(t,x,w^*\!\!+\delta\tilde w) - 2\, J_{m,\psi}^\mu(t,x,w^*) + J_{m,\psi}^\mu(t,x,w^*\!\!-\delta\tilde w)$ of the payoff functional $J_{m,\psi}^\mu(t,x,\cdot)$ at any $w^*\in\cW[0,t]$ is strictly negative as per \er{eq:second-difference-bound} for any direction $\delta\, \tilde w\in\cW[0,t]$ in the input space defined by $\delta\in\R$, $\delta\, \|\tilde w\|_{\cW[0,t]}\ne 0$.
\end{theorem}
\begin{proof}
Fix $t\in\R_{>0}$, $x\in\cX_\half$, $w^*\in\cW[0,t]$, $\delta\, \tilde w\in\cW[0,t]$ and $\delta\in\R$, with $\delta\, \|\tilde w\|_{\cW[0,t]}\ne 0$, as per the theorem statement. Define the trajectories corresponding to inputs $w^*$ and $\hat w \doteq w^*+\delta\, \tilde w$ via \er{eq:dynamics} as
\begin{align}
	\xi^*(r)
	& \doteq x + \int_0^r w^*(s)\, ds\,,
	\quad
	\hat\xi(r)
	\doteq x + \int_0^r w^*(s) + \delta\, \tilde w(s)\, ds = \xi^*(r) + \delta\, \tilde\xi(r)\,, \quad
	\tilde\xi(r) \doteq \int_0^r \tilde w(s)\, ds\,,
	\label{eq:star-traj}
\end{align}
where $r\in[0,t]$. The integrated action functional in the payoff \er{eq:payoff} is of the form $\int_0^t V(\xi(s)) - T^\mu(w(s)) \, ds$, where $V$ and $T^\mu$ are quadratic functionals given by
\begin{align}
	V(\xi(s))
	& \doteq \ts{\frac{\kappa}{2}}\, \|\xi(s)\|_\half^2\,,
	\quad
	T^\mu(w(s)) \doteq \ts{\frac{m}{2}}\, \|\op{J}^\mu\, w(s)\|_\half^2\,,
	\label{eq:action-quads}
\end{align}
with operator $\op{J}^\mu$ as per \er{eq:op-J-mu}. Applying \er{eq:star-traj} in \er{eq:action-quads}, 
\begin{align}
	V(\xi^*(r) + \delta\, \tilde\xi(r))
	& = V(\xi^*(r)) + \delta\, \kappa\, \left\langle \xi^*(r),\, \int_0^r \tilde w(s)\, ds \right\rangle_\half + 
	\delta^2 \, (\ts{\frac{\kappa}{2}}) \left\| \int_0^r \tilde w(s)\, ds \right\|_\half^2\,,
	\label{eq:series-V}
	\\
	T^\mu(w^*(r) + \delta\, \tilde w(r))
	& = T(w^*(r)) + \delta\, m \left[ \langle \op{J}\, w^*(r),\, \op{J}\, \tilde w(r) \rangle_\half 
			+ \mu^2 \, \langle w^*(r),\, \tilde w(r) \rangle_\half \right]
	\nn\\
	& \hspace{18.5mm}
			+ \delta^2 \, (\ts{\frac{m}{2}}) \left[  \|\op{J}\, \tilde w(r)\|_\half^2 + \mu^2\, \|\tilde w(r)\|_\half^2 \right]\,.
	\label{eq:series-T} 
\end{align}
Hence, combining \er{eq:payoff}, \er{eq:series-V}, and \er{eq:series-T}, % and \er{eq:series-payoff},
\begin{align}
	& J_{m,\psi}^\mu (t,x,w^* + \delta\, \tilde w) - J_{m,\psi}^\mu(t,x,w^*)
	\nn\\
	& = \int_0^t V(\xi^*(r) + \delta\, \tilde\xi(r)) - V(\xi^*(r)) - \left[ T^\mu(w^*(r) + \delta\, \tilde w(r)) - T^\mu(w^*(r)) \right] \, dr
	+ \psi(\xi^*(t) + \delta\, \tilde\xi(t)) - \psi(\xi^*(t))
	\nn\\
	& = \int_0^t \delta\, \kappa\, \left\langle \xi(r),\, \int_0^r \tilde w(s)\, ds \right\rangle_\half 
	+ \delta^2 \, (\ts{\frac{\kappa}{2}}) \left\| \int_0^r \tilde w(s)\, ds \right\|_\half^2
	\nn\\
	& \qquad\qquad
	- \delta\, m \left[ \langle \op{J}\, w^*(r),\, \op{J}\, \tilde w(r) \rangle_\half 
			+ \mu^2 \, \langle w^*(r),\, \tilde w(r) \rangle_\half \right]
	- \delta^2 \, (\ts{\frac{m}{2}}) \left[  \|\op{J}\, \tilde w(r)\|_\half^2 + \mu^2\, \|\tilde w(r)\|_\half^2 \right]
	\, dr
	\nn\\
	& \qquad
	+ \psi(\xi^*(t) + \delta\, \tilde\xi(t)) - \psi(\xi^*(t))
	\label{eq:payoff-first-diff}
\end{align}
A corresponding expression for $J_{m,\psi}^\mu (t,x,w^* - \delta\, \tilde w) - J_{m,\psi}^\mu(t,x,w^*)$ follows by replacing $\delta$ with $-\delta$ in \er{eq:payoff-first-diff}. Adding this expression to \er{eq:payoff-first-diff} yields the second difference $\Delta J_{m,\psi}^\mu(t,x,w^*,\delta\, \tilde w)$ of $J_{m,\psi}^\mu(t,x,w^*)$ at $w^*$ in direction $\delta\, \tilde w$, with
\begin{align}
	\Delta J_{m,\psi}^\mu(t,x,w^*\!\!,\delta\, \tilde w) 
	& = \int_0^t \delta^2 \, \kappa \left\| \int_0^r \tilde w(s)\, ds \right\|_\half^2
	- \delta^2 \, m \left[  \|\op{J}\, \tilde w(r)\|_\half^2 + \mu^2\, \|\tilde w(r)\|_\half^2 \right] \, dr
	+ \Delta \psi(\xi^*(t), \delta\, \tilde\xi(t))\,,
	\label{eq:diff-bound-1}
\end{align}
where $\Delta \psi(\xi^*(t), \delta\, \tilde\xi(t)) \doteq \psi(\xi^*(t) + \delta\, \tilde\xi(t)) - 2\, \psi(\xi^*(t)) + \psi(\xi^*(t) - \delta\, \tilde\xi(t))$ is the second difference of $\psi$ at $\xi^*(t)$ in direction $\delta\, \tilde\xi(t)$.
With a view to dealing with first term on the right-hand side of \er{eq:diff-bound-1}, define $q_{\tilde w}\doteq \tilde\xi(r) = \int_0^r \tilde w(\sigma)\, d\sigma\in\cX_\half$ and $\Pi:\cX_\half\mapsinto\R$ by $\Pi\, \omega \doteq \langle \omega,\, q_{\tilde w} \rangle_\half$. Note that $\Pi$ is a closed linear operator (a functional) in $\cL(\cX_\half;\R)$.  Also note that for $\tilde w\in\cW[0,r]$, $r\in[0,t]$, H\"{o}lder's inequality and Cauchy-Schwartz implies that
\begin{align}
	\int_0^r \|\tilde w(s)\|_\half\, ds
	& \le \sqrt{r} \, \|\tilde w\|_{\cW[0,r]}\,,
	\quad
	\int_0^r |\Pi \, \tilde w(s)|\, ds
	= \int_0^r | \langle q_{\tilde w}, \, \tilde w(s)\rangle_\half |\, ds
	% \le \int_0^t \|q_{\tilde w}\| \, \|\tilde w(s)\| \, ds
	\le \sqrt{r} \, \|q_{\tilde w}\|_\half \, \|\tilde w\|_{\cW[0,r]} < \infty\,.
	\label{eq:diff-bound-1a}
\end{align}
That is, $\tilde w\in\Lone([0,r];\cX_\half)$ and $\Pi\, \tilde w\in\Lone([0,r];\R)$. Hence, as $\cX_\half$ and $\R$ are separable Hilbert spaces (separability of the former follows by existence of a countable basis, see Lemma \ref{lem:basis-Riesz}), it follows by \cite[Theorem A.5.23, p.628]{CZ:95} (for example) that 
\begin{align}
	& \Pi \int_0^r \tilde w(s)\, ds = \int_0^r \Pi\, \tilde w(s)\, ds\,.
	\label{eq:Pi-0}
\end{align}
Recalling the definition of $\Pi$,
\begin{align}
	\Pi\, \int_0^r \tilde w(s) \, ds
	& =
	\left\langle \int_0^r \tilde w(s)\, ds,\, q_{\tilde w} \right\rangle_\half
	=
	\left\langle \int_0^r \tilde w(s)\, ds,\, \int_0^r \tilde w(\sigma) \, d\sigma \right\rangle_\half
	= \left\| \int_0^r \tilde w(s)\, ds \right\|_\half^2\,,
	\label{eq:Pi-1}
\end{align}
while
\begin{align}
	\int_0^r \Pi\, \tilde w(s)\, ds
	& = \int_0^r \langle \tilde w(s),\, q_{\tilde w} \rangle_\half \, ds
	=
	\int_0^r \left\langle
		\tilde w(s),\, \int_0^r \tilde w(\sigma)\, d\sigma
	\right\rangle_\half \, ds
	= \int_0^r \int_0^r \langle \tilde w(s),\, \tilde w(\sigma) \rangle_\half \, d\sigma\, ds
	\nn\\
	% & \le  \int_0^r \int_0^r |\langle \tilde w(s),\, \tilde w(\sigma) \rangle_\half| \, d\sigma\, ds
	& \le \int_0^r \int_0^r \| \tilde w(s) \|_\half \, \| \tilde w(\sigma) \|_\half \, d\sigma \, ds
	= \left( \int_0^r \| \tilde w(s) \|_\half \, ds \right)^2
	% \nn\\
	% & 
	\le r \, \|\tilde w\|_{\cW[0,r]}^2 \le r \, \|\tilde w\|_{\cW[0,t]}^2\,, 
	\label{eq:Pi-2}
\end{align}
in which \cite[Theorem A.5.23, p.628]{CZ:95} is applied a second time to obtain the third equality, and the left-hand inequality in \er{eq:diff-bound-1a} is applied to obtained the upper bound. Hence, combining \er{eq:Pi-1} and \er{eq:Pi-2} in \er{eq:Pi-0} yields
\begin{align}
	\left\| \int_0^r \tilde w(s)\, ds \right\|_\half^2
	& \le  r \, \|\tilde w\|_{\cW[0,t]}^2\,,
	% \label{eq:diff-bound-2}
	\nn
\end{align}
which in turn implies that the first term on the right-hand side of \er{eq:diff-bound-1} is
\begin{align}
	\int_0^t \delta^2\, \kappa \left\| \int_0^r \tilde w(s)\, ds \right\|_\half^2 \, dr
	& \le \delta^2\, \kappa \int_0^t r  \, dr \, \|\tilde w\|_{\cW[0,t]}^2 = \delta^2\, \kappa\, (\ts{\frac{t^2}{2}}) \, \|\tilde w\|_{\cW[0,t]}^2\,.
	\label{eq:diff-bound-3}
\end{align}
Substituting \er{eq:diff-bound-3} in \er{eq:diff-bound-1} thus yields the second difference bound
\begin{align}
	\Delta J_{m,\psi}^\mu(t,x,w^*\!\!,\delta\, \tilde w) 
	& \le -\delta^2 \left[ m\, \mu^2 - \kappa\, (\ts{\frac{t^2}{2}}) \right] \| \tilde w\|_{\cW[0,t]}^2 - \delta^2\, m \, \|\op{J}\, \tilde w\|_{\cW[0,t]}^2
	% + \delta^2\,  \langle \ggrad_{xx} \psi(\bar\xi(t))\, \tilde\xi(t),\, \tilde\xi(t) \rangle_\half\,.
	+ \Delta \psi(\xi^*(t), \delta\, \tilde\xi(t))\,.
	\label{eq:diff-bound-4}
\end{align}
% As the second {\Frechet} derivative of the terminal payoff $\psi$ is negative semi-definite, 
As the second difference $\Delta \psi(\xi^*(t), \delta\, \tilde\xi(t))$ is non-positive by concavity of $\psi$,
\er{eq:diff-bound-4} is strictly negative if $m\, \mu^2 - \kappa \, (\ts{\frac{t^2}{2}}) > 0$. That is, if $t\in[0,\bar t^\mu)$, where $\bar t^\mu\in\R_{>0}$ is as per \er{eq:t-bar-mu}. Under these conditions, it follows immediately that the payoff functional $J_{m,\psi}^\mu(t,x,\cdot)$ of \er{eq:payoff} is strictly concave.
\end{proof}

%%%%%%%%%%%%%%%%%%%%%%%%%%%%%%%%%%%%%%%%%%%%%%%%%%%%%%%%%%%%%%%%%%%%
%%
%%		Fundamental solution
%%

\section{Fundamental solution to the approximating optimal control problem}
\label{sec:fundamental}

A fundamental solution in this optimal control context is an object from which the value function $W^\mu$ of \er{eq:W} can be computed given any concave terminal payoff $\psi$. This fundamental solution is constructed via four auxiliary control problems.

%%		Auxiliary control problems

\subsection{Auxiliary control problems}
\label{sec:aux-control}
The auxiliary control problems of interest employ the same running payoff as used in \er{eq:payoff} to define the approximating optimal control problem \er{eq:W}. A specific terminal payoff is used in each auxiliary problem. Two of these terminal payoffs depend on an additional function $z\in\cX_\half$ describing the terminal displacement. These terminal payoffs are denoted by $\psi^{\mu,c}:\cX_\half\times\cX_\half\mapsinto\R$, $\psi^{\mu,\infty}:\cX_\half\times\cX_\half\mapsinto\R\cup\{-\infty\}$, and $\psi^0:\cX_\half\mapsinto\R$, where $\mu,\, c\in\R_{\ge 0}$ denote real-valued parameters. Specifically,
\begin{gather}
	\psi^0(x)
	\doteq 0\,,
	% \label{eq:psi-0}
	% \\
	\qquad
	\psi^{\mu,c}(x,z)
	\doteq -\ts{\frac{c}{2}} \, \|\op{K}_\mu\, (x-z) \|_\half^2\,,
	% \label{eq:psi-c}
	% \\
	\qquad
	\psi^{\mu,\infty}(x,z)
	\doteq \left\{ \ba{rl}
		0\,,	& \|\op{K}_\mu\, (x - z)\|_\half = 0\,,
		\\
		-\infty\,, & \|\op{K}_\mu\, (x - z)\|_\half > 0\,,
	\ea \right.
	\label{eq:psi}
\end{gather}
where $\op{K}_\mu\in\bo(\cX_\half)$ is a boundedly invertible operator to be defined later. Using these terminal payoffs and a fixed real-valued parameter $\ul{m}\in(0,m)$, the four auxiliary control problems of interest are defined via their respective (auxiliary) value functions 
$$
	\ol{W}^\mu:\R_{\ge 0}\times\cX_\half\mapsinto\R\,, \quad 
	\ol{W}^{\mu,c},\, W^{\mu,c}:\R_{\ge 0}\times\cX_\half\times\cX_\half\mapsinto\R\,, \quad
	W^{\mu,\infty}:\R_{\ge 0}\times\cX_\half\times\cX_\half\mapsinto\R\cup\{-\infty\},
$$
where
\begin{align}
	\ol{W}^\mu(t,x)
	& \doteq \sup_{w\in\cW[0,t]} J_{\ul{m},\psi^0}^\mu(t,x,w)\,,
	% \label{eq:Wbar-0}
	& 
	\ol{W}^{\mu,c}(t,x,z)
	& \doteq \sup_{w\in\cW[0,t]} J_{\ul{m},\psi^{\mu,c}(\cdot,z)}^\mu(t,x,w)\,,
	\label{eq:Wbar}
	\\
	W^{\mu,c}(t,x,z)
	& \doteq \sup_{w\in\cW[0,t]} J_{m,\psi^{\mu,c}(\cdot, z)}^\mu(t,x,w)\,,
	\label{eq:W-c}
	\\
	W^{\mu,\infty}(t,x,z)
	& \doteq \sup_{w\in\cW[0,t]} J_{m,\psi^{\mu,\infty}(\cdot, z)}^\mu(t,x,w)\,.
	\label{eq:W-infty}
\end{align}
The majority of the subsequent analysis will concern the value function $W^{\mu,c}$ of \er{eq:W-c} and its convergence to $W^{\mu,\infty}$ of \er{eq:W-infty} as $c\rightarrow\infty$. Ultimately, $W^{\mu,\infty}$ plays the role of the fundamental solution for the optimal control problem \er{eq:W}, in the sense that
\begin{align}
	W^\mu(t,x)
	& = \sup_{z\in\cX_\half} \left\{ W^{\mu,\infty}(t,x,z) + \psi(z) \right\}
	\label{eq:W-from-fund-conjecture}
\end{align}
for all $t\in\R_{\ge 0}$, $x\in\cX_\half$. The remaining value functions $\ol{W}^\mu$ and $\ol{W}^{\mu,c}$ are useful in ensuring that these auxiliary problems are well-defined. In particular, note by inspection of the value functions \er{eq:Wbar}--\er{eq:W-infty} that
\begin{align}
	& \ol{W}^\mu(t,x) \ge \ol{W}^{\mu,c}(t,x,z) \ge W^{\mu,c}(t,x,z) \ge W^{\mu.\infty}(t,x,z) > -\infty
	\label{eq:value-order}
\end{align}
for all $t\in\R_{>0}$, $x,z\in\cX_\half$. Here, the last inequality follows by noting that the specific constant input $\hat w\in\cW[0,t]$ defined by $\hat w(s) \doteq \ts{\frac{1}{t}} (z-x)$ for all $s\in[0,t]$ is suboptimal in the definition \er{eq:W-infty} of $W^{\mu,\infty}(t,x,z)$. In particular, $W^{\mu,\infty}(t,x,z) \ge J_{m,\psi^{\mu,\infty}(\cdot,z)}^\mu (t,x,\hat w) \ge -\ts{\frac{m}{2\, t}} \, \|z-x\|_\half^2>-\infty$. The following is assumed throughout.
\begin{assumption}
\label{ass:finite}
$\ol{W}^\mu(t,x) < \infty$ for all $\mu\in(0,1], \, t\in(0,\bar t^\mu),\, x\in\cX_\half$.
\end{assumption}

%%		A limit property

%%%%%%%%%%%%%%%%%%%%%%%%%%%%%%%%%%%%%%%%%%%%%%%%%%%%%%%
%%
%%		TODO (31AUG14)
%%
%%		Sort out the $t=0$ case in Assumption 1 and the Lemmas / Theorems that follow it. 
%%		This can be done by treating the t=0 case separately.
%%
%%%%%%%%%%%%%%%%%%%%%%%%%%%%%%%%%%%%%%%%%%%%%%%%%%%%%%%

%%		Fundamental solution.

\subsection{Fundamental solution}
In order to construct a fundamental solution to the optimal control problem \er{eq:W}, it is useful to define the value functional $\What^\mu:\R_{\ge 0}\times\cX_\half\mapsinto\R$ by
\begin{align}
	\What^\mu(t,x)
	& \doteq
	\sup_{z\in\cX_\half} \left\{ W^{\mu,\infty}(t,x,z) + \psi(z) \right\}.
	\label{eq:What}
\end{align}
\begin{theorem}
\label{thm:reach}
The value functionals $W^\mu$ and $\What^\mu$ of \er{eq:W} and \er{eq:What} are equivalent, with $W^\mu(t,x) = \What^\mu(t,x)$ for all $t\in[0,\bar t^\mu)$ and $x\in\cX_\half$.
\end{theorem}
\begin{proof} % [Theorem \ref{thm:reach}]
Fix $t\in[0,\bar t^\mu)$, $x\in\cX_\half$. Substituting \er{eq:W-infty} in \er{eq:What}, and recalling \er{eq:psi},
\begin{align}
	\What^\mu(t,x)
	% & = \sup_{z\in\cX_\half} \biggl\{ \sup_{w\in\cW[0,t]} J_{m,\psi^{\mu,\infty}(\cdot, z)}^\mu(t,x,w) + \psi(z) \biggr\}
	% \nn\\
	& = \sup_{z\in\cX_\half} \sup_{w\in\cW[0,t]} \left\{ J_{m,\psi^0}^\mu(t,x,w) + \psi^{\mu,\infty}(\xi(t),z) + \psi(z) \left| \ba{c}
		\text{\er{eq:dynamics} holds with} \\
		\xi(0) = x
	\ea \right. \right\}
	\nn\\
	& = \sup_{w\in\cW[0,t]} \sup_{z\in\cX_\half} \left\{ J_{m,\psi^0}^\mu(t,x,w) + \psi^{\mu,\infty}(\xi(t),z) + \psi(z) \left| \ba{c}
		\text{\er{eq:dynamics} holds with} \\
		\xi(0) = x
	\ea \right. \right\}
	\nn\\
	& = \sup_{w\in\cW[0,t]} \left\{ J_{m,\psi^0}^\mu(t,x,w) + \sup_{z\in\cX_\half} \left\{ \psi^{\mu,\infty}(\xi(t),z) + \psi(z) \right\} \left| \ba{c}
		\text{\er{eq:dynamics} holds with} \\
		\xi(0) = x
	\ea \right. \right\}.
	\label{eq:new-fund-proof-1}
\end{align}
By inspection of \er{eq:psi}, the inner supremum must be achieved at $z = z^* \doteq \xi(t)$, with 
$
	\sup_{z\in\cX_\half} \{ \psi^{\mu,\infty}(\xi(t),z) + \psi(z) \} 
	= \psi^{\mu,\infty}(\xi(t),z^*) + \psi(z^*) = 0 + \psi(\xi(t)) = \psi(\xi(t))
$.
Substituting in \er{eq:new-fund-proof-1} and recalling \er{eq:W} yields that
\begin{align}
	\What^\mu(t,x)
	& = \sup_{w\in\cW[0,t]} \left\{ J_{m,\psi^0}^\mu(t,x,w) + \psi(\xi(t)) \left| \ba{c}
		\text{\er{eq:dynamics} holds with} \\
		\xi(0) = x
	\ea \right. \right\}
	= \sup_{w\in\cW[0,t]} J_{m,\psi}^\mu(t,x,w)
	= W^\mu(t,x).
	\nn
\end{align}
\end{proof}

Theorem \ref{thm:reach} provides an explicit decomposition of the approximating optimal control problem associated with the principle of stationary action. In particular, it provides a means of evaluating of the value functional $W^\mu$ of \er{eq:W} for {\em any} concave terminal payoff $\psi$, including that of \er{eq:stationary-payoff}. In this regard, inspection of \er{eq:What} via Theorem \ref{thm:reach} reveals that $W^{\mu,\infty}$ of \er{eq:W-infty} can be regarded as an approximation (for $\mu\ne 0$) of the fundamental solution to the TPBVP \er{eq:TPBVP} via the principle of stationary action. Consequently, characterization of an explicit representation of $W^{\mu,\infty}$ is important for its application in the computation of $W^\mu$.

%%
%%		Explicit representation of the fundamental solution.
%%

\subsection{Explicit representation of the fundamental solution}
In order to characterize the fundamental solution of the approximating optimal control problem \er{eq:W} via Theorem \ref{thm:reach}, an explicit form for the value function $W^{\mu,\infty}$ of \er{eq:W-infty} may be constructed via three steps:
\begin{center}
\parbox[t]{17cm}{
\begin{itemize}
\item[\done]
Show that the value function $W^{\mu,\infty}$ of \er{eq:W-infty} may be obtained as the limit of $W^{\mu,c}$ of \er{eq:W-c} as $c\rightarrow\infty$;
\item[\dtwo]
Develop a verification theorem that provides a means for proposing and validated an explicit representation for the value function $W^{\mu,c}$ of \er{eq:W-c};
\item[\dthree]
Find an explicit representation satisfying the conditions of the verification theorem of {\dtwo}, and apply the limit argument of {\done} to obtain the corresponding representation for $W^{\mu,\infty}$ of \er{eq:W-infty}. 
\end{itemize}
}
\end{center}

%%
%%		Step 1.
%%

\subsubsection{Limit argument -- {\done}} 

This first step is formalized via the following theorem.

%%	Limit theorem.

\begin{theorem}
\label{thm:limit}
The auxiliary value functions $W^{\mu,c}$, $W^{\mu,\infty}$ of \er{eq:W-c}, \er{eq:W-infty} satisfy the limit relationship
\begin{align}
	\lim_{c\rightarrow\infty} W^{\mu,c}(t,x,z)
	& = W^{\mu,\infty}(t,x,z)
	\label{eq:aux-limit}
\end{align}
for all $t\in[0,\bar t^\mu)$, $x,z\in\cX_\half$.
\end{theorem}

In order to demonstrate the limit property summarized by Theorem \ref{thm:limit}, it is useful to first note that a ball of any fixed radius centered on $z\in\cX_\half$ can be reached by a sufficiently near-optimal trajectory defined with respect to \er{eq:W-c}.
%%
%% near optimal trajectories in \er{eq:W-c} always reach a ball of radius centered on $z\in\cX_\half$. 
%% 
\begin{lemma}
\label{lem:eps-ball}
Fix $t\in[0,\bar t^\mu)$ and $x,z\in\cX_\half$. For each $\eps\in\R_{>0}$, there exists a $\bar c \doteq \bar c_{t,x,z}^\eps\in\R_{>0}$, $\bar\delta\in(0,1]$, such that $\left\| \xi^{c,\delta}(t) - z \right\|_\half \le \eps$ for all $c\in(\bar c,\infty)$ and $\delta\in(0,\bar\delta)$, where $\xi^{c,\delta}(\cdot)$ denotes the trajectory of system \er{eq:dynamics} corresponding to any $\delta$-optimal input $w^{c,\delta}\in\cW[0,t]$ in the definition \er{eq:W-c} of $W^{\mu,c}(t,x,z)$.
\end{lemma}
%%
% \if{false}

\begin{proof} % [Lemma \ref{lem:eps-ball}]
Fix $t\in[0,\bar t^\mu)$ and $x,z\in\cX_\half$. Recalling the assumed bounded invertibility of $\op{K}_\mu$ on $\cX_\half$, see \er{eq:psi}, set $\kappa_\mu\doteq \left\|\op{K}_\mu^{-1}\right\|_\half^2\in\R_{>0}$. Suppose the statement of the lemma is false. That is, there exists an $\eps\in\R_{>0}$ such that for all $\bar c\in\R_{>0}$ and $\bar\delta\in(0,1]$, there exists a $c\in(\bar c,\infty)$ and $\delta\in(0,\bar\delta)$ such that $\left\| \xi^{c,\delta}(t) - z \right\|_\half>\eps$. So, given this $\eps\in\R_{>0}$, choose a specific $\bar c\in\R_{>0}$ and $\bar\delta\in(0,1]$ such that
\begin{align}
	(\ts{\frac{\bar c}{2\, \kappa_\mu}}) \, \eps^2 - \bar\delta
	& \ge 
	\ol{W}^\mu(t,x) - W^{\mu,\infty}(t,x,z)\,.
	\label{eq:eps-ball-1}
\end{align}
(Note that this is always possible by Assumption \ref{ass:finite}.) Let $c\in(\bar  c,\infty)$ and $\delta\in(0,\bar\delta)$ be such that $\left\|\xi^{c,\delta}(t) - z\right\|_\half > \eps$ as per the hypothesis above. Note by bounded invertibility of $\op{K}_\mu$ on $\cX_\half$,
\begin{align}
	\eps^2 
	& < \left\| \op{K}_\mu^{-1}\, \op{K}_\mu\, (\xi^{c,\delta}(t) - z) \right\|_\half^2
	\le 
	% \left\| \op{K}_\mu^{-1} \right\|_\half^2 
	\kappa_\mu\, \left\| \op{K}_\mu\, (\xi^{c,\delta}(t) - z) \right\|_\half^2\,.
	\label{eq:eps-ball-1b}
\end{align}
Hence, by definition of any $\delta$-optimal input $w^{c,\delta}$ in $W^{\mu,c}(t,x,z)$ of \er{eq:W-c},
\begin{align}
	W^{\mu,c}(t,x,z) - \delta
	& < J_{m,\psi^c(\cdot,z)}^\mu\left(t,x,w^{c,\delta}\right)
	= J_{m,\psi^0}^\mu(t,x,w^{c,\delta}) + \psi^{\mu,c}\left(\xi^{c,\delta}(t),\, z \right)
	\nn\\
	& \le \ol{W}^\mu(t,x) - \textstyle{\frac{c}{2}}  \left\|\op{K}_\mu\, (\xi^{c,\delta}(t)-z) \right\|_\half^2
	\le \ol{W}^\mu(t,x) - (\ts{\frac{c}{2\, \kappa_\mu}}) \, \eps^2\,,
	\nn
\end{align}
where the equality follows by \er{eq:psi}, while the inequalities follow by suboptimality of $w^{c,\delta}$ in the definition \er{eq:W-c} of $W^{\mu,c}$, \er{eq:value-order}, and \er{eq:eps-ball-1b}. Consequently,
$
	\ol{W}^\mu(t,x) - W^{\mu,\infty}(t,x,z)
	\ge 
	\ol{W}^\mu(t,x) - W^{\mu,c}(t,x,z)
	\ge 
	(\ts{\frac{c}{2\, \kappa_\mu}}) \, \eps^2 - \delta
	> (\ts{\frac{\bar c}{2\, \kappa_\mu}}) \, \eps^2 - \bar\delta
$, 
which contradicts \er{eq:eps-ball-1}. Hence, the assertion in the lemma statement is true.
\end{proof}

% \fi
%%
An upper norm bound on near-optimal inputs is also useful.
\begin{lemma}
\label{lem:near-opt-input-bound}
With $c\in\R_{>0}$, $\delta\in(0,1]$, $t\in[0,\bar t^\mu)$, and $x,z\in\cX_\half$ fixed, any input $w^{c,\delta}\in\cW[0,t]$ that is $\delta$-optimal in the definition \er{eq:W-c} of $W^{\mu,c}(t,x,z)$ satisfies the bound
\begin{align}
	\| \op{J}^\mu\, w^{c,\delta} \|_{\cW[0,t]}
	& \le M^{\mu,c,\delta}(t,x,z) 
	\le \overline{M}^\mu(t,x,z)
	\label{eq:near-opt-input-bound}
\end{align}
where
\begin{align}
	M^{\mu,c,\delta}(t,x,z)
	& \doteq \left( \frac{\ol{W}^{\mu,c}(t,x,z) - W^{\mu,c}(t,x,z) + \delta}{\demi (m - \ul{m})} \right)^{\frac{1}{2}}\,,
	% \label{eq:near-opt-M}
	\quad
	% \\
	\overline{M}^\mu(t,x,z)
	% & 
	\doteq \left( \frac{\ol{W}^\mu(t,x) - W^{\mu,\infty}(t,x,z)+ 1}{\demi (m - \ul{m})} \right)^{\frac{1}{2}}\,.
	% \label{eq:near-opt-M-bar}
	\nn
\end{align}
\end{lemma}
%%
% \if{false}

\begin{proof} % [Lemma \ref{lem:near-opt-input-bound}]
Input $w^{c,\delta}\in\cW[0,t]$ is respectively sub-optimal and $\delta$-optimal in the definitions \er{eq:Wbar} and \er{eq:W-c} of $\ol{W}^{\mu,c}(t,x,z)$ and $W^{\mu,c}(t,x,z)$, so that
\begin{align}
	\ol{W}^{\mu,c}(t,x,z)
	& \ge J_{\ul{m},\psi^c(\cdot,z)}^\mu(t,x,w^{c,\delta})
	= \int_0^t \ts{\frac{\kappa}{2}}\, \|\xi^{c,\delta}(s) \|_\half^2 - \ts{\frac{\ul{m}}{2}}\, \|\op{J}^\mu\, w^{c,\delta}(s)\|_\half^2 \, ds 
			+ \psi^{\mu,c}\left(\xi^{c,\delta}(t),\, z\right)\,,
	\nn\\
	W^{\mu,c}(t,x,z) - \delta
	& < J_{m,\psi^c(\cdot,z)}^\mu(t,x,w^{c,\delta})
	= \int_0^t \ts{\frac{\kappa}{2}}\, \| \xi^{c,\delta}(s) \|_\half^2 - \ts{\frac{m}{2}}\, \|\op{J}^\mu \, w^{c,\delta}(s) \|_\half^2 \, ds
			+ \psi^{\mu,c}\left(\xi^{c,\delta}(t),\, z\right)\,,
	\nn
\end{align}
in which $\xi^{c,\delta}(\cdot)$ denotes the trajectory corresponding to $w^{c,\delta}(\cdot)$. The left-hand inequality of \er{eq:near-opt-input-bound} follows by subtracting the second inequality above from the first, while the right-hand inequality of \er{eq:near-opt-input-bound} follows by application of \er{eq:value-order} in the definition of $M^{\mu,c,\delta}(t,x,z)$ to yield the upper bound $\ol{M}^\mu(t,x,z)$.
\end{proof}

Lemmas \ref{lem:eps-ball} and \ref{lem:near-opt-input-bound} facilitate the proof of the required limit property of Theorem \ref{thm:limit}. 

\begin{proof}[Theorem \ref{thm:limit}]
Fix $t\in[0,\bar t^\mu)$ and $x,z\in\cX_\half$. Observe that for all $c\in\R_{\ge 0}$, $c_1\in\R_{\ge c}$,
\begin{align}
	\psi^{\mu,c}(x,z) & \ge \psi^{\mu,c_1}(x,z) \ge \psi^{\mu,\infty}(x,z)\,,
	% \nn\\
	\qquad
	W^{\mu,c}(t,x,z) \ge W^{\mu,c_1}(t,x,z) \ge W^{\mu,\infty}(t,x,z)\,,
	\nn
\end{align}
where the first set of inequalities follows immediately by inspection of \er{eq:psi}, which in turn implies the second set of inequalities by inspection of \er{eq:W-c} and \er{eq:W-infty}. That is, $W^{\mu,c}$ is non-increasing in $c$ and satisfies
\begin{align}
	\lim_{c\rightarrow\infty} W^{\mu,c}(t,x,z)
	& \ge W^{\mu,\infty}(t,x,z)\,.
	\label{eq:aux-limit-1}
\end{align}
In order to prove the opposite inequality required to demonstrate \er{eq:aux-limit}, a sub-optimal input for $W^{\mu,\infty}$ is constructed from a near-optimal input for $W^{\mu,c}$. To this end, fix an arbitrary $\eps\in\R_{>0}$. With $\xi^{c,\delta}(\cdot)$ and $w^{c,\delta}(\cdot)$ as per the statement of Lemma \ref{lem:eps-ball}, there exists a $\bar c\in\R_{>0}$ and $\bar\delta\in(0,1]$ such that 
\begin{align}
	& \left\| \xi^{c,\delta}(t) - z \right\|_\half \le \eps
	\label{eq:aux-limit-1b}
\end{align}
for all $c\in(\bar c,\infty)$ and $\delta\in(0,\bar\delta)$. Define a new input $\hat w^{c,\delta}\in\cW[0,t]$ by
\begin{align}
	\hat w^{c,\delta}(s)
	& \doteq w^{c,\delta}(s) + \textstyle{\frac{1}{t}} \left( z - \xi^{c,\delta}(t) \right)
	\label{eq:aux-limit-2}
\end{align}
for all $s\in[0,t]$. By inspection of \er{eq:dynamics} and \er{eq:aux-limit-2}, the corresponding state trajectory $\hat \xi^{c,\delta}(\cdot)$ satisfies
\begin{align}
	\hat\xi^{c,\delta}(s) 
	& = x + \int_0^s \hat w^{c,\delta}(\sigma)\, d\sigma 
	= x + \int_0^s w^{c,\delta}(\sigma)\, d\sigma + \ts{\frac{s}{t}}\, \left( z - \xi^{c,\delta}(t) \right)
	= \xi^{c,\delta}(s) + \ts{\frac{s}{t}}\, \left( z - \xi^{c,\delta}(t) \right)\,,
	\label{eq:aux-limit-2b}
\end{align}
so that $\hat \xi^{c,\delta}(0) = x$ and $\hat \xi^{c,\delta}(t) = z$. However, as $\xi^{c,\delta}(t)$ need not equal $z$, \er{eq:psi} implies that 
$
	\psi^{\mu,c}\left(\xi^{c,\delta}(t),\, z\right) 
	= -\ts{\frac{c}{2}} \, \left\| \op{K}_\mu\, ( \xi^{c,\delta}(t) - z ) \right\|_\half^2
	\le 0 
	= \psi^{\mu,\infty} ( \hat \xi^{c,\delta}(t),\, z)
$.
So, for all $c\in(\bar c,\infty)$, $\delta\in(0,\bar\delta)$,
\begin{align}
	& W^{\mu,c}(t,x,z) - \delta 
	< J_{m,\psi^c(\cdot,z)}^\mu(t,x,w^{c,\delta})
	= \int_0^t \ts{\frac{\kappa}{2}}\, \|\xi^{c,\delta}(s) \|_\half^2 - \ts{\frac{m}{2}}\, \|\op{J}^\mu\, w^{c,\delta}(s) \|_\half^2 \, ds
			+ \psi^{\mu,c}(\xi^{c,\delta}(t),\, z)
	\nn\\
	& \le \int_0^t \!\! \ts{\frac{\kappa}{2}}\, \| \hat\xi^{c,\delta}(s) \|_\half^2 \! - \ts{\frac{m}{2}}\, \|\op{J}^\mu\, \hat w^{c,\delta}(s) \|_\half^2 \, ds
			+ \psi^{\mu,\infty}(\hat\xi^{c,\delta}(t),\, z) + \Delta^{\mu,\eps}(t,x,z)
	= J_{m,\psi^\infty(\cdot,z)}^\mu \left(t,x,\hat w^{c,\delta}\right) + \Delta^{\mu,\eps}(t,x,z)
	\nn\\
	& \le W^{\mu,\infty}(t,x,z) + \Delta^{\mu,\eps}(t,x,z)\,,
	\label{eq:aux-limit-4}
\end{align}
where sub-optimality of $\hat w^{c,\delta}$ in the definition \er{eq:W-infty} of $W^{\mu,\infty}(t,x,z)$ has been applied, and
\begin{align}
	\Delta^{\mu,\eps}(t,x,z)
	& \doteq
	\int_0^t \ts{\frac{\kappa}{2}} \left( 
			\| \xi^{c,\delta}(s) \|_\half^2 - \| \hat\xi^{c,\delta}(s) \|_\half^2 
		\right)
	 - \ts{\frac{m}{2}} \left( 
			\|\op{J}^\mu\, w^{c,\delta}(s) \|_\half^2 - \|\op{J}^\mu\, \hat w^{c,\delta}(s) \|_\half^2
		\right) \, ds
	\nn\\
	& \le \int_0^t \ts{\frac{\kappa}{2}} \left( \| \hat\xi^{c,\delta}(s) \|_\half + \| \xi^{c,\delta}(s) \|_\half \right)
		 \| \hat\xi^{c,\delta}(s) - \xi^{c,\delta}(s) \|_\half
	+ \ts{\frac{m}{2}} \left(
			\| \op{J}^\mu\, \hat w^{c,\delta}(s) \|_\half + \| \op{J}^\mu\, w^{c,\delta}(s) \|_\half
		\right) \times
	\nn\\
	& \hspace{25mm}
		\| \op{J}^\mu\, (\hat w^{c,\delta}(s) - w^{c,\delta}(s) ) \|_\half
		\, ds\,.
	\label{eq:aux-limit-5}
\end{align}
(Here, the upper bound follows by the triangle inequality.) Note that $\Delta^{\mu,\eps}(t,x,z)$ is parameterized by $\eps\in\R_{>0}$ via $\bar c$ and $\bar\delta$ (see Lemma \ref{lem:eps-ball}). In order to bound the right-hand side of \er{eq:aux-limit-5}, H\"{o}lder's inequality implies that for any Hilbert space $\cZ$ (with norm denoted by $\|\cdot\|_\cZ$) and any $z, \hat z\in\cZ[0,t] \doteq \Ltwo([0,t];\cZ)$,
\begin{align}
	& \int_0^t \left( \|\hat z(s)\|_\cZ + \|z(s)\|_\cZ \right) \left\| \hat z(s) - z(s) \right\|_\cZ \, ds
	\le \left( \int_0^t \left( \|\hat z(s)\|_\cZ + \|z(s)\|_\cZ \right)^2 \, ds \right)^\half 
	\left( \int_0^t \left\| \hat z(s) - z(s) \right\|_\cZ^2\, ds \right)^\half
	\nn\\
	% & \le \sqrt{2} \left( \int_0^t \|\hat z(s)\|_\cZ^2 + \|z(s)\|_\cZ^2 \, ds \right)^\half \left\| \hat z - z \right\|_{\cZ[0,t]}
	% \nn\\
	% & = 
	& \qquad \le \sqrt{2} \left( \|\hat z\|_{\cZ[0,t]}^2 + \|z\|_{\cZ[0,t]}^2 \right)^\half \left\| \hat z - z \right\|_{\cZ[0,t]}
	\le \sqrt{2} \left( \|\hat z\|_{\cZ[0,t]} + \|z\|_{\cZ[0,t]} \right) \left\| \hat z - z \right\|_{\cZ[0,t]}\,,
	\label{eq:aux-limit-6}
\end{align}
in which $\|z\|_{\cZ[0,t]}^2 \doteq \int_0^t \|z(s)\|_{\cZ}^2 \, ds$. Meanwhile, the triangle inequality states that
\begin{align}
	& \|\hat z\|_{\cZ[0,t]}
	\le \|\hat z - z \|_{\cZ[0,t]} + \|z\|_{\cZ[0,t]}\,.
	% &&
	% \|z\|_{\cZ[0,t]}
	% \le \|\hat z - z \|_{\cZ[0,t]} + \|\hat z\|_{\cZ[0,t]}\,.
	\label{eq:aux-limit-6b}
\end{align}
With a view to applying \er{eq:aux-limit-6} and \er{eq:aux-limit-6b} to the right-hand side of \er{eq:aux-limit-5}, note that by \er{eq:op-J-mu} and \er{eq:aux-limit-2},
\begin{align}
	& \| \op{J}^\mu\, (\hat w^{c,\delta}(s) - w^{c,\delta}(s)) \|_\half^2
	= \| \op{J}\, (\hat w^{c,\delta}(s) - w^{c,\delta}(s)) \|_\half^2 
	+  \mu^2 \, \| \hat w^{c,\delta}(s) - w^{c,\delta}(s) \|_\half^2
	\nn\\
	& = \| \op{J}\, \ts{\frac{1}{t}} (z - \xi^{c,\delta}(t)) \|_\half^2 +  \mu^2 \, \| \ts{\frac{1}{t}} (z - \xi^{c,\delta}(t)) \|_\half^2
	= \ts{\frac{1}{t^2}} \, \| \op{J}\, \opAsqrt (z - \xi^{c,\delta}(t)) \|^2 +  \ts{\frac{\mu^2}{t^2}} \, \| z - \xi^{c,\delta}(t) \|_\half^2
	\nn\\
	& \le \ts{\frac{\|\op{J}\|^2 + \mu^2}{t^2}} \,  \| z - \xi^{c,\delta}(t) \|_\half^2 \le \left( \ts{\frac{\|\op{J}\|^2 + \mu^2}{t^2}} \right) \, \eps^2\,,
	\nn
\end{align}
where commutation of $\op{J}\in\bo(\cX)$ and $\opAsqrt$ follows by Lemma \ref{lem:op-A-properties}.
Hence, integration yields % \er{eq:aux-limit-1b}, \er{eq:aux-limit-2}, and \er{eq:op-J-mu} imply that
\begin{align}
	\left\| \op{J}^\mu\, (\hat w^{c,\delta} - w^{c,\delta}) \right\|_{\cW[0,t]}
	% & = \left( \int_0^t \|\op{J}^\mu\, (\hat w^{c,\delta}(s) - w^{c,\delta}(s)) \|_\half^2 \, ds
	% \right)^\half
	% \nn\\
	& 
	\le \left( \int_0^t \left( \ts{\frac{\|\op{J}\|^2 + \mu^2}{t^2}} \right) \, \eps^2 \, ds \right)^\half
	= \left(\ts{\frac{\|\op{J}\|^2 + \mu^2}{t}} \right)^\half \, \eps\,.
	\label{eq:aux-limit-7a}
\end{align}
% in which the definition of $\op{J}\in\mathcal{L}(\cX)$ has been applied. 
% In particular, the third equality above follows by \er{eq:op-A-ass-2d}, while the first inequality is immediate from the fact that $\op{J}\in\mathcal{L}(\cX)$. 
% The last equality and inequality follow by \er{eq:op-A-ass-2c} and \er{eq:aux-limit-1b} respectively. 
Consequently, Lemma \ref{lem:near-opt-input-bound}, \er{eq:aux-limit-6b}, and \er{eq:aux-limit-7a} together imply that
\begin{align}
	\left\| \op{J}^\mu\, w^{c,\delta} \right\|_{\cW[0,t]}
	& \le \ol{M}^\mu(t,x,z)\,,
	\qquad
	\left\| \op{J}^\mu\, \hat w^{c,\delta}  \right\|_{\cW[0,t]}
	\le \ol{M}^\mu(t,x,z) + \left(\ts{\frac{\|\op{J}\|^2 + \mu^2}{t}}\right)^\half \, \eps\,.
	\label{eq:aux-limit-7b}
\end{align}
Similarly, \er{eq:aux-limit-1b} and \er{eq:aux-limit-2b} imply that
\begin{align}
	\left\| \hat\xi^{c,\delta} - \xi^{c,\delta} \right\|_{\cW[0,t]}
	% & = \left(\int_0^t \left\|  \hat\xi^{c,\delta}(s) - \xi^{c,\delta}(s) \right\|_\half^2 \, ds\right)^\half
	% \nn\\
	& 
	% = \left(\int_0^t \left\| \ts{\frac{s}{t}}\, \left( z - \xi^{c,\delta}(t) \right) \right\|_\half^2 \, ds\right)^\half
	= \ts{\frac{1}{t}} \left\| z - \xi^{c,\delta}(t) \right\|_\half \left( \int_0^t s^2 \, ds \right)^\half
	% \nn\\
	% & 
	\le
	\ts{\frac{1}{t}} \, \eps \, \left( \ts{\frac{t^3}{3}} \right)^\half
	% \nn\\
	% & 
	= \left(\ts{\frac{t}{3}}\right)^{\frac{1}{2}}\, \eps\,.
	\label{eq:aux-limit-8a}
\end{align}
As $\hat w^{c,\delta}$ is sub-optimal in the definition \er{eq:W-infty} of $W^{\mu,\infty}(t,x,z)$, while $\psi^{\mu,\infty}(\hat\xi^{c,\delta}(t),z) = 0$ by \er{eq:psi} and \er{eq:aux-limit-2b},
\begin{align}
	W^{\mu,\infty}(t,x,z)
	& \ge 
	\int_0^t \ts{\frac{\kappa}{2}} \, \| \hat \xi^{c.\delta}(s)\|_\half^2 - \ts{\frac{m}{2}}\, \| \op{J}^\mu\, \hat w^{c,\delta}(s)\|_\half^2 \, ds 
			+ \psi^{\mu,\infty}(\hat\xi^{c,\delta}(t),z)
	% \nn\\
	% & 
	= \ts{\frac{\kappa}{2}} \, \|\hat\xi^{c,\delta}\|_{\cW[0,t]}^2 - \ts{\frac{m}{2}} \, \|\op{J}^\mu\, \hat w^{c,\delta}\|_{\cW[0,t]}^2\,,
	\nn
\end{align}
or, equivalently,
$
	\|\hat \xi^{c,\delta}\|_{\cW[0,t]}^2
	\le \ts{\frac{2}{\kappa}} \, W^{\mu,\infty}(t,x,z) + \ts{\frac{m}{\kappa}} \, \| \op{J}^\mu\, \hat w^{c,\delta} \|_{\cW[0,t]}^2
$.
Applying the second inequality of \er{eq:aux-limit-7b}, the triangle inequality, and \er{eq:aux-limit-8a} yields the respective inequalities
\begin{align}
	\left\| \hat\xi^{c,\delta} \right\|_{\cW[0,t]}
	& \le \widetilde{M}^{\mu,\eps}(t,x,z)\,,
	\qquad
	\left\| \xi^{c,\delta} \right\|_{\cW[0,t]}
	\le \widetilde{M}^{\mu,\eps}(t,x,z) + \left(\ts{\frac{t}{3}}\right)^{\frac{1}{2}}\, \eps\,,
	\label{eq:aux-limit-8b}
\end{align}
in which
$
	\widetilde{M}^{\mu,\eps}(t,x,z)
	\doteq
	[ \ts{\frac{2}{\kappa}}\, W^{\mu,\infty}(t,x) + \ts{\frac{m}{\kappa}} \left(\ts{\frac{\|\op{J}\|^2 + \mu^2}{t}}\right)^\half \, \eps 
				+ \ts{\frac{m}{\kappa}} \, \ol{M}^\mu(t,x,z) ]^\half
$.
So, combining \er{eq:aux-limit-7a}, \er{eq:aux-limit-7b}, \er{eq:aux-limit-8a}, \er{eq:aux-limit-8b} in \er{eq:aux-limit-5} yields that
\begin{align}
	& \Delta^{\mu,\eps} (t,x,z)
	\le \ts{\frac{\kappa}{\sqrt{2}}} (
		\| \hat\xi^{c,\delta} \|_{\cW[0,t]} + \| \xi^{c,\delta} \|_{\cW[0,t]}
	) \| \hat\xi^{c,\delta} - \xi^{c,\delta} \|_{\cW[0,t]}
	\nn\\
	& \hspace{30mm}
	+ \ts{\frac{m}{\sqrt{2}}} (
		\| \op{J}^\mu\,  \hat w^{c,\delta} \|_{\cW[0,t]} + \| \op{J}^\mu\,  w^{c,\delta} \|_{\cW[0,t]}
	) \| \op{J}^\mu\, ( \hat w^{c,\delta} - w^{c,\delta} ) \|_{\cW[0,t]}
	\nn\\
	& \le \ts{\frac{\kappa}{\sqrt{2}}} \left(
		2\, \widetilde M^{\mu,\eps}(t,x,z) + \left(\ts{\frac{t}{3}}\right)^{\frac{1}{2}}\, \eps
	\right) \left(\ts{\frac{t}{3}}\right)^{\frac{1}{2}}\, \eps
	+ \ts{\frac{m}{\sqrt{2}}} \left( 2\, \ol{M}^\mu(t,x,z) + \left(\ts{\frac{\|\op{J}\|^2+\mu^2}{t}}\right)^\half \, \eps \right)
	\left(\ts{\frac{\|\op{J}\|^2 + \mu^2}{t}}\right)^\half \, \eps
	= O(\eps)
	\label{eq:aux-limit-9}
\end{align}
This bound is independent of $c$ and $\delta$. Fix any $\bar\eps\in\R_{>0}$. With $t$, $x$, and $z$ given, there exists an $\eps\in\R_{>0}$ such that 
$\Delta^{\mu,\eps}(t,x,z) < \bar\eps$. Inequality \er{eq:aux-limit-4} then implies that
$
	W^{\mu,c}(t,x,z) - \delta 
	< W^{\mu,\infty}(t,x,z) + \bar\eps
$
for all $c\in(\bar c,\infty)$ and $\delta\in(0,\bar\delta)$. So, sending $\delta\rightarrow 0^+$ and $c\rightarrow\infty$ yields
$
	\lim_{c\rightarrow\infty} W^{\mu,c}(t,x,z)
	\le W^{\mu,\infty}(t,x,z) + \bar\eps
$. 
As $\bar\eps\in\R_{>0}$ is arbitrary, it follows that
$
	\lim_{c\rightarrow\infty} W^{\mu,c}(t,x,z) \le W^{\mu,\infty}(t,x,z)
$
for any $t\in[0,\bar t^\mu)$, $x,z\in\cX_\half$. Combining this inequality with \er{eq:aux-limit-1} completes the proof.
\end{proof}

%%
%%		Step 2.
%%

\subsubsection{Verification theorem -- {\dtwo}}
The second step in explicitly characterizing the fundamental solution to the optimal control problem \er{eq:W} utilizes a verification theorem. In stating this theorem, it is convenient to define operator $\op{I}_\mu\in\bo(\cX)$ by 
\begin{align}
	\op{I}_\mu \, y
	& \doteq (\op{I} + \mu^2\, \op{A})^{-1}\, y\,,\quad \dom(\op{I}_\mu) \doteq \cX,\, \ran(\op{I}_\mu) = \cX_0\,,
	\label{eq:op-I-mu}
\end{align}
where boundedness, the stated range, and a number of other useful properties follow by Lemma \ref{lem:op-I-mu-properties}. 
\begin{theorem}[Verification]
\label{thm:verify}
Given $\mu\in\R_{>0}$, $\bar t^\mu\in\R_{\ge 0}$ as per \er{eq:t-bar-mu}, $c\in\R_{\ge 0}$, and $z\in\cX_\half$, suppose that a functional $W\in
%% C(\overline{\mathbb{D}};\, \cZ) \cap C^1(\mathbb{D};\cZ)
%% Fundamental
C([0,\bar t^\mu]\times\cX_\half\times\cX_\half;\R)\cap C^1((0,\bar t^\mu)\times\cX_\half\times\cX_\half; \R)$
%  C^{01}((0,\bar t^\mu)\times\cX_\half\times\cX_\half;\R)$
satisfies
\begin{align}
	& 0 = 
	-\pdtone{W}{t}(t,x,z) + H(x, \ggrad_x W(t,x,z))\,,
	\label{eq:verify-DPE}
	\\
	% & \ggrad_x W(t,x,z)\in\cX_\half\,,
	% \label{eq:verify-grad}
	% \\
	& W(0,x,z) = 
	\psi^{\mu,c}(x,z)
	\label{eq:verify-IC}
\end{align}
for all $t\in[0,\bar t^\mu)$ and $x\in\cX_\half$, where $\ggrad_x W(t,x,z)\in\cX_\half$ denotes the {\Frechet} derivative of $W(t,\cdot,z)$ at $x\in\cX_\half$, defined with respect to inner product $\langle\, , \rangle_\half$ on $\cX_\half$, and $H:\cX_\half\times\cX_\half\mapsinto\R$ is the Hamiltonian
\begin{align}
	H(x,p)
	& \doteq \ts{\frac{\kappa}{2}} \, \|x\|_\half^2 + \ts{\frac{1}{2\, m}} \, \| \op{I}_\mu^\half\, \opAsqrt\, p\|_\half^2\,,
	\label{eq:verify-H}
\end{align}
in which $\op{I}_\mu^\half$ is the unique square root of $\op{I}_\mu$ of \er{eq:op-I-mu}, see Lemma \ref{lem:op-I-mu-properties}. Then, $W(t,x,z) \ge J_{m,\psi^{\mu,c}(\cdot, z)}^\mu(t,x,w)$ for all $x\in\cX_\half$, $w\in\cW[0,t]$, $t\in[0,\bar t^\mu)$. Furthermore, if there exists a mild solution $\xi^*$ as per \er{eq:mild} corresponding to a distributed input $w^*$ defined via the feedback characterization
\begin{align}
	w^*(s)
	& \doteq k(s,\, \xi^*(s))\,,
	\qquad
	k(s,x) \doteq \ts{\frac{1}{m}}\, \opAsqrt\, \op{I}_\mu\, \opAsqrt\, \ggrad_x W(t-s,x,z)\,,
	\label{eq:w-star}
\end{align}
such that $\xi^*(s)\in\cX_\half$ for all $s\in[0,t]$, then $W(t,x,z) = J_{m,\psi^{\mu,c}(\cdot, z)}^\mu(t,x,w^*)$, and $W(t,x,z) = W^{\mu,c}(t,x,z)$.
\end{theorem}

The verification Theorem \ref{thm:verify} may be proved via completion of squares and a chain rule for {\Frechet} differentiation, summarized via the following preliminary lemmas.

\begin{lemma}
\label{lem:half-L2-squares}
Given any $p\in\cX_\half$, the quadratic functional $\pi_p^\mu:\cX_\half\mapsto\R$,
$
	\pi_p^\mu(w) 
	\doteq \langle p,\, w \rangle_\half - \ts{\frac{m}{2}} \, \|\op{J}^\mu\, w\|_\half^2
	% \label{eq:squares-pi}
$,
satisfies
$
	\sup_{w\in\cX_\half}
	\pi_p^\mu(w) = \pi_p^\mu(w^*) = \ts{\frac{1}{2\, m}} \, \| \op{I}_\mu^\half \, \opAsqrt\, p \|_\half^2
	% \label{eq:half-L2-squares}
$
with $w^* \doteq \ts{\frac{1}{m}}\, \opAsqrt\, \op{I}_\mu \, \opAsqrt\, p\in\cX_\half$ and $\op{I}_\mu$ as per \er{eq:op-I-mu}.
\end{lemma}
\begin{proof}
Fix $p\in\cX_\half$ and $w\in\cX_\half$. Note that $\opAsqrt\, p\in\cX$ and $\op{J}\, w\in\cX_0 = \dom(\op{A})$, by Lemma \ref{lem:op-A-properties}. Note also that $(\op{J}^2 + \mu^2\, \op{I})\, w = \op{J}\, (\op{I} + \mu^2\, \op{A} )\, \op{J}\, w  = \op{J}\, \op{I}_\mu^{-1}\, \op{J}\, w$. Hence, by definition of $\pi_p^\mu(w)$ and \er{eq:op-J-mu},
\begin{align}
	& \pi_p^\mu(w)
	= -\ts{\frac{m}{2}} \left[ \langle w,\, \op{J}\, \op{I}_\mu^{-1}\, \op{J}\, w \rangle_\half - \ts{\frac{2}{m}}\, \langle p,\, w \rangle_\half 
	\right]\,,
	\label{eq:squares-1}
\end{align}
As $\op{I}_\mu^{-1}$ has a unique, positive, self-adjoint and boundedly invertible square root (Lemma \ref{lem:op-I-mu-properties}), it follows that
$
	\langle p,\, w\rangle_\half = \langle \op{I}_\mu^\half \, \opAsqrt\, p,\, \op{I}_\mu^\minushalf\, \op{J}\, w \rangle_\half
$ and
$
	\langle w,\, \op{J}\, \op{I}_\mu^{-1}\, \op{J}\, w \rangle_\half = \| \op{I}_\mu^\minushalf\, \op{J}\, w\|_\half^2
$,
where $\op{I}_\mu^\half:\cX\mapsinto\cX_\half$ and $\op{I}_\mu^\half\in\bo(\cX)$. Substituting in \er{eq:squares-1},
\begin{align}
	\pi_p^\mu(w) 
	= -\ts{\frac{m}{2}} \left[  \| \op{I}_\mu^\minushalf\, \op{J}\, w\|_\half^2 - 
		\ts{\frac{2}{m}} \, \langle \op{I}_\mu^\half \, \opAsqrt\, p,\, \op{I}_\mu^\minushalf\, \op{J}\, w \rangle_\half \right] 
	& = \ts{\frac{1}{2\, m}} \, \| \op{I}_\mu^\half \, \opAsqrt\, p \|_\half^2 - \ts{\frac{m}{2}} \,
	\| \op{I}_\mu^\minushalf\, \op{J}\, w - \ts{\frac{1}{m}}\, \op{I}_\mu^\half \, \opAsqrt\, p \|_\half^2
	\nn\\
	& =  \ts{\frac{1}{2\, m}} \, \| \op{I}_\mu^\half \, \opAsqrt\, p \|_\half^2 - \ts{\frac{m}{2}}
	\| \op{I}_\mu^\minushalf\, \op{J}\, ( w - w^* ) \|_\half
	\nn
\end{align}
for all $w\in\cX_\half$, where $w^*\in\cX_\half$ is as per the lemma statement. Taking the supremum of $\pi_p^\mu(w)$ over $w\in\cX_\half$, \begin{align}
	\sup_{w\in\cX_\half} \pi_p^\mu(w)
	& = \ts{\frac{1}{2\, m}} \, \| \op{I}_\mu^\half \, \opAsqrt\, p \|_\half^2  
			- \ts{\frac{m}{2}} \! \inf_{w\in\cX_\half} \! \| \op{I}_\mu^\minushalf\, \op{J}\, ( w - w^* ) \|_\half
	=  \ts{\frac{1}{2\, m}} \, \| \op{I}_\mu^\half \, \opAsqrt\, p \|_\half^2 = \pi_p^\mu(w^*),
	\nn
\end{align}
as per the lemma statement, with the supremum attained at $w^*\in\cX_\half$.
\end{proof}

The following lemma is standard and its proof is omitted (see for example \cite{BDDM:07}).

\begin{lemma}[Fundamental Theorem of Calculus]
\label{lem:fund-calc}
Given any $\mu\in\R_{>0}$, $t\in(0,\bar t^\mu)$ with $\bar t^\mu\in\R_{>0}$ as per \er{eq:t-bar-mu}, $x\in\cX_\half$,  $w\in\cW[0,t]$, let $\xi:[0,t]\mapsinto\cX_\half$ denote the mild solution \er{eq:mild}, and let
$
	W\in C([0,\bar t^\mu]\times\cX_\half\times\cX_\half;\R)\cap C^1((0,\bar t^\mu)\times\cX_\half\times\cX_\half; \R)
$.
Then, for any $\tau\in[0,t]$, 
\begin{align}
	W(t-\tau,\, \xi(\tau), z) - W(t,x,z)
	& = \int_0^\tau -\ddtone{W}{t}(t-s,\, \xi(s),\, z) + \langle \ggrad_x W(t-s,\xi(s),z),\, w(s) \rangle_\half \, ds\,,
	\nn
\end{align}
where $\ggrad_x W(t,x,z)\in\cX_\half$ denotes the Riesz representation of the {\Frechet} derivative of $W(t,\cdot,z)$ at $x$, defined with respect to inner product $\langle\, , \rangle_\half$ on $\cX_\half$. 
\end{lemma}
%%%%% if falsed fundamental theorem of calculus lemma %%%%%
\if{false}

\begin{proof}
Fix any $\mu\in\R_{>0}$, $t\in(0,\bar t^\mu)$, $x\in\cX_\half$, $w\in\cW[0,t]$. By definition \er{eq:mild}, the mild solution $\xi:[0,t]\mapsinto\cX_\half$ is absolutely continuous, and hence continuous and bounded, on $[0,t]$. That is, there exists an $R(t,x,w)<\infty$ such that $\|\xi(s)\| \le R\doteq R(t,x,w) < \infty$ for all $s\in[0,t]$. By the assumed smoothness of $W$, note that $W(t,\cdot,z)$ is Lipschitz on $B_R(0)$. So, combining the absolute continuity of $\xi(s)$ as a function of $s\in[0,t]$ with the Lipschitz behaviour of $W(t,\cdot,z)$ on the range of $\xi(\cdot)$ yields that absolute continuity of $W(t - \cdot,\xi(\cdot),z):[0,t]\mapsinto \R$. That is, 
\begin{align}
	W(t-\tau,\xi(\tau),z)-W(t,x,z)
	& = \int_{0}^{\tau} \dot\omega(s) \, ds\,,
	\qquad 
	\omega(s) \doteq W(t-s,\xi(s),z)\,,
	\label{eq:chain-Frechet-1}
\end{align}
for all $\tau\in[0,t]$. Let $\cV \doteq \R \otimes \cX_\half$ and define an inner product on $\cV$ by
\begin{align}
	\langle \zeta_1,\, \zeta_2 \rangle_{\cV}
	& \doteq
	s_1\, s_2 + \langle \xi_1,\, \xi_2 \rangle_\half\,,
	\qquad \zeta_1 = (s_1,\xi_1)\,, \ \zeta_2 = (s_2,\xi_2)\,,
	\label{eq:chain-Frechet-2}
\end{align}
for all $\zeta_1,\, \zeta_2\in\cV$. Define $\zeta:[0,t]\mapsinto\cV$ and $\Theta:\cV\mapsinto\R$ by
\begin{align}
	\zeta(s) & \doteq (\sigma(s), \xi(s))\,,
	\qquad
	\Theta(\zeta) \doteq W(t - s_\zeta,\, \xi_\zeta,\, z)\,,
	\ \zeta = (s_\zeta,\xi_\zeta)\,,
	\label{eq:chain-Frechet-3}
\end{align}
where $\sigma:(0,t)\mapsinto (0,t)$ is defined by $\sigma(s)\doteq s$, and $\xi(\cdot)$ is the mild solution \er{eq:mild} as per the lemma statement. By definition \er{eq:chain-Frechet-1} of $\omega$, note that
\begin{align}
	\omega(s)
	& = \Theta(\zeta(s))\,.
	\label{eq:chain-Frechet-4}
\end{align}
Recall that $\xi(\cdot)$ is {\Frechet} differentiable on $(0,\bar t^\mu)$ by \er{eq:dynamics}, \er{eq:mild}. By definition, $\sigma(\cdot)$ is also {\Frechet} differentiable. Hence, $\zeta(\cdot)$ must also be {\Frechet} differentiable by inspection of \er{eq:chain-Frechet-3}. In particular, applying definition \er{eq:Frechet} of {\Frechet} differentiability, for each $s\in(0, t)$ there exists an operator $d\zeta(s)\in\cL((-s, t-s);\cV)$ such that
\begin{align}
	0 
	& = \lim_{\eps\rightarrow 0} \frac{\| \zeta(s + \eps) - \zeta(s) - d\zeta(s)\, \eps\|_\cV}{|\eps|}\,,
	\label{eq:chain-Frechet-5}
\end{align}
where the {\Frechet} differential is given by $d\zeta(s)\, \eps = (\eps\, \dot\sigma(s),\eps\, \dot\xi(s)) = (\eps,\, \eps\, \dot\xi(s)) = \eps\, (1,\, \dot\xi(s))$. By inspection, \er{eq:chain-Frechet-5} implies that for each $s\in(0,t)$ there exists an operator $d^2 \zeta(s):(-s, t-s)\mapsinto\cV$ such that
\begin{align}
	\zeta(s+\eps) 
	& = \zeta(s) + d\zeta(s)\, \eps + \eps\, d^2 \zeta(s)\, \eps
	\label{eq:chain-Frechet-6}
\end{align}
where $\|d^2 \zeta(s) \, \eps\|_\cV\rightarrow 0$ as $\eps\rightarrow 0$, with $d^2 \zeta(s) \, \eps \doteq \frac{\zeta(s + \eps) - \zeta(s) - d\zeta(s)\, \eps}{\eps}$ for all $\eps\in(-s, t-s)$.

Similarly, $W(\cdot,\cdot,z)$ is {\Frechet} differentiable on $(0,t)\times\cX_\half$ by the assumed smoothness of $W$. That is, $\Theta$ of \er{eq:chain-Frechet-3} is {\Frechet} differentiable at $\zeta=(s_\zeta,\xi_\zeta)\in\cV$, $s_\zeta\in(0,t)$. So, applying the definition \er{eq:Frechet} of {\Frechet} differentiability, there exists an operator $d\Theta(\zeta)\in\cL(\cV;\R)$ such that
\begin{align}
	0 & = \lim_{\|h\|_\cV\rightarrow 0} \frac{|\Theta(\zeta+h) - \Theta_z(\zeta) - d\Theta_z(\zeta)\, h|}{\|h\|_\cV}\,,
	\label{eq:chain-Frechet-6a}
\end{align}
where the {\Frechet} differential is (by application of the Riesz representation theorem)
\begin{align}
	d\Theta(\zeta)\, h & = \langle \ggrad_\zeta \Theta(\zeta),\, h \rangle_\cV
	= \left\langle \left(-\pdtone{W}{t}(t-s_\zeta,\xi_\zeta,z),\, \ggrad_x W(t - s_\zeta,\xi_\zeta, z) \right),\, h \right\rangle_\cV\,, 
	\qquad \zeta = (s_\zeta,\xi_\zeta)\,.
	\label{eq:chain-Frechet-6b}
\end{align}
By inspection, \er{eq:chain-Frechet-6a} implies that for each $\zeta = (s_\zeta,\xi_\zeta)\in\cV$, $s_\zeta\in(0,t)$, there exists an operator $d^2\Theta(\zeta):\cV\mapsinto\R$ (not necessarily linear) such that
\begin{align}
	\Theta(\zeta+h)
	& = \Theta(\zeta) + d\Theta(\zeta)\, h + \|h\|_\cV\, d^2\Theta(\zeta)\, h
	\label{eq:chain-Frechet-7a}
\end{align}
where $|d^2\Theta(\zeta)\, h| \rightarrow 0$ as $\|h\|_\cV\rightarrow 0$, with $d^2\Theta(\zeta)\, h \doteq \frac{\Theta(\zeta+h) - \Theta(\zeta) - d\Theta(\zeta)\, h}{\|h\|_\cV}$ for all $h = (s_h,\xi_h)\in\cV$, $s_h\in(-s_\zeta,t-s_\zeta)$. 
% As $d\Theta(\zeta)\in\cL(\cV;\R)$, the Riesz representation theorem implies that there exists $\theta(\zeta)\in\cV$ such that $d\Theta(\zeta)\, h = \langle \theta(\zeta),h \rangle_\cV$, where $\theta(\zeta) = \ggrad_\zeta \Theta(\zeta)$ is the {\Frechet} derivative of $\Theta$ at $\zeta$. 
Define the operator $\widetilde{d^2\Theta}(\zeta):\cV\mapsinto\cV$ by
\begin{align}
	\widetilde{d^2\Theta}(\zeta) \, h
	& \doteq \left\{ \ba{ccl}
			\left(\frac{d^2\Theta(\zeta) \, h}{\|h\|_\cV} \right)\, h \,,
			&&	\|h\|_\cV\ne 0\,,
			\\
			0\,,
			&& 	\|h\|_\cV = 0\,,
	\ea \right.
	\nn
\end{align}
and note that $\widetilde{d^2\Theta}(\zeta)\, h$ is continuous at $h=0\in\cV$ as $\|\widetilde{d^2\Theta}(\zeta) \, h\|_\cV = \left( \frac{|d^2\Theta(\zeta) \, h|}{\|h\|_\cV} \right) \|h\|_\cV = |d^2\Theta(\zeta)\, h| \rightarrow 0$ as $\|h\|_\cV\rightarrow 0$. Further note that $\langle \widetilde{d^2\Theta}(\zeta) \, h,\, h \rangle_\cV = \|h\|_\cV\, d^2\Theta(\zeta)\, h$. Hence, recalling \er{eq:chain-Frechet-6b}, \er{eq:chain-Frechet-7a} may be equivalently written as
\begin{align}
	\Theta(\zeta+h)
	& = \Theta(\zeta) + \left\langle \ggrad_\zeta \Theta(\zeta) + \widetilde{d^2\Theta}(\zeta)\, h,\, h \right\rangle_\cV\,.
	\label{eq:chain-Frechet-7b} 
\end{align}
Hence, combining \er{eq:chain-Frechet-4}, \er{eq:chain-Frechet-6} and \er{eq:chain-Frechet-7b}, it follows that for a.e. $s\in(0,t]$,
\begin{align}
	\frac{\omega(s+\eps) - \omega(s)}{\eps}
	& = \frac{\Theta(\zeta(s+\eps)) - \Theta(\zeta(s))}{\eps}
	\nn\\
	& = \frac{\Theta(\zeta(s) + d\zeta(s)\, \eps + \eps\, d^2\zeta(s)\, \eps) - \Theta(\zeta(s))}{\eps}
	\nn\\
	& = \frac{\left\langle \ggrad_\zeta \Theta(\zeta(s)) + 
		\widetilde{d^2\Theta}(\zeta(s)) \, \left[ d\zeta(s)\, \eps + \eps\, d^2\zeta(s)\, \eps \right],\, 
		d\zeta(s)\, \eps + \eps\, d^2\zeta(s)\, \eps 
	\right\rangle_\cV}{\eps}
	\nn\\
	& = \left\langle \ggrad_\zeta \Theta(\zeta(s)),\, \frac{1}{\eps} \, d\zeta(s)\, \eps \right\rangle_\cV + \theta(\eps)
	\nn\\
	& = \left\langle
		\left(-\pdtone{W}{t}(t-s,\xi(s),z),\, \ggrad_x W(t - s,\xi(s), z) \right),\, (1,\dot\xi(s))
	\right\rangle_\cV + \theta(\eps)
	\nn\\
	& = -\pdtone{W}{t}(t-s,\xi(s),z) + \left\langle \ggrad_x W(t - s,\xi(s), z),\, \dot\xi(s) \right\rangle_\half + \theta(\eps)
	\label{eq:chain-Frechet-8a}
\end{align}
where 
\begin{align}
	\theta(\eps)
	& \doteq
	\left\langle \ggrad_\zeta \Theta(\zeta(s)), \, d^2\zeta(s)\, \eps  \right\rangle_\cV + 
	\left\langle \widetilde{d^2\Theta}(\zeta(s)) \, \left[ d\zeta(s)\, \eps + \eps\, d^2\zeta(s)\, \eps \right],\,
	\frac{1}{\eps}\, d\zeta(s)\, \eps + d^2\zeta(s)\, \eps \right\rangle_\cV.
	\label{eq:chain-Frechet-8b}
\end{align}
Taking the limit as $\eps\rightarrow 0$ as per \er{eq:Frechet} and recalling the limit properties of $\widetilde{d^2\Theta}(\zeta(s))$ and $d\zeta(s)$,
\begin{align}
	& \lim_{\eps\rightarrow 0}
	\frac{\left| \omega(s+\eps) - \omega(s) - 
		\left( -\pdtone{W}{t}(t-s,\xi(s),z) + \left\langle \ggrad_x W(t - s,\xi(s), z),\, \dot\xi(s) \right\rangle_\half \right) \eps
	\right|}{|\eps|}
	= \lim_{\eps\rightarrow 0} |\theta(\eps)|
	\nn\\
	&
	\le 
	\lim_{\eps\rightarrow 0} \left| \left\langle \ggrad_\zeta \Theta(\zeta(s)), \, d^2\zeta(s)\, \eps  \right\rangle_\cV \right| 
	+
	\lim_{\eps\rightarrow 0} \left| \left\langle \widetilde{d^2\Theta}(\zeta(s)) \, \left[ d\zeta(s)\, \eps + \eps\, d^2\zeta(s)\, \eps \right],\,
	\frac{1}{\eps}\, d\zeta(s)\, \eps + d^2\zeta(s)\, \eps \right\rangle_\cV \right|
	\nn\\
	& \le 
	\lim_{\eps\rightarrow 0} \left\|  \ggrad_\zeta \Theta(\zeta(s)) \right\|_\cV \left\|  d^2\zeta(s) \, \eps \right\|_\cV +
	% \lim_{\eps\rightarrow 0} \left\|\widetilde{d^2\Theta}(\zeta(s)) \, \left[ d\zeta(s)\, \eps + \eps\, d^2\zeta(s)\, \eps \right] \right\|_\cV\,
	%	\left\| \frac{1}{\eps} d\zeta(s)\, \eps + d^2\zeta(s)\, \eps \right\|_\cV
	% \nn\\
	% & \hspace{1cm}
	% = 
	\lim_{\eps\rightarrow 0} \left\|\widetilde{d^2\Theta}(\zeta(s)) \, \left[ d\zeta(s)\, \eps + \eps\, d^2\zeta(s)\, \eps \right] \right\|_\cV\,
		\lim_{\eps\rightarrow 0} \left\| \frac{1}{\eps} d\zeta(s)\, \eps + d^2\zeta(s)\, \eps \right\|_\cV
	= 0\,,
	\nn
\end{align}
so that the {\Frechet} derivative of $\omega(\cdot)$ at $s\in(0,t)$ is $\dot\omega(s) = -\pdtone{W}{t}(t-s,\xi(s),z) + \left\langle \ggrad_x W(t - s,\xi(s), z),\, \dot\xi(s) \right\rangle_\half$. Recalling \er{eq:chain-Frechet-1} completes the proof.
\end{proof}

\fi
%%%%% end of if falsed fundamental theorem of calculus lemma %%%%%

Lemmas \ref{lem:half-L2-squares} and \ref{lem:fund-calc} facilitate the proof of the verification Theorem \ref{thm:verify}.

\begin{proof}[Theorem \ref{thm:verify}]
Given $\bar t^\mu\in\R_{>0}$ and $z\in\cX_\half$, let $W\in C([0,\bar t^\mu]\times\cX_\half\times\cX_\half;\R)\cap C^1((0,\bar t^\mu)\times\cX_\half\times\cX_\half; \R)$ denote a solution of \er{eq:verify-DPE} -- \er{eq:verify-IC} as per the theorem statement. Fix $\ol{w}\in\cW[0,t]$, $t\in[0,\bar t^\mu)$, and let $\ol{\xi}(\cdot)$ denote the mild solution \er{eq:mild} of \er{eq:dynamics} with $\xi(0) = x\in\cX_\half$ and $w(s) = \ol{w}(s)$, $s\in[0,t]$. Recall that $\ol{\xi}(s)\in\cX_\half$ for all $s\in[0,t]$. Set $p(s) \doteq \ggrad_x W(t-s, \ol{\xi}(s),z)$, and note that $p(s)\in\cX_\half$ for all $s\in[0,t]$. Hence, both terms in $H(\ol{\xi}(s), p(s))$ as per \er{eq:verify-H} are well-defined for all $s\in[0,t]$. 
% Considering the $\ts{\frac{1}{2\, m}}\, \| \op{I}_\mu^\half\, \opAsqrt\, p(s)\|_\half^2$ term, 
In particular, Lemma \ref{lem:half-L2-squares} implies that
\begin{align}
	\ts{\frac{1}{2\, m}} \, \left\| \op{I}_\mu^\half\, \opAsqrt\, \ggrad_x W(t-s, \, \ol{\xi}(s), \, z) \right\|_\half^2
	& \ge \langle \ggrad_x W(t-s,\, \ol{\xi}(s),z),\, \ol{w}(s) \rangle_\half - \ts{\frac{m}{2}}\, \|\op{J}^\mu\, \ol{w}(s)\|_\half^2
	\label{eq:verify-1}
\end{align}
for all $s\in[0,t]$. Hence, substituting \er{eq:verify-1} in \er{eq:verify-DPE} yields that
$	0 
	\ge 
	- \pdtone{W}{t}(t-s, \, \ol{\xi}(s), \, z) 
	+ \langle \ggrad_x W(t-s,\, \ol{\xi}(s),z),\,\ol{w}(s) \rangle_\half 
	+ \ts{\frac{\kappa}{2}} \, \left\| \ol{\xi}(s) \right\|_\half^2 - \ts{\frac{m}{2}} \left\|\op{J}^\mu\, \ol{w}(s) \right\|_\half^2
$
for all $s\in[0,t]$. Integrating with respect to $s\in[0,t]$, the Fundamental Theorem of Calculus (Lemma \ref{lem:fund-calc}) implies that
\begin{align}
	0 
	& \ge \int_0^t - \pdtone{W}{t}(t-s, \, \ol{\xi}(s), \, z) 
	+  \langle \ggrad_x W(t-s,\, \ol{\xi}(s),z),\, \ol{w}(s) \rangle_\half \, ds
	+ \int_0^t  \ts{\frac{\kappa}{2}} \, \left\| \ol{\xi}(s) \right\|_\half^2 - \ts{\frac{m}{2}} \left\| \op{J}^\mu\, \ol{w}(s) \right\|_\half^2 \, ds
	\nn\\
	& = W(0,\, \ol{\xi}(t),\, z) - W(t,x,z)
	+ \int_0^t  \ts{\frac{\kappa}{2}} \, \left\| \ol{\xi}(s) \right\|_\half^2 - 
	\ts{\frac{m}{2}} \left\| \op{J}^\mu\, \ol{w}(s) \right\|_\half^2 \, ds\,.
	\nn
\end{align}
Applying \er{eq:verify-IC} and \er{eq:payoff} to this yields $W(t,x,z) \ge J_{m,\psi^{\mu,c}(\cdot,z)}^\mu(t,x,\ol{w})$ as per the first assertion. In order to prove the second assertion, define $w^*$ as per \er{eq:w-star}. By assumption, $\xi^*(s)\in\cX_\half$ and $\ggrad_x W(t,\xi^*(s),z)\in\cX_\half$ for all $s\in[0,t]$. Hence, the argument from \er{eq:verify-1} onwards may be repeated, this time with equality, yielding that $W(t,x,z) = J_{m,\psi^{\mu,c}(\cdot,z)}^\mu(t,x,w^*) = W^{\mu,c}(t,x,z)$ as required.
\end{proof}

%%
%%		Step 3.
%%

\subsubsection{An explicit representation of the fundamental solution $W^{\mu,\infty}$ of \er{eq:W-infty}
% Application of Theorem \ref{thm:verify} and an explicit representation 
-- {\dthree}}

The third step in explicitly characterizing the fundamental solution to approximating optimal control problem \er{eq:W} involves the construction of a functional that satisfies the conditions of Theorem \ref{thm:verify}, followed by an application of the limit argument {\done}. To this end, define the bi-quadratic functional $\Wbreve^{\mu,c}:[0,\bar t^\mu)\times\cX_\half\times\cX_\half\mapsinto\R$ by
\begin{align}
	\Wbreve^{\mu,c}(t,x,z)
	& \doteq
	\demi \langle x,\, \op{P}^{\mu,c}(t)\, x \rangle_\half + 
	\langle x,\, \op{Q}^{\mu,c}(t)\, z \rangle_\half 
	+ \demi\, \langle z,\, \op{R}^{\mu,c}(t)\, z \rangle_\half\,,
	\label{eq:W-explicit}
\end{align}
where $\op{P}^{\mu,c}$, $\op{Q}^{\mu,c}$, $\op{R}^{\mu,c}:[0,\bar t^\mu)\mapsinto\bo(\cX_\half)$ denote operator-valued functions of time that satisfy the operator differential equations
\begin{align}
	\dot{\op{P}}^{\mu,c}(t)
	& = \kappa\, \op{I} + \ts{\frac{1}{m}}\, \op{P}^{\mu,c}(t)\, \opAsqrt\, \op{I}_\mu\, \opAsqrt\,\op{P}^{\mu,c}(t),
	& \hspace{-2mm}
	\op{P}^{\mu,c}(0) & = - c\, \op{M}_{\mu},
	\label{eq:op-P}
	\\
	 \dot{\op{Q}}^{\mu,c}(t)
	& = \ts{\frac{1}{m}}\, \op{P}^{\mu,c}(t)\, \opAsqrt\, \op{I}_\mu\, \opAsqrt\, \op{Q}^{\mu,c}(t),
	& \hspace{-2mm}
	\op{Q}^{\mu,c}(0) & = +c\, \op{M}_\mu,
	\label{eq:op-Q}
	\\
	 \dot{\op{R}}^{\mu,c}(t)
	& = \ts{\frac{1}{m}}\, (\op{Q}^{\mu,c})(t)'\, \opAsqrt\, \op{I}_\mu\, \opAsqrt\, \op{Q}^{\mu,c}(t),
	& \hspace{-2mm}
	\op{R}^{\mu,c}(0) & = -c\,\op{M}_\mu,
	\label{eq:op-R}
\end{align}
in which $\op{I}$ denotes the identity operator on $\cX_\half$, and 
\begin{align}
	\op{M}_\mu
	& \doteq (\op{K}_\mu)'\, \op{K}_\mu\in\bo(\cX_\half)
	\label{eq:op-M-and-K}
\end{align}
is self-adjoint, positive, and boundedly invertible by definition of $\op{K}_\mu$. It may be shown that the functional $\Wbreve^{\mu,c}$ of \er{eq:W-explicit} satisfies the conditions of the verification Theorem \ref{thm:verify}, thereby providing an explicit representation for the value function $W^{\mu,c}$ of \er{eq:W-c} in terms of the operator-valued functions $\op{P}^{\mu,c}$, $\op{Q}^{\mu,c}$, $\op{R}^{\mu,c}$.

%%		Explicit DPE solution.

\begin{theorem}
\label{thm:explicit}
The functionals $W^{\mu,c}$ of \er{eq:W-c} and $\Wbreve^{\mu,c}$ of \er{eq:W-explicit} are equivalent. That is,
\begin{align}
	W^{\mu,c}(t,x,z) & = \Wbreve^{\mu,c}(t,x,z) 
	\label{eq:explicit-2}
\end{align}
for all $t\in[0,\bar t^\mu)$, $x,z\in\cX_\half$.
\end{theorem}
\begin{proof} % [Theorem \ref{thm:explicit}]
As indicated above, it is sufficient to demonstrate that $\breve W^{\mu,c}$ of \er{eq:W-explicit} satisfies the conditions of Theorem 3.6. To this end, fix $t\in[0,\bar t^\mu)$, $x,\, z\in\cX_\half$. Firstly, in order to show that $\Wbreve^{\mu,c}$ satisfies \er{eq:verify-DPE}, note that $\Wbreve^{\mu,c}(\cdot,x,z)$ and $\Wbreve^{\mu,c}(t,\cdot,z)$ are {\Frechet} differentiable. In particular, % for all $t\in(0,\bar t^\mu)$, $x,z\in\cX_\half$,
\begin{align}
	\pdtone{\Wbreve^{\mu,c}}{t}(t,x,z)
	& = 
	\demi\, \langle x,\, \dot{\op{P}}^{\mu,c}(t)\, x \rangle_\half + \langle x,\, \dot{\op{Q}}^{\mu,c}(t)\, z \rangle_\half
	+ \demi\, \langle z,\, \dot{\op{R}}^{\mu,c}(t)\, z \rangle_\half\,,
	\label{eq:explicit-1a}
	\\
	\ggrad_x \Wbreve^c(t,x,z)
	& = \op{P}^{\mu,c}(t)\, x + \op{Q}^{\mu,c}(t)\, z\,.
	\label{eq:explicit-1b}
\end{align}
With a view to verifying that \er{eq:verify-DPE} holds, further note that
\begin{align}
	\ts{\frac{\kappa}{2}}\, \|x\|_\half^2
	& = \demi\, \langle x,\, \kappa\, \op{I}\, x \rangle_\half\,,
	\label{eq:explicit-1c}
	\\
	\ts{\frac{1}{2\, m}} \, \| \op{I}_\mu^\half\, \opAsqrt\, \ggrad_x \Wbreve^{\mu,c}(t,x,z) \|_\half^2
	& = \demi\, \langle x,\, \ts{\frac{1}{m}}\, \op{P}^{\mu,c}(t)\, \opAsqrt\, \op{I}_\mu\, \opAsqrt\,\op{P}^{\mu,c}(t)\, x \rangle_\half
	+ \langle x,\,  \ts{\frac{1}{m}}\, \op{P}^{\mu,c}(t)\, \opAsqrt\, \op{I}_\mu\, \opAsqrt\, \op{Q}^{\mu,c}(t)\, z \rangle_\half
	\nn\\
	& \qquad
	+ \demi\, \langle z,\, \ts{\frac{1}{m}}\, (\op{Q}^{\mu,c}(t))'\, \opAsqrt\, \op{I}_\mu\, \opAsqrt\, \op{Q}^{\mu,c}(t)\, z \rangle_\half\,,
	\label{eq:explicit-1d}
\end{align}
where the second equality also exploits the fact that $\opbreve{P}^{\mu,c}(t)$ is self-adjoint. Hence, substitution of \er{eq:explicit-1a} \er{eq:explicit-1b}, \er{eq:explicit-1c}, \er{eq:explicit-1d} in the right-hand side of \er{eq:verify-DPE} yields the bi-quadratic functional (in $x$ and $z$)
\begin{align}
	\demi\, \langle x,\, \op{X}(t)\, x \rangle_\half + 
	\langle x,\, \op{Y}(t)\, z \rangle_\half +
	\demi\, \langle z,\, \op{Z}(t)\, z \rangle_\half
	\label{eq:explicit-DPE}
\end{align}
in which $\op{X}(t) \doteq -\dot{\op{P}}^{\mu,c}(t) + \kappa\, \op{I} + \ts{\frac{1}{m}}\, \op{P}^{\mu,c}(t)\, \opAsqrt\, \op{I}_\mu\, \opAsqrt\,\op{P}^{\mu,c}(t)$, $\op{Y}(t)\doteq - \dot{\op{Q}}^{\mu,c}(t) + \ts{\frac{1}{m}}\, \op{P}^{\mu,c}(t)\, \opAsqrt\, \op{I}_\mu\, \opAsqrt\, \op{Q}^{\mu,c}(t)$, and $\op{Z}(t)\doteq -\dot{\op{R}}^{\mu,c}(t) + \ts{\frac{1}{m}}\, (\op{Q}^{\mu,c}(t))'\, \opAsqrt\, \op{I}_\mu\, \opAsqrt\, \op{Q}^{\mu,c}(t)$. However, \er{eq:op-P}, \er{eq:op-Q}, \er{eq:op-R} imply that these three operator-valued functions are identically zero, so that \er{eq:explicit-DPE} must be zero. Hence, the explicit functional $\Wbreve^c$ of \er{eq:W-explicit} satisfies \er{eq:verify-DPE}.

Secondly, \er{eq:explicit-1b}, in which $\op{P}^{\mu,c}(t),\, \op{Q}^{\mu,c}(t):\cX_\half\mapsinto\cX_\half$, implies that $\ggrad_x \Wbreve^{\mu,c}(t,x,z)\in\cX_\half$.

Finally, in order to show that $\Wbreve^{\mu,c}$ satisfies the initial condition \er{eq:verify-IC}, note by inspection of \er{eq:W-explicit}, the initial conditions of \er{eq:op-P}, \er{eq:op-Q}, \er{eq:op-R}, and the identities $\op{J}\, \opAsqrt = \op{I}$ and $\op{M}_\mu = (\op{K}_\mu)'\, \op{K}_\mu$, that
\begin{align}
	\Wbreve^{\mu,c}(0,x,z)
	& = \demi\, \langle x,\, \op{P}^{\mu,c}(0)\,  x \rangle_\half
		+ \langle  x,\, \op{Q}^{\mu,c}(0)\, z \rangle_\half
		+ \demi\, \langle z,\, \op{R}^{\mu,c}(0)\, z \rangle_\half
	\nn\\
	& = \demi\, \langle x,\, (-c\, \op{M}_\mu) \, x \rangle_\half
		+ \langle x,\, (+c\, \op{M}_\mu) \, z \rangle_\half
		+ \demi\, \langle z,\, (-c\, \op{M}_\mu)\, z \rangle_\half
	= - \ts{\frac{c}{2}} \, \| \op{K}_\mu\, (x - z) \|_\half^2
	= \psi^{\mu,c}(x,z)\,,
	\nn
\end{align}
as required by \er{eq:verify-IC}. That is, the explicit functional $\Wbreve^{\mu,c}$ of \er{eq:W-explicit} satisfies the conditions \er{eq:verify-DPE}, \er{eq:verify-IC} of Theorem \ref{thm:verify}. Consequently, $\Wbreve^{\mu,c}(t,x,z) = W^{\mu,c}(t,x,z)$. 
% As $t\in[0,\bar t^\mu)$, $x,\, z\in\cX_\half$ are arbitrary, the proof is complete.
\end{proof}

Theorem \ref{thm:explicit} provides a representation for $W^{\mu,c}$ of \er{eq:W-c}, via $\breve{W}^{\mu,c}$ of \er{eq:W-explicit}, in terms of  operator-valued functions $\op{P}^{\mu,c}$, $\op{Q}^{\mu,c}$, $\op{R}^{\mu,c}$ satisfying \er{eq:op-P}, \er{eq:op-Q}, \er{eq:op-R}. Candidate definitions for these functions are Riesz-spectral operator-valued functions of the form \er{eq:op-F}, see Appendix \ref{app:Riesz-functions}. In particular, define an operator-valued function $\opbreve{P}^{\mu,c}$ of the form \er{eq:op-F} by
\begin{align}
	\opbreve{P}^{\mu,c}(t)\, x
	& \doteq \sum_{n=1}^\infty p_n^{\mu,c}(t)\, \langle x,\, \tilde\varphi_n \rangle_\half\, \tilde\varphi_n\,,
	\label{eq:op-P-def}
\end{align}
where $\{p_n^{\mu,c}(t)\}_{n\in\N}$ denotes the set of eigenvalues of $\opbreve{P}^{\mu,c}(t)$ corresponding to its eigenvectors $\{\tilde\varphi_n\}_{n\in\N}$ defined by the Riesz basis \er{eq:basis-Riesz} for $\cX_\half$. Motivated by the initial condition $\opbreve{P}^{\mu,c}(0) = -c\, \op{M}_\mu$ specified in \er{eq:op-P}, restrict $\op{M}_\mu$ to be a Riesz-spectral operator of the form \er{eq:op-Riesz}, with simple (i.e. non-repeated) eigenvalues $\{m_n^\mu\}_{n\in\N}$. Note in particular that $p_n^{\mu,c}(0) = -c\, m_n^\mu$. Analogously define operator-valued functions $\opbreve{Q}^{\mu,c}$ and $\opbreve{R}^{\mu,c}$. Select the respective eigenvalues of the operators in the range of these three operator-valued functions to be
\begin{gather}
	p_n^{\mu,c}(t)
	\doteq - \ts{\frac{1}{\alpha_n^\mu}}\, \frac{1}{\tan\left( \omega_n^\mu\, t + \theta_n^{\mu,c} \right)}\,,
	% \label{eq:ode-p-sol}
	% \\
	\qquad
	q_n^{\mu,c}(t)
	\doteq + \ts{\frac{1}{\alpha_n^\mu}}\, \left( \frac{1}{1 + ( \frac{1}{\alpha_n^\mu\, m_n^\mu\, c})^2} \right)^\half\, 
	\frac{1}{\sin\left( \omega_n^\mu\, t + \theta_n^{\mu,c} \right)}\,,
	% \label{eq:ode-q-sol}
	\label{eq:ode-p-sol}
	\\
	r_n^{\mu,c}(t)
	\doteq - \ts{\frac{1}{\alpha_n^\mu}} \left( \frac{1}{1 + (\ts{\frac{1}{\alpha_n^\mu\, m_n^\mu\, c}})^2} \right) 
	\left[ \ts{\frac{1}{\alpha_n^\mu\, m_n^\mu\, c}} + \frac{1}{\tan \left(\omega_n^\mu\, t + \theta_n^{\mu,c} \right)} \right],	
	\label{eq:ode-r-sol}
\end{gather}
for all $\mu\in(0,1]$, $t\in[0,\bar t^{\mu})$, $c\in\R_{>0}$,  where 
\begin{equation}
	\begin{aligned}
	\alpha_n^\mu & \doteq [ ( \ts{\frac{1}{m\, \kappa}})\, \lambda_n^\mu]^\half\,,
	&
	\lambda_n^\mu &\doteq \frac{\lambda_n}{1 + \mu^2\, \lambda_n}\,,
	&
	\bar t^\mu & \doteq \mu\, (\ts{\frac{2\, m}{\kappa}})^\half\,,
	\\
	\omega_n^\mu & \doteq [ (\ts{\frac{\kappa}{m}})\, \lambda_n^\mu]^\half\,,
	&
	\lambda_n & \doteq (\ts{\frac{n\, \pi}{L}})^2\,,
	& 
	\theta_n^{\mu,c} & \doteq \tan^{-1} (\ts{\frac{1}{\alpha_n^\mu\, m_n^\mu\, c}})\,.
	% \bar c & \doteq \frac{2\, \tan \sqrt{2}}{\alpha_n^\mu\, m_n^\mu}\,.
	\end{aligned}
	\label{eq:alpha-omega-lambda}
\end{equation}
Note by inspection that $\{\lambda_n\}$, $\{\lambda_n^\mu\}$, $\{\alpha_n^\mu\}$, $\{\omega_n^\mu\}$ define strictly increasing sequences in $n\in\N$. In particular,
\begin{align}
	\lambda_1^1
	& \le \lambda_1^\mu \le \lambda_n^\mu \le \lambda_\infty^\mu\doteq \ts{\frac{1}{\mu^2}}\,,
	\label{eq:alpha-omega-lambda-order}
\end{align}
with corresponding inequalities holding for $\alpha_n^\mu$, $\omega_n^\mu$. 

In order to establish that the Riesz-spectral operator-valued functions $\opbreve{P}^{\mu,c}$, $\opbreve{Q}^{\mu,c}$, $\opbreve{R}^{\mu,c}$ defined by \er{eq:op-P-def}--\er{eq:ode-r-sol} satisfy the respective operator-valued initial value problems \er{eq:op-P}--\er{eq:op-R}, it is important to first establish differentiability of these operator-valued functions, given a specific choice of initial condition operator $\op{M}_\mu$. This can be achieved by application of Lemma \ref{lem:op-F}. In particular, motivated by condition {\em (i)} of Lemma \ref{lem:op-F} (concerning strict monotonicity of sequences $\{p_n^{\mu,c}(t)\}$, $\{q_n^{\mu,c}(t)\}$, $\{r_n^{\mu,c}(t)\}$, $n\in\N$), it is convenient by inspection of \er{eq:ode-p-sol}--\er{eq:alpha-omega-lambda} to select $\op{M}_\mu$ to be a Riesz-spectral operator of the form \er{eq:op-Riesz} with eigenvalues $\{m_n^\mu\}_{n\in\N}$ satisfying $\frac{1}{\alpha_n^\mu\, m_n^\mu} = \sqrt{m\, \kappa} = \frac{\sqrt{\lambda_n^\mu}}{\alpha_n^\mu}$. That is,
\begin{align}
	\op{M}_\mu\, x
	& = \sum_{n=1}^\infty m_n^\mu \, \langle x,\, \tilde\varphi_n \rangle_\half\, \tilde\varphi_n\,,
	\quad
	m_n^\mu
	\doteq \frac{1}{\sqrt{\lambda_n^\mu}} = \left( \frac{1 + \mu^2\, \lambda_n}{\lambda_n} \right)^\half
	&& x\in\dom(\op{M}_\mu) = \cX_\half\,.
	\label{eq:eig-m}
\end{align}
(Note that $\op{M}_\mu\in\bo(\cX_\half)$ as $\{ \mu_n^\mu \}_{n\in\N}$ is bounded.)
The eigenvalues \er{eq:ode-p-sol}--\er{eq:ode-r-sol} subsequently simplify to
\begin{gather}
	p_n^{\mu,c}(t)
	= -\ts{\frac{1}{\alpha_n^\mu}} \, \frac{1}{\tan\left( \omega_n^\mu\, t + \tan^{-1} \left( \ts{\frac{\sqrt{m\, \kappa}}{c}} \right) \right)}\,,
	% \label{eq:eig-p}
	% \\
	\quad
	q_n^{\mu,c}(t)
	= + \ts{\frac{1}{\alpha_n^\mu}}\, \left( \frac{1}{1 + ( \frac{\sqrt{m\, \kappa}}{c})^2} \right)^\half\, 
	\frac{1}{\sin\left( \omega_n^\mu\, t + \tan^{-1} \left( \ts{\frac{\sqrt{m\, \kappa}}{c}} \right) \right)}\,,
	% \label{eq:eig-q}
	\label{eq:eig-p}
	\\
	r_n^{\mu,c}(t)
	= - \ts{\frac{1}{\alpha_n^\mu}} \left( \frac{1}{1 + (\ts{\frac{\sqrt{m\, \kappa}}{c}})^2} \right) 
	\left[ \ts{\frac{\sqrt{m\, \kappa}}{c}} + \frac{1}{\tan \left(\omega_n^\mu\, t + \tan^{-1} \left( \ts{\frac{\sqrt{m\, \kappa}}{c}} \right) \right)} \right].	
	\label{eq:eig-r}
\end{gather}
For convenience, define
\begin{align}
	\bar c
	& \doteq \sqrt{m\, \kappa}\, \tan \sqrt{2}\,.
	\label{eq:c-bar}
\end{align}
\begin{lemma}
\label{lem:explicit}
Given $\mu\in(0,1]$, $c\in(\bar c,\infty)$, $\bar c\in\R_{>0}$ as per \er{eq:c-bar}, and $\op{M}_\mu\in\bo(\cX_\half)$ as per \er{eq:eig-m}, the Riesz-spectral operator-valued functions $\opbreve{P}^{\mu,c},\, \opbreve{Q}^{\mu,c},\, \opbreve{R}^{\mu,c}$ of the form \er{eq:op-P-def} and defined by the respective eigenvalues \er{eq:eig-p}--\er{eq:eig-r}
%
% \er{eq:ode-p-sol}--\er{eq:ode-r-sol}, \er{eq:eig-m},
%
satisfy $\opbreve{P}^{\mu,c}(t),\, \opbreve{Q}^{\mu,c}(t),\, \opbreve{R}^{\mu,c}(t)\in\bo(\cX_\half)$ for every $t\in[0,\bar t^\mu)$, while $\opbreve{P}^{\mu,c},\, \opbreve{Q}^{\mu,c},\, \opbreve{R}^{\mu,c}:[0,\bar t^\mu)\mapsinto \bo(\cX_\half)$ are {\Frechet} differentiable and satisfy the initial value problems \er{eq:op-P}--\er{eq:op-R}.
\end{lemma}
\begin{proof}
The proof proceeds by demonstrating that the conditions of Lemma \ref{lem:op-F} hold for each of the Riesz-spectral operator-valued functions $\opbreve{P}^{\mu,c}$, $\opbreve{Q}^{\mu,c}$, $\opbreve{R}^{\mu,c}$, thereby demonstrating their {\Frechet} differentiability. Satisfaction of the initial value problems \er{eq:op-P}--\er{eq:op-R} then follows by inspection.

In order verify that condition {\em (i)} of Lemma \ref{lem:op-F} holds for each of the Riesz-spectral operator-valued functions $\opbreve{P}^{\mu,c}$, $\opbreve{Q}^{\mu,c}$, $\opbreve{R}^{\mu,c}$, strict monotonicity of their respective eigenvalues must be demonstrated. To this end, by inspection of the eigenvalues \er{eq:eig-m} of $\op{M}_\mu$, it is straightforward to show via \er{eq:alpha-omega-lambda} that
\begin{align}
	\omega_n^\mu\, t\in[0,\sqrt{2})\,, \quad
	\omega_n^\mu\, t + \theta_n^{\mu,c} \in \left( \tan^{-1} \left( \ts{\frac{\sqrt{m\, \kappa}}{c}} \right),\,
			\sqrt{2} + \tan^{-1} \left( \ts{\frac{\sqrt{m\, \kappa}}{\bar c}} \right) \right)
		\subset (0,\ts{\frac{\pi}{2}} )
	\label{eq:angle-argument-ok}
\end{align}
for all $\mu\in(0,1]$, $t\in[0,\bar t^\mu)$, $c\in(\bar c,\infty)$, where $\bar c\in\R_{>0}$ is as per \er{eq:c-bar}. Hence, by inspection of \er{eq:eig-p}--\er{eq:eig-r}, the eigenvalues of operators $\opbreve{P}^{\mu,c}(t)$, $\opbreve{Q}^{\mu,c}(t)$, $\opbreve{Q}^{\mu,c}(t)$ are well-defined for all $\mu\in(0,1]$, $t\in[0,\bar t^\mu)$, $c\in(\bar c,\infty)$, and $n\in\N$. Furthermore, strict monotonicity of the sequences $\{\alpha_n^\mu\}$ and $\{\omega_n^\mu\}$ in $n\in\N$, and strict monotonicity of the trigonometric functions $\tan$ and $\sin$ on $[0,\ts{\frac{\pi}{2}})$, implies that the sequences $\{p_n^{\mu,c}(t)\}$, $\{q_n^{\mu,c}(t)\}$, $\{r_n^{\mu,c}(t)\}$ are strictly monotone in $n\in\N$ for all $\mu\in(0,1]$, $t\in[0,\bar t^\mu)$, $c\in(\bar c,\infty)$. It is straightforward to see that the respective closures of these sets of eigenvalues are totally disconnected. Hence, condition {\em (i)} of Lemma \ref{lem:op-F} holds for each of the Riesz-spectral operator-valued functions $\opbreve{P}^{\mu,c}$, $\opbreve{Q}^{\mu,c}$, $\opbreve{R}^{\mu,c}$.

In order verify that condition {\em (ii)} of Lemma \ref{lem:op-F} holds for each of the Riesz-spectral operator-valued functions $\opbreve{P}^{\mu,c}$, $\opbreve{Q}^{\mu,c}$, $\opbreve{R}^{\mu,c}$, first recall that by \er{eq:angle-argument-ok} that the functions $p_n^{\mu,c}$, $q_n^{\mu,c}$, $r_n^{\mu,c}$ of \er{eq:eig-p}--\er{eq:eig-r} are continuous on $[0,\bar t^\mu)$, for every $\mu\in(0,1]$, $c\in(\bar c,\infty)$, and $n\in\N$. These functions are (twice) differentiable, with
\begin{align}
	\dot p_n^{\mu,c}(t)
	& = \kappa + \ts{\frac{1}{m}}\, \lambda_n^\mu\, (p_n^{\mu,c}(t))^2\,,
	& \ddot p_n^{\mu,c}(t) & = \ts{\frac{2}{m}} \, \lambda_n^\mu\, p_n^{\mu,c}(t)\, \dot p_n^{\mu,c}(t)\,,
	& p_n^{\mu,c}(0) & = -c\, m_n^\mu\,,
	\label{eq:dot-p}
	\\
	\dot q_n^{\mu,c}(t)
	& = \ts{\frac{1}{m}} \, \lambda_n^\mu\, p_n^{\mu,c}(t) \, q_n^{\mu,c}(t)\,,
	& \ddot q_n^{\mu,c}(t) 
			& = \ts{\frac{1}{m}}\, \lambda_n^\mu\, \left( p_n^{\mu,c}(t) \, \dot q_n^{\mu,c}(t) + \dot p_n^{\mu,c}(t) \, q_n^{\mu,c}(t) \right)\,,
	& q_n^{\mu,c}(0) & = + c\, m_n^\mu\,,
	\label{eq:dot-q}
	\\
	\dot r_n^{\mu,c}(t)
	& = \ts{\frac{1}{m}}\, \lambda_n^\mu\, (q_n^{\mu,c}(t))^2\,,
	& \ddot r_n^{\mu,c}(t) & = \ts{\frac{2}{m}}\, \lambda_n^\mu\, q_n^{\mu,c}(t) \, \dot q_n^{\mu,c}(t)\,,
	& r_n^{\mu,c}(0) & = -c\, m_n^\mu\,,
	\label{eq:dot-r}
\end{align}
for all $\mu\in(0,1]$, $t\in[0,\bar t^\mu)$, $c\in(\bar c,\infty)$, and $n\in\N$. Hence, the first and second derivatives \er{eq:dot-p}--\er{eq:dot-r} must also be continuous by inspection. That is, condition {\em (ii)} of Lemma \ref{lem:op-F} holds for each of the Riesz-spectral operator-valued functions $\opbreve{P}^{\mu,c}$, $\opbreve{Q}^{\mu,c}$, $\opbreve{R}^{\mu,c}$.

In order verify that condition {\em (iii)} of Lemma \ref{lem:op-F} holds for each of the Riesz-spectral operator-valued functions $\opbreve{P}^{\mu,c}$, $\opbreve{Q}^{\mu,c}$, $\opbreve{R}^{\mu,c}$, note by inspection of \er{eq:alpha-omega-lambda-order} and \er{eq:eig-p}--\er{eq:dot-r} that
\begin{gather}
	|p_n^{\mu,c}(t)|
	\le \ts{\frac{1}{\alpha_1^\mu}} \, \frac{c}{\sqrt{m\, \kappa}} \doteq M_p^{\mu,c}<\infty\,,
	\quad
	|\dot p_n^{\mu,c}(t)|
	\le \kappa + \ts{\frac{1}{m}} \, \lambda_\infty^\mu\, (M_p^{\mu,c})^2 \doteq M_{\dot p}^{\mu,c}<\infty\,,
	\nn\\
	|\ddot p_n^{\mu,c}(t)|
	\le \ts{\frac{2}{m}}\, \lambda_\infty^\mu\, M_p^{\mu,c}\, M_{\dot p}^{\mu,c} \doteq M_{\ddot p}^{\mu,c}<\infty\,,
	\nn
\end{gather}
for all $\mu\in(0,1]$, $t\in[0,\bar t^\mu)$, $c\in(\bar c,\infty)$, and uniformly in $n\in\N$, with analogous bounds holding for $q_n^{\mu,c}(t)$ and $r_n^{\mu,c}(t)$, and their first and second derivatives. That is, condition {\em (iii)} of Lemma \ref{lem:op-F} holds for each of the Riesz-spectral operator-valued functions $\opbreve{P}^{\mu,c}$, $\opbreve{Q}^{\mu,c}$, $\opbreve{R}^{\mu,c}$.

In summary, Lemma \ref{lem:op-F} thus implies that the Riesz-spectral operator-valued functions $\opbreve{P}^{\mu,c},\, \opbreve{Q}^{\mu,c},\,\opbreve{R}^{\mu,c}:[0,\bar t^\mu)\mapsinto\bo(\cX_\half)$ of the form \er{eq:op-P-def} and defined by the eigenvalues \er{eq:eig-p}--\er{eq:eig-r} are {\Frechet} differentiable. Furthermore, their {\Frechet} derivatives are also Riesz-spectral operators, and take the form \er{eq:op-F-dot}. Hence, combining \er{eq:op-F-dot} and \er{eq:dot-p}, and recalling Lemma \ref{lem:compose-Riesz},
\begin{align}
	\dot{\opbreve{P}}^{\mu,c}(t)\, x
	& = \sum_{n=1}^\infty \dot p_n^{\mu,c}(t) \, \langle x,\, \tilde\varphi_n\rangle_\half\, \tilde\varphi_n
	= \sum_{n=1}^\infty \left( \kappa + \ts{\frac{1}{m}}\, \left( \frac{\lambda_n}{1 + \mu^2\, \lambda_n} \right)  (p_n^{\mu,c}(t))^2 \right)
											\langle x,\, \tilde\varphi_n\rangle_\half\, \tilde\varphi_n
	\nn\\
	& = \kappa\, \op{I}\, x + \ts{\frac{1}{m}}\, \sum _{n=1}^\infty \left( 
						p_n^{\mu,c}(t)\, \sqrt{\lambda_n} \, (1 + \mu^2\, \lambda_n)^{-1}\, \sqrt{\lambda_n} \, p_n^{\mu,c}(t)
					\right) \langle x,\, \tilde\varphi_n\rangle_\half\, \tilde\varphi_n
	\nn\\
	& = \left( \kappa\, \op{I} + \ts{\frac{1}{m}}\, \opbreve{P}^{\mu,c}(t)\, \opAsqrt\, \op{I}_\mu\, \opAsqrt\, \opbreve{P}^{\mu,c}(t) \right) x\,,
	\label{eq:opdot-P-computation}
\end{align}
for all $\mu\in(0,1]$, $t\in(0,\bar t^\mu)$, $c\in(\bar c,\infty)$, $x\in\cX_\half$. Recalling the definition \er{eq:eig-m} of the eigenvalues of $\op{M}^\mu$,
\begin{align}
	\opbreve{P}^{\mu,c}(0)\, x
	& = \sum_{n=1}^\infty p_n^{\mu,c}(0) \, \langle x,\, \tilde\varphi_n\rangle_\half\, \tilde\varphi_n
	= \sum_{n=1}^\infty -c\, m_n^\mu \, \langle x,\, \tilde\varphi_n\rangle_\half\, \tilde\varphi_n
	= -c \, \op{M}_\mu\, x\,,
	\label{eq:op-P-zero}
\end{align}
for all $\mu\in(0,1]$, $c\in(\bar c,\infty)$, $x\in\cX_\half$. That is, \er{eq:opdot-P-computation} and \er{eq:op-P-zero} imply that $\opbreve{P}^{\mu,c}$ satisfies the initial value problem \er{eq:op-P}. Analogous calculations similarly imply that $\opbreve{Q}^{\mu,c}$ and $\opbreve{R}^{\mu,c}$ satisfy \er{eq:op-Q} and \er{eq:op-R} respectively. 
\end{proof}

Given the role of the eigenvalues \er{eq:eig-m} of the operator $\op{M}_\mu$ in the definition of operators $\opbreve{P}^{\mu,c}$, $\opbreve{Q}^{\mu,c}$, $\opbreve{R}^{\mu,c}$, it is convenient to construct a closed-form for $\op{M}_\mu$, and subsequently $\op{K}_\mu$ of \er{eq:op-M-and-K}. 

\begin{lemma}
\label{lem:explicit-op-M-and-K}
$\op{M}_\mu$, $\op{K}_\mu$ of \er{eq:op-M-and-K}, \er{eq:eig-m} are bounded, self-adjoint, positive, and boundedly invertible, with
\begin{align}
	\op{M}_\mu\, x
	& = (\op{A}^{-1} + \mu^2\, \op{I})^\half\, x\,,
	&& x\in\dom(\op{M}_\mu) = \cX_\half\,,
	\label{eq:explicit-op-M}
	\\
	\op{K}_\mu\, x
	& = \op{M}_\mu^\half\, x\,,
	&& x\in\dom(\op{K}_\mu) = \cX_\half\,.
	\label{eq:explicit-op-K}
\end{align}
\end{lemma}
\begin{proof}
Recall that $\op{A}$ of \er{eq:op-A} is a Riesz-spectral operator of the form \er{eq:op-Riesz}, with eigenvalues $\{\lambda_n\}_{n\in\N}$ (see Lemmas \ref{lem:eigenvalues} and \ref{lem:specific-op-Riesz}). Consequently, noting the form \er{eq:identity-Riesz} of the identity $\op{I}$, it follows that $\op{I}_\mu^{-1} = \op{I} + \mu^2\, \op{A}$ is also a Riesz-spectral operator of the same form \er{eq:op-Riesz}, defined on $\cX_0$ via \er{eq:op-I-mu}, with eigenvalues $\{1 + \mu^2\, \lambda_n \}_{n\in\N}$. Furthermore, as $\op{I}_\mu^{-1}$ is self-adjoint and positive (by Lemma \ref{lem:op-A-properties}), it also has a unique self-adjoint and positive square root $\op{I}_\mu^{-\half}:\cX_0\mapsinto\cX_\half$, which is also a Riesz-spectral operator of the form \er{eq:op-Riesz} by Corollary \ref{cor:unique-inverse-Riesz}. Similarly, $\opAsqrt$ and hence $\op{J}\doteq (\opAsqrt)^{-1}$ are Riesz-spectral operators of the same form (see Lemma \ref{lem:specific-op-Riesz}). In particular,
\begin{align}
	\op{I}_\mu^{-\half}\, x
	& = \sum_{n=1}^\infty (1 + \mu^2\, \lambda_n)^\half \, \langle x,\, \tilde\varphi_n\rangle_\half\, \tilde\varphi_n\,,
	&& x\in\dom(\op{I}_\mu^{-\half}) = \cX_0\,, \ \ran(\op{I}_\mu^{-\half}) = \cX_\half\,,
	\label{eq:op-inv-I-mu-Riesz}
	\\
	\op{J}\, x
	& = \sum_{n=1}^\infty \frac{1}{(\lambda_n)^\frac{1}{2}} \, \langle x,\, \tilde\varphi_n\rangle_\half\, \tilde\varphi_n\,,
	&& x\in\dom(\op{J}) = \cX_\half\,, \ \ran(\op{J}) = \cX_0\,.
	\label{eq:op-J-Riesz}
\end{align}
Applying Lemma \ref{lem:compose-Riesz}, the composition $\ophat{M}_\mu \doteq \op{I}_\mu^{-\half}\, \op{J}:\cX_\half\mapsinto\cX_\half$ is also a Riesz-spectral operator, with
\begin{align}
	\ophat{M}_\mu\, x
	& \doteq \op{I}_\mu^{-\half}\, \op{J} \, x
	= \sum_{n=1}^\infty \left( \frac{1 + \mu^2\, \lambda_n}{\lambda_n} \right)^\half \, \langle x,\, \tilde\varphi_n \rangle_\half\, \tilde\varphi_n
	= \sum_{n=1}^\infty m_n^\mu \, \langle x,\, \tilde\varphi_n \rangle_\half\, \tilde\varphi_n
	= \op{M}_\mu\, x\,,
	&& x\in\dom(\ophat{M}_\mu) = \cX_\half\,,
	\label{eq:op-M-hat-Riesz}
\end{align}
where the third and fourth equalities follow by definition \er{eq:eig-m} of the eigenvalues $\{m_n^\mu\}_{n\in\N}$ of $\op{M}_\mu$. Furthermore, again applying Lemma \ref{lem:compose-Riesz}, and the fact that $\op{J}$ and $\op{A}$ are Riesz-spectral operators,
\begin{align}
	\op{M}_\mu^2 \, x = \op{M}_\mu\, \op{M}_\mu \, x
	& = \sum_{n=1}^\infty \left( \frac{1+\mu^2\, \lambda_n}{\lambda_n}\right) \langle x,\, \tilde\varphi_n \rangle_\half\, \tilde\varphi_n
	= \sum_{n=1}^\infty \left( \frac{1}{\sqrt{\lambda_n}} \, ( 1 + \mu^2\, \lambda_n) \, \frac{1}{\sqrt{\lambda_n}} \right)
			\langle x,\, \tilde\varphi_n \rangle_\half\, \tilde\varphi_n
	\nn\\
	& = \op{J}\, (\op{I} + \mu^2 \, \op{A}) \, \op{J}\, x = (\op{A}^{-1} + \mu^2\, \op{I})\, x\,,
	\qquad\qquad x\in\dom(\op{M}_\mu^2) = \cX_\half\,.
	\label{eq:ophat-M-squared}
\end{align}
(Note that this equivalently follows from \er{eq:op-M-hat-Riesz} via commutation of $\op{I}_\mu^{-\half}$ and $\op{J}$ in $\op{M}_\mu^2 = \ophat{M}_\mu^2 = \op{I}_\mu^{-\half}\, \op{J}\, \op{I}_\mu^{-\half}\, \op{J}$.)
Applying Lemma \ref{lem:op-A-properties}, $\op{A}^{-1} + \mu^2\, \op{I}$ is bounded, self-adjoint, and positive, and so has a unique, bounded, self-adjoint, and positive square root defined on $\cX_\half$. That is, $\op{M}_\mu$ is equivalently defined by \er{eq:explicit-op-M}, and it is bounded, self-adjoint, and positive. Consequently, a unique $\op{K}_\mu \doteq \op{M}_\mu^\half\in\bo(\cX_\half)$ exists as per \er{eq:op-M-and-K} and \er{eq:explicit-op-K}, with the additional properties that it is also self-adjoint and positive. 

It remains to be shown that $\op{M}_\mu$ and $\op{K}_\mu$ are boundedly invertible. To this end, note that $\op{M}_\mu^{2} = \op{I}_\mu^{-1}\, \op{A}^{-1}$ by commuting the left-hand $\op{J}$ with $\op{I} + \mu^2\, \op{A}$ in the fourth equality of \er{eq:ophat-M-squared}. Hence, $\op{M}_\mu^2$ is boundedly invertible, as $\op{M}_\mu^{-2} = \op{A}\, \op{I}_\mu = \ts{\frac{1}{\mu^2}}\, (\op{I} - \op{I}_\mu) \in \bo(\cX_\half)$ by Lemma \ref{lem:op-I-mu-properties}. Also, as $\op{M}_\mu^2$ is positive and self-adjoint, so is $\op{M}_\mu^{-2}$. Consequently, $\op{M}_\mu^{-2}$ has a unique, bounded, positive, and self-adjoint square-root and fourth root, namely $\op{M}_\mu^{-1}$ and $\op{M}_\mu^{-\half} = \op{K}_\mu^{-1}$. That is, $\op{M}_\mu$ and $\op{K}_\mu$ are boundedly invertible as required.
\end{proof}

With Lemma \ref{lem:explicit-op-M-and-K} in place, $\op{K}_\mu$ of \er{eq:explicit-op-M} satisfies the properties required by definition \er{eq:psi} and the proof of Lemma \ref{lem:eps-ball}. Consequently, an explicit form for the fundamental solution \er{eq:W-infty} may be established.

%% if falsed old $\op{K}_\mu$ argument %%
\if{false}

These observations combined with \er{eq:eig-m} and Lemma \ref{lem:compose-Riesz} motivate the definition of an operator $\ophat{M}_\mu$ by 
\begin{align}
	\ophat{M}_\mu\, x
	& \doteq 
	% (\op{A}^{-1} + \mu^2\, \op{I})\, x = \op{J}\, (\op{I} + \mu^2\, \op{A})\, \op{J}\, x =  
	\op{I}_\mu^{-\half}\, \op{J}\, \op{I}_\mu^{-\half}\, \op{J}\, x = (\op{I}_\mu^{-\half}\, \op{J})^2\, x\,,
	&& x\in\dom(\ophat{M}_\mu) \doteq \cX_\half\,.
	\label{eq:op-M-hat}
\end{align}
Note further that $\op{I}_\mu^{-\half} = (\op{I}+\mu^2\, \op{A})^\half$ and $\op{J} = \op{A}^{-\half}$ commute on $\cX_\half$, so that
\begin{align}
	\ophat{M}_\mu\, x
	& = \op{I}_\mu^{-\half} \, \op{I}_\mu^{-\half} \, \op{J}\, \op{J}\, x % = \op{I}_\mu^{-1}\, \op{J}^2\, x 
	= (\op{I} + \mu^2\, \op{A})\, \op{A}^{-1}\, x
	= (\op{A}^{-1} + \mu^2\, \op{I})\, x\,,
	&& x\in\cX_\half\,.
	\label{eq:op-M-hat-commute}
\end{align}
Hence, $\ophat{M}_\mu$ is bounded, self-adjoint, and positive by Lemma \ref{lem:op-A-properties}. 
Consequently, it has a unique, bounded, self-adjoint, and positive square root, defined on $\cX_\half$ as per \er{eq:op-M-hat} by $\ophat{M}_\mu^\half = \op{I}_\mu^{-\half}\, \op{J}$.
Combining \er{eq:op-inv-I-mu-Riesz}, \er{eq:op-J-Riesz}, \er{eq:op-M-hat} via Lemma \ref{lem:compose-Riesz}, $\ophat{M}_\mu$ and $\ophat{M}_\mu^\half$ are both Riesz-spectral operators of the form \er{eq:op-Riesz} on $\cX_\half$, with
\begin{align}
	% \ophat{M}_\mu\, x
	% & = \sum_{n=1}^\infty \left( \frac{1 + \mu^2\, \lambda_n}{\lambda_n} \right) \langle x,\, \tilde\varphi_n\rangle_\half\, \tilde\varphi_n\,,
	% && \dom(\ophat{M}_\mu) = \cX_\half\,,
	% \nn\\
	\ophat{M}_\mu^\half\, x
	& = \sum_{n=1}^\infty \left( \frac{1 + \mu^2\, \lambda_n}{\lambda_n} \right)^\half \langle x,\, \tilde\varphi_n\rangle_\half\, \tilde\varphi_n\,,
	&& \dom(\ophat{M}_\mu^\half) = \cX_\half\,.
	\label{eq:op-M-hat-sqrt}
\end{align}
(The domains of both operators may be checked via Lemma \ref{lem:dom-Riesz}.) By inspection of \er{eq:op-M-hat-commute}, $\ophat{M}_\mu$ is invertible, with 
% $\ophat{M}_\mu^{-1} = \opAsqrt\, \op{I}_\mu\, \opAsqrt$ and $\dom(\ophat{M}_\mu^{-1}) = \cX_\half$.  Applying Lemma \ref{lem:op-I-mu-properties}, 
$\ophat{M}_\mu^{-1} \equiv \op{A}\, \op{I}_\mu = \frac{1}{\mu^2} \, (\op{I} - \op{I}_\mu) \in \bo(\cX_\half)$. Hence, $\ophat{M}_\mu^\half$ is also boundedly invertible. So, by comparison of \er{eq:eig-m} and \er{eq:op-M-hat-sqrt}, an explicit closed form for $\op{M}_\mu$ of \er{eq:op-P}, \er{eq:op-Q}, \er{eq:op-R} follows immediately as
\begin{align}
	\op{M}_\mu\, x
	& = \ophat{M}_\mu^\half\, x = (\op{A}^{-1} + \mu^2\, \op{I})^\half\, x\,,
	&& x\in\dom(\op{M}_\mu) \doteq \cX_\half\,,
	\label{eq:op-M-mu}
\end{align}
which is bounded, self-adjoint, positive, and boundedly invertible. Consequently, $\op{K}_\mu \doteq \op{M}_\mu^\half\in\bo(\cX_\half)$ exists and may be selected in payoff \er{eq:psi} and associated definitions. 

\fi
%% end of if falsed old $\op{K}_\mu$ argument %%

% The following theorem is immediate.

\begin{theorem}
\label{thm:Wbreve-ok}
With $\op{K}_\mu$ as per \er{eq:explicit-op-K} in \er{eq:psi}, $\mu\in(0,1]$, and $c\in(\bar c,\infty)$, $\bar c\in\R_{>0}$ as per \er{eq:c-bar}, the value functional $W^{\mu,c}$ of \er{eq:W-c} takes the explicit form of $\Wbreve^{\mu,c}$ of \er{eq:W-explicit} with the operator-valued functions $\op{P}^{\mu,c}$, $\op{Q}^{\mu,c}$, $\op{R}^{\mu,c}$ given by the Riesz-spectral operator-valued functions $\opbreve{P}^{\mu,c},\, \opbreve{Q}^{\mu,c},\, \opbreve{R}^{\mu,c}:[0,\bar t^\mu)\mapsinto\bo(\cX_\half)$ of the form \er{eq:op-P-def} with respective eigenvalues defined by \er{eq:ode-p-sol}--\er{eq:ode-r-sol}. Furthermore, the value functional $W^{\mu,\infty}$ of \er{eq:W-infty} defining the fundamental solution of the approximating optimal control problem \er{eq:W} via \er{eq:W-from-fund-conjecture} is given by 
\begin{align}
	W^{\mu,\infty}(t,x,z) 
	& \doteq \lim_{c\rightarrow\infty} \Wbreve^{\mu,c}(t,x,z)
	=  \demi\, \langle x,\, \opbreve{P}^{\mu,\infty}(t)\, x\rangle_\half + \langle x,\, \opbreve{Q}^{\mu,\infty}(t)\, z \rangle_\half
						+ \demi\, \langle z,\, \opbreve{R}^{\mu,\infty}(t)\, z \rangle_\half\,,
	\label{eq:explicit-W-infty}
\end{align}
for all $t\in(\delta,\, \bar t^\mu)$, $x,z\in\cX_\half$, given any $\delta\in(0,\bar t^\mu)$, with $\opbreve{P}^{\mu,\infty}, \, \opbreve{Q}^{\mu,\infty},\, \opbreve{R}^{\mu,\infty}:(\delta,\bar t^\mu)\mapsinto\bo(\cX_\half)$ defined by
\begin{align}
	\opbreve{P}^{\mu,\infty}(t)\, x
	& \doteq \sum_{n=1}^\infty p_n^{\mu,\infty}(t)\, \langle x,\, \tilde\varphi_n\rangle_\half\, \tilde\varphi_n\,,
	&& x\in\dom(\opbreve{P}^{\mu,\infty}(t)) = \cX_\half\,,
	\label{eq:op-P-infty}
	\\
	\opbreve{Q}^{\mu,\infty}(t)\, z
	& \doteq \sum_{n=1}^\infty q_n^{\mu,\infty}(t)\, \langle z,\, \tilde\varphi_n\rangle_\half\, \tilde\varphi_n\,,
	&& z\in\dom(\opbreve{Q}^{\mu,\infty}(t)) = \cX_\half\,,
	\label{eq:op-Q-infty}
	\\
	\opbreve{R}^{\mu,\infty}(t)\, z
	& \doteq \sum_{n=1}^\infty r_n^{\mu,\infty}(t)\, \langle z,\, \tilde\varphi_n\rangle_\half\, \tilde\varphi_n\,,
	&& z\in\dom(\opbreve{R}^{\mu,\infty}(t)) = \cX_\half\,,
	\label{eq:op-R-infty}
\end{align}
and
\begin{align}
	p_n^{\mu,\infty}(t)
	& \doteq -\ts{\frac{1}{\alpha_n^\mu}} \, \frac{1}{\tan(\omega_n^\mu\, t)}\,,
	\quad
	q_n^{\mu,\infty}(t)
	\doteq +\ts{\frac{1}{\alpha_n^\mu}} \, \frac{1}{\sin(\omega_n^\mu\, t)}\,,
	\quad
	r_n^{\mu,\infty}(t)
	\doteq -\ts{\frac{1}{\alpha_n^\mu}} \, \frac{1}{\tan(\omega_n^\mu\, t)}\,.
	\label{eq:eig-pqr-infty}
\end{align}
\end{theorem}
\begin{proof}
The first assertion concerning the explicit form of $W^{\mu,c}$ follows by Theorem \ref{thm:explicit}, Lemma \ref{lem:explicit}, and the specified functional $\Wbreve^{\mu,c}$. In order to prove the second assertion concerning the explicit form \er{eq:explicit-W-infty} of the limit value function $W^{\mu,\infty}$ of \er{eq:W-infty}, the operator-valued functions $\opbreve{P}^{\mu,c},\, \opbreve{Q}^{\mu,c},\, \opbreve{R}^{\mu,c}:[0,\bar t^\mu)\mapsinto\bo(\cX_\half)$ must be shown to converge (either strongly or uniformly) to their respective candidate limits defined by \er{eq:op-P-infty}--\er{eq:op-R-infty}, whereupon Theorem \ref{thm:limit} can be used to complete the proof. To this end, fix $\mu\in(0,1]$ and $\delta\in(0,\bar t^\mu)$, and note that the eigenvalues of $\opbreve{P}^{\mu,c}(t),\, \opbreve{Q}^{\mu,c}(t),\, \opbreve{R}^{\mu,c}(t)\in\bo(\cX_\half)$, given by  \er{eq:eig-p}--\er{eq:eig-r}, satisfy (after straightforward calculation of the respective Taylor series expansions with respect to $1/c$)
\begin{align}
	|p_n^{\mu,c}(t) - p_n^{\mu,\infty}(t)|
	& \le \ts{\frac{1}{c}}\, \Delta_p^{\mu}(\delta)\,,
	&& % \Delta_p^\mu(\delta) \doteq \frac{\sqrt{m\, \kappa}}{\alpha_1^\mu} \, \frac{1}{\tan^2(\omega_1^\mu\, \delta)}\,,
	\Delta_p^\mu(\delta) \doteq \frac{\sqrt{m\, \kappa}}{\alpha_1^\mu} \, \frac{1}{\sin^2(\omega_1^\mu\, \delta)}\,,
	\nn\\
	|q_n^{\mu,c}(t) - q_n^{\mu,\infty}(t)|
	& \le \ts{\frac{1}{c}}\, \Delta_q^{\mu}(\delta) \,,
	&& \Delta_q^{\mu}(\delta) \doteq
	% \frac{\sqrt{m\, \kappa}}{\alpha_1^\mu} \, \frac{1}{\sin(\omega_1^\mu\, \delta)} \left[
	%	\frac{\sqrt{m\, \kappa}}{\bar c} +  \frac{1}{\tan(\omega_1^\mu\, \delta)}
	\frac{\sqrt{m\, \kappa}}{\alpha_1^\mu} \, \frac{1}{\sin(\omega_1^\mu\, \delta)} \left[
		\frac{\sqrt{m\, \kappa}}{\bar c} +  \frac{1}{\tan(\omega_1^\mu\, \delta)}
	\right],
	\nn\\
	|r_n^{\mu,c}(t) - r_n^{\mu,\infty}(t)|
	& \le \ts{\frac{1}{c}}\, \Delta_r^{\mu}(\delta) \,,
	&& \Delta_r^\mu(\delta) \doteq
	% \frac{\sqrt{m\, \kappa}}{\alpha_1^\mu} \left[ 1 + \frac{1}{\tan^2(\omega_1^\mu\, \delta)} \right]
	% + \frac{2\, m\, \kappa}{\alpha_1^\mu\, \bar c} \left[ \frac{\sqrt{m\, \kappa}}{\bar c} + \frac{1}{\tan(\omega_1^\mu\, \delta)} \right]
	\frac{\sqrt{m\, \kappa}}{\alpha_1^\mu} \left[ 1 + \frac{1}{\tan(\omega_1^\mu\, \delta)} + \frac{1}{\sin^2(\omega_1^\mu\, \delta)} \right]
	\nn
\end{align}
for all $n\in\N$, $t\in(\delta,\bar t^\mu)$, $c\in(\bar c,\infty)$. Hence, \er{eq:op-P-def}, \er{eq:eig-p}--\er{eq:eig-r}, \er{eq:op-P-infty}--\er{eq:op-R-infty}, \er{eq:eig-pqr-infty}, and \er{eq:identity-Riesz} imply that
\begin{align}
	% \| [\opbreve{P}^{\mu,c} - \opbreve{P}^{\mu,\infty}]\, x \|_\half^2
	% & \le \sum_{n=1}^\infty | p_n^{\mu,c}(t) - p_n^{\mu,\infty}(t) | \, \langle x,\, \tilde\varphi_n \rangle_\half \, \tilde\varphi_n
	% \le \ts{\frac{1}{c^2}}\, (\Delta_p^{\mu}(\delta))^2\, \|x\|_\half^2
	% \nn\\
	\| \opbreve{P}^{\mu,c}(t) - \opbreve{P}^{\mu,\infty}(t) \|_\half
	& \le \ts{\frac{1}{c}}\, \Delta_p^{\mu}(\delta)\,,
	\quad
	\| \opbreve{Q}^{\mu,c}(t) - \opbreve{Q}^{\mu,\infty}(t) \|_\half
	\le \ts{\frac{1}{c}}\, \Delta_q^{\mu}(\delta)\,,
	\quad
	\| \opbreve{R}^{\mu,c}(t) - \opbreve{R}^{\mu,\infty}(t) \|_\half
	\le \ts{\frac{1}{c}}\, \Delta_r^{\mu}(\delta)\,,
	\nn
\end{align}
for all $\mu\in(0,1]$, $\delta\in(0,\bar t^\mu)$, $t\in(\delta,\bar t^\mu)$, and $c\in(\bar c,\infty)$, where $\|\cdot\|_\half$ denotes the induced operator norm in $\cX_\half$. Consequently, Lemma \ref{lem:explicit} and the triangle inequality imply that $\opbreve{P}^{\mu,\infty}, \, \opbreve{Q}^{\mu,\infty},\, \opbreve{R}^{\mu,\infty}:(\delta,\bar t^\mu)\mapsinto\bo(\cX_\half)$. Furthermore, the Riesz-spectral operator-valued functions $\opbreve{P}^{\mu,c},\, \opbreve{Q}^{\mu,c},\, \opbreve{R}^{\mu,c}:(\delta,\bar t^\mu)\mapsinto\bo(\cX_\half)$ converge uniformly to $\opbreve{P}^{\mu,\infty},\, \opbreve{Q}^{\mu,\infty},\, \opbreve{R}^{\mu,\infty}:(\delta,\bar t^\mu)\mapsinto\bo(\cX_\half)$ as $c\rightarrow\infty$.
\end{proof}

\begin{corollary}
\label{cor:optimal-input}
Under the conditions of Theorem \ref{thm:Wbreve-ok}, the state feedback characterization of the optimal input $w^*$ of \er{eq:w-star} corresponding to the fundamental solution $W^{\mu,\infty}(t,x,y)$ of \er{eq:W-infty} of the approximating optimal control problem \er{eq:W} is given by
\begin{align}
	w^*(s)
	& = k(s,\xi^*(s))\,,
	\qquad
	k(s,x)
	= \ts{\frac{1}{m}}\, \opAsqrt\, \op{I}_\mu\, \opAsqrt \left( \opbreve{P}^{\mu,\infty}(t-s) \, x + \opbreve{Q}^{\mu,\infty}(t-s)\, y \right)
	\label{eq:optimal}
\end{align}
for all $\mu\in(0,1]$, $\delta\in(0,\bar t^\mu)$, $t\in(\delta,\bar t^\mu)$, $s\in[0,t-\delta)$, and $x,y\in\cX_\half$, where $\xi^*$ is the corresponding optimal trajectory generated by the open-loop dynamics \er{eq:dynamics} in feedback with policy $k$ of \er{eq:optimal}.
\end{corollary}
\begin{proof}
Immediate by Theorems \ref{thm:limit}, \ref{thm:verify}, and \ref{thm:Wbreve-ok}.
\end{proof}

\subsection{Application of the fundamental solution to solve optimal control problem \er{eq:W}}
The fundamental solution \er{eq:W-infty}, \er{eq:explicit-W-infty} can be applied via \er{eq:What} and Theorems \ref{thm:reach} and \ref{thm:Wbreve-ok} to solve the approximating optimal control problem \er{eq:W} for any concave terminal payoff $\psi:\cX_\half\mapsinto\R\cup\{-\infty\}$ for which the associated value function $W^\mu$ is finite. In particular, given $t\in(0,\bar t^\mu)$, $x\in\cX_\half$, the optimal control $w^*\in\cW[0,t]$ that maximizes the payoff $J_{m,\psi}(t,x,\cdot)$ in \er{eq:W} is given by \er{eq:optimal} with $y=z^*$, where
\begin{align}
	z^* & \in \argmax_{\zeta\in\cX_\half} \left\{ W^{\mu,\infty}(t,x,\zeta) + \psi(\zeta) \right\}\,.
	\label{eq:z-star}
\end{align}
For the specific terminal payoff $\psi \doteq \psi^{\mu,\infty}(\cdot,z)$, $z\in\cX_\half$, given by \er{eq:psi}, $z^* = z$ by inspection of \er{eq:z-star}. In that case, the value functional $W^\mu(t,\cdot)$ of \er{eq:W} and fundamental solution $W^{\mu,\infty}(t,\cdot,z)$ of \er{eq:W-infty} coincide, as do their corresponding optimal inputs, see \er{eq:optimal}. Furthermore, by substituting the series representations \er{eq:op-P-infty}, \er{eq:op-Q-infty} for $\opbreve{P}^{\mu,\infty}(t)$, $\opbreve{Q}^{\mu,\infty}(t)$ in \er{eq:optimal}, a state feedback characterization of the optimal control $w^*$ is given by
\begin{align}
	w^*(s)
	% & = k(s,\xi^*(s)) 
	% = \ts{\frac{1}{m}} \, \opAsqrt \, \op{I}_\mu\, \opAsqrt\, \left( \opbreve{P}^{\mu,\infty}(t-s)\, \xi^*(s) + \opbreve{Q}^{\mu,\infty}(t-s)\, z^* \right)
	% \nn\\
	& = \ts{\frac{1}{m}} \sum_{n=1}^\infty \frac{\lambda_n}{1 + \mu^2\, \lambda_n} \left( 
			p_n^{\mu,\infty}(t-s)\, \langle \xi^*(s),\, \tilde\varphi_n \rangle_\half
			+ q_n^{\mu,\infty}(t-s)\, \langle z^*,\, \tilde\varphi_n \rangle_\half
	\right),
	\qquad
	\dot\xi^*(s) = w^*(s)\,,
	\label{eq:series-w-star}
\end{align}
for all $s\in[0,t-\delta)$, where $\delta\in(0,t)$ is as per Theorem \ref{thm:Wbreve-ok}, $\xi^*(0) = x\in\cX_\half$, and $z^* = z\in\cX_\half$.

Alternatively, with $\psi\doteq\psi_v$ as per \er{eq:stationary-payoff}, $v\in\cX_\half$, Theorem \ref{thm:Wbreve-ok} and \er{eq:z-star} imply that
\begin{align}
	z^*
	& = - \left( \opbreve{R}^{\mu,\infty}(t) \right)^{-1} \left[ \opbreve{Q}^{\mu,\infty}(t)' \, x + m \, \op{J}\, \op{J}\, v \right]
	= \sum_{n=1}^\infty \frac{1}{r_n^{\mu,\infty}(t)} \left[ q_n^{\mu,\infty}(t) \, \langle x,\, \tilde\varphi_n \rangle_\half 
				+ \frac{m}{\lambda_n} \, \langle v,\, \tilde\varphi_n \rangle_\half \right],
	\label{eq:z-star-terminal-velocity}
\end{align}
where the series representation follows by substitution of \er{eq:op-Q-infty}, \er{eq:op-R-infty} for the Riesz-spectral operators $\opbreve{Q}^{\mu,\infty}(t)$, $\opbreve{R}^{\mu,\infty}(t)$ respectively. (Note that existence of the inverse involved, and a representation for it, follows by Corollary \ref{cor:unique-inverse-Riesz}.) The optimal control $w^*$ is again given by \er{eq:series-w-star}, with $z^*\in\cX_\half$ given by \er{eq:z-star-terminal-velocity}.

% ?? 01Jan15 -- discuss the two choices of terminal payoff ??

% In the specific case where the terminal payoff is $\psi^{\mu,\infty}(\cdot,\bar z)$ or $\psi_v$ is used, see respectively \er{eq:psi} and \er{eq:stationary-payoff}, the corresponding choice of $z^*\in\cX_\half$ is
% 
% ?? HERE 1Jan15 ??
% 
Finally, it is important to note that the optimal input defined by \er{eq:optimal} and \er{eq:z-star} 
% that approximately solve the optimal control problem \er{eq:W} for some $\mu\in(0,1]$ on a horizon $t\in(0,\bar t^\mu)$, 
is not defined everywhere on the time interval $[0,t]$. In particular, by inspection of \er{eq:explicit-W-infty}, this input is not defined on a time interval $[t-\delta,t]$ containing the final time, where $\delta\in(0,t)$ is arbitrarily small. While this might appear to be a problematic limitation, it is the initial input $w^*(0)$ that is required for the approximate solution of TPBVPs such as \er{eq:TPBVP} via the approximating optimal control problem \er{eq:W}. 
% This will be demonstrated in the next section.

%%
%%		Approximate solution of the TPBVP \er{eq:TPBVP}.
%%

\section{Approximate solution of two-point boundary value problems}
\label{sec:approx-solution}
For sufficiently short time horizons, Theorem \ref{thm:second-difference} guarantees that stationarity of the action functional \er{eq:action-mu} is achieved as a maximum. In particular, for horizons $t\in[0,\bar t^\mu)$, $\bar t^\mu\in\R_{>0}$ as per \er{eq:t-bar-mu}, the value function $W^\mu(t,\cdot)$ of \er{eq:W} is finite, and the corresponding optimal trajectory defined by \er{eq:dynamics}, \er{eq:optimal}, and \er{eq:z-star} renders the action functional \er{eq:action-mu} stationary in the calculus of variations sense. However, as the action principle only requires stationarity of the action functional with respect to trajectories, concavity of the action functional \er{eq:action-mu} may be lost for longer horizons. This implies a loss of concavity of the associated payoff $J_{m,\psi}^\mu$ of \er{eq:payoff}, and hence an infinite corresponding value function \er{eq:W}. 
% In that case, the optimal trajectory defined by the associated optimal control problem, as specified by \er{eq:dynamics} and \er{eq:optimal}, no longer satisfies the stationary action principle. 
% That is, for longer time horizons, the formulated correspondence between optimal control and stationary action breaks down. 
In that case, the stationary action trajectory is no longer the optimal trajectory defined by \er{eq:dynamics}, \er{eq:optimal}, and \er{eq:z-star}, so that more analysis is required. Below, the short horizon case is discussed first, i.e. where the stationary and maximal action coincide. An indication of an extension to longer horizons is provided subsequently. 

\subsection{Short horizons}
\label{sec:short-horizon}
On shorter time horizons, i.e. those satisfying $t\in[0,\bar t^\mu)$, 
%
% Theorem \ref{thm:Wbreve-ok} provides an explicit solution for the HJB PDE \er{eq:verify-DPE}, and so defines 
%
the optimal trajectory defined by \er{eq:dynamics}, \er{eq:optimal}, and \er{eq:z-star}
%
%
%
% the optimal trajectory that renders the approximate action functional \er{eq:action-mu} stationary. It also implies that the value functional 
% $W^{\mu,\infty}$ is a fundamental solution for the approximating optimal control defined by $W^\mu$ of \er{eq:W}, in the sense of 
% Theorem \ref{thm:reach}. The aforementioned optimal trajectory 
is described by the characteristic equations corresponding to the Hamiltonian $H$ of \er{eq:verify-H} for HJB \er{eq:verify-DPE}. These characteristic equations together define the abstract Cauchy problem 
\begin{align}
	\left( \ba{c} \dot\xi(s) \\ \dot \pi(s) \ea \right)
	& = \op{A}_\mu^\oplus
	\left( \ba{c} \xi(s) \\ \pi(s) \ea \right)\,,
	\quad 
	\op{A}_\mu^\oplus \doteq
	\left( \ba{cc}
		0 & \ts{\frac{1}{m}} \, \op{I}_\mu^\half
		\\
		-\kappa\, \opAsqrt\, \op{I}_\mu^\half\, \opAsqrt & 0 
	\ea \right)\,,
	\quad
	\dom(\op{A}_\mu^\oplus) \doteq \cY_\half\,,
	\ooer{eq:Cauchy-mu}
\end{align}
where $\cY_\half$ is the Hilbert space defined in \er{eq:cY-oplus}. Here, the augmented state is constructed from the (position) state $\xi(s)\in\cX_\half$ of the dynamics \er{eq:dynamics} driven by the optimal input $w^*(s)$ of \er{eq:optimal}, \er{eq:z-star}, together with a transformed (momentum) costate
$
	\pi(s)
	\doteq  \op{I}_\mu^{-\half}\, p(s)
	= m\, \op{I}_\mu^{-\half}\, w^*(s)
	= \opAsqrt\, \op{I}_\mu^\half\, \opAsqrt ( \opbreve{P}^{\mu,\infty}(t-s) \, \xi(s) + \opbreve{Q}^{\mu,\infty}(t-s)\, z^* ) \in \cX
$
for all $s\in[0,t-\delta)$, $z\in\cX_\half$, where $\delta\in(0,\bar t^\mu)$ is as per Theorem \ref{thm:Wbreve-ok}, and $z^*\in\cX_\half$ is as per \er{eq:z-star}. (Note that $\opAsqrt\, \op{I}_\mu^\half\, \opAsqrt\in\bo(\cX_\half;\cX)$ by Lemma \ref{lem:op-I-mu-properties}.) Meanwhile, the wave equation \er{eq:wave} defines an analogous abstract Cauchy problem, namely,
\begin{align}   
    \left( \ba{c} \dot x(s) \\ \dot p(s) \ea \right)
    & = \op{A}^\oplus
    \left( \ba{c} x(s) \\ p(s) \ea \right)\,,
    \quad 
    \op{A}^\oplus 
    \doteq
    \left( \ba{cc}
        0 & \ts{\frac{1}{m}} \, \op{I}
        \\
        -\kappa\, \op{A} & 0 
    \ea \right),
    \quad \dom(\op{A}^\oplus) = \cY_0\,,
    \ooer{eq:Cauchy}
\end{align}
where $\cY_0$ is the set defined in \er{eq:cY-0}. As noted in the following lemma, operators $\op{A}_\mu^\oplus$ and $\op{A}^\oplus$ generate respective semigroups of bounded linear operators defined on {\em all} time horizons. Crucially, these operators converge in an appropriate sense as $\mu\rightarrow 0$. Furthermore, the subsequent theorem shows that the generated semigroups also converge, implying that any solution of the abstract Cauchy problem \er{eq:Cauchy-mu} converges to an analogous solution of the abstract Cauchy problem \er{eq:Cauchy}. This naturally includes respective trajectories corresponding to the approximate and exact solution of TPBVPs such as \er{eq:TPBVP}.

% Where a classical solution of \er{eq:Cauchy} exists, it is precisely the solution of \er{eq:wave}.

\begin{lemma}
\label{lem:op-A-oplus-mu-bounded}
Given $\mu\in(0,1]$, the operators $\op{A}_\mu^\oplus$ and $\op{A}^\oplus$ of \er{eq:Cauchy-mu} and \er{eq:Cauchy} satisfy the following properties: 
\begin{enumerate}[(i)]
\item $\op{A}_\mu^\oplus\in\bo(\cY_\half)$; 
\item $\op{A}_\mu^\oplus$ generates a uniformly continuous semigroup of bounded linear operators $\op{T}_\mu^\oplus(t)\in\bo(\cY_\half)$, $t\in\R_{\ge 0}$;
\item $\op{A}^\oplus$ is unbounded, closed, and densely defined on $\cY_0$ (with $\ol{\cY_0} = \cY_\half$);
\item $\op{A}^\oplus$ generates a strongly continuous semigroup of bounded linear operators $\op{T}^\oplus(t)\in\bo(\cY_\half)$, $t\in\R_{\ge 0}$;
\item $\op{A}_\mu^\oplus$ converges strongly to $\op{A}^\oplus$ as $\mu\rightarrow 0$, i.e. $\lim_{\mu\rightarrow 0} \| \op{A}_\mu^\oplus\, y - \op{A}^\oplus\, y\|_\oplus = 0$ for all $y\in\cY_0$.
\end{enumerate}
\end{lemma}
\begin{proof}
{\em (i)} Fix any $y\doteq \left(\ba{c}  \xi \\ \pi \ea\right)\in\cY_\half$. Applying definitions \er{eq:cY-oplus} and \er{eq:Cauchy-mu} of $\|\cdot\|_\oplus$ and $\op{A}_\mu^\oplus$, 
\begin{align}
	\| \op{A}_\mu^\oplus\, y\|_\oplus^2
	= \left\| \left( \ba{c} \ts{\frac{1}{m}}\, \op{I}_\mu^\half\, \pi \\ -\kappa\, \opAsqrt\, \op{I}_\mu^\half\, \opAsqrt\, \xi \ea \right) \right\|_\oplus^2
	& = \ts{\frac{1}{m}}\, \left\| \op{I}_\mu^\half\, \pi \right\|_\half^2 
			+ \kappa\, \left\| \opAsqrt\, \op{I}_\mu^\half\, \opAsqrt\, \xi \right\|^2
	= \ts{\frac{1}{m}}\, \left\| \opAsqrt\, \op{I}_\mu^\half\, \pi \right\|^2 
			+  \kappa\, \left\| \opAsqrt\, \op{I}_\mu^\half\, (\opAsqrt\, \xi) \right\|^2
	\nn\\
	& \le \ts{\frac{1}{m}}\, \left\| \opAsqrt\, \op{I}_\mu^\half\ \right\|^2 \, \|\pi\|^2 + \kappa\, \left\| \opAsqrt\, \op{I}_\mu^\half \right\|^2 \, \|\xi\|_\half^2
	=  M^2\, \| y \|_\oplus^2\,,
	\nn
\end{align}
where $M \doteq (\ts{\frac{\kappa}{m}})^\half \left\| \opAsqrt\, \op{I}_\mu^\half\ \right\| < \infty$ by assertion \er{eq:op-I-mu-ass-4a} of Lemma \ref{lem:op-I-mu-properties}, as required.

{\em (ii)} Immediate by {\em(i)} and \cite[Theorem 1.2, p.2]{P:83}.

{\em (iii)} and {\em (iv)} Follows by an analogous argument to \cite[Example 2.2.5, p.34]{CZ:95}.

{\em (v)} Fix any $y\doteq \left(\ba{c}  \xi \\ \pi \ea\right)\in\dom(\op{A}^\oplus) = \cY_0 = \cX_0\oplus\cX_\half$. Recalling \er{eq:cY-oplus}, \er{eq:Cauchy-mu} and \er{eq:Cauchy},
\begin{align}
	\left\| \op{A}_\mu^\oplus\, y - \op{A}^\oplus\, y \right\|_\oplus^2
	& = \left\| \left( \ba{cc}
		0 & \ts{\frac{1}{m}} \, \op{I}_\mu^\half
		\\
		-\kappa\, \opAsqrt\, \op{I}_\mu^\half\, \opAsqrt & 0 
	\ea \right) y - 
	\left( \ba{cc}
		0 & \ts{\frac{1}{m}} \, \op{I}
		\\
		-\kappa\, \op{A} & 0 
	\ea \right) y
	\right\|_\oplus^2
	= \left\| \left( \ba{c}
			\ts{\frac{1}{m}} \, (\op{I}_\mu^\half - \op{I}) \, \pi
			\\
			-\kappa \, \opAsqrt \, (  \op{I}_\mu^\half - \op{I} ) \, \opAsqrt \, \xi
		\ea \right)
	\right\|_\oplus^2
	\nn\\
	& = \ts{\frac{1}{m}}\, \| (\op{I}_\mu^\half - \op{I}) \, \pi \|_\half^2 + 
	\ts{\frac{1}{\kappa}}\, \| \opAsqrt \, (  \op{I}_\mu^\half - \op{I} ) \, \opAsqrt \, \xi \|^2
	= \ts{\frac{1}{m}}\, \| (\op{I}_\mu^\half - \op{I}) \, \opAsqrt\, \pi \|^2 + 
	\kappa\, \| (  \op{I}_\mu^\half - \op{I} ) \, \op{A} \, \xi \|^2\,,
	\label{eq:strong-half-mu-0}
\end{align}
where the last equality follows by definition of $\|\cdot\|_\half$ and assertion \er{eq:op-I-mu-ass-3b} of Lemma \ref{lem:op-I-mu-properties}. Note further that $\opAsqrt \, \pi,\, \op{A}\, \xi\in\cX$ by definition of $y\in\cY_0$. Consequently, it remains to be shown that $\op{I}_\mu^\half$ converges strongly to $\op{I}$ on $\cX$ as $\mu\rightarrow 0$. To this end, fix any $x\in\cX_0$, and note that $\|\op{A}\, x\| < \infty$. Note also that $\op{I}_\mu^\half - \op{I}$ is a Riesz-spectral operator on $\cX$, with $\dom(\op{I}_\mu^\half - \op{I}) = \cX$, so that
\begin{align}
	\left\| (\op{I}_\mu^\half - \op{I})\, x \right\|^2
	& = \sum_{n=1}^\infty \beta_{\lambda_n}(\mu^2) \, |\langle x,\, \varphi_n \rangle|^2\,,
	\label{eq:strong-half-mu-1}
\end{align}
where $\beta_\lambda:\R_{\ge 0}\mapsinto[0,1)$ is defined for each $\lambda\in\R_{>0}$ by $\beta_{\lambda}(\eps) \doteq [ 1 - \frac{1}{\sqrt{1 + \eps\, \lambda}} ]^2$. Taylor's theorem implies that for any $\eps\in\R_{\ge 0}$, there exists an $c_\eps\in(0,\eps)$ such that
$
	\beta_\lambda(\eps)
	= [ \ddttwo{\beta_\lambda}{\eps}(c_\eps) ] {\frac{\eps^2}{2}}
	\le \demi\, \lambda^2 \, \eps^2 \left[ \ts{\frac{3}{2}}\, (1 + c_\eps\, \lambda)^{-\frac{5}{2}} + 2\, (1 + c_\eps\, \lambda)^{-3} \right]
	\le \ts{\frac{7}{4}}\, \lambda^2\, \eps^2
$
for all $\lambda\in\R_{>0}$. Substitution in \er{eq:strong-half-mu-1} yields that
$
	\| (\op{I}_\mu^\half - \op{I})\, x \|^2
	\le 
	% \sum_{n=1}^\infty \ts{\frac{7}{4}}\, \lambda_n^2\, \mu^4  \, |\langle x,\, \varphi_n \rangle|^2
	% = 
	\ts{\frac{7}{4}}\, \mu^4  \,  \sum_{n=1}^\infty  |\lambda_n|^2\, |\langle x,\, \varphi_n \rangle|^2
	= \ts{\frac{7}{4}}\, \mu^4  \,  \|\op{A}\, x\|^2
$.
Recalling that $x\in\cX_0$, so that $\|\op{A}\, x\|<\infty$, it follows immediately that $\lim_{\mu\rightarrow 0}  \| (\op{I}_\mu^\half - \op{I})\, x \| = 0$ for any $x\in\cX_0$. As $\op{I}_\mu^\half\in\bo(\cX)$ by Lemma \ref{lem:op-I-mu-properties}, and $\cX_0$ is dense in $\cX$, it may also be concluded that $\lim_{\mu\rightarrow 0}  \| (\op{I}_\mu^\half - \op{I})\, x \| = 0$ for any $x\in\cX$. Applying this fact in \er{eq:strong-half-mu-0} completes the proof.
\end{proof}

\begin{theorem}
\label{thm:trotter-kato}
$\op{T}_\mu^\oplus(t)$ converges strongly to $\op{T}^\oplus(t)$ as $\mu\rightarrow 0$, uniformly for $t\in\R_{>0}$ in compact intervals. In particular, $\lim_{\mu\rightarrow 0} \left\| \op{T}_\mu^\oplus(t)\, y - \op{T}^\oplus(t)\, y \right\|_\oplus = 0$ for all $y\in\cY$, $t\in\Omega$, $\Omega\subset\R_{\ge 0}$ compact.
\end{theorem}
\begin{proof}
The proof follows by application of the First Trotter-Kato Approximation Theorem (see for example \cite[Theorem 4.8, p.209]{EN:00}), via Lemma \ref{lem:op-A-oplus-mu-bounded}.
\end{proof}

With the convergence property of all solutions of \er{eq:Cauchy-mu} and \er{eq:Cauchy} provided by Theorem \ref{thm:trotter-kato}, and formulae \er{eq:optimal}, \er{eq:z-star} for the optimal input that generates the corresponding optimal trajectory that renders the approximate action functional stationary, a recipe for approximating the solution of TPBVPs such as \er{eq:TPBVP} on short time horizons may be enumerated.
\begin{center}
\vspace{3mm}
{\bf Recipe for the approximate solution of TPBVPs for horizons $t\in(0,\bar t^\mu)$}

\parbox[t]{16cm}{%
\begin{itemize}
\item[\done] Select the approximation parameter $\mu\in(0,1]$, and a truncation order $N\in\N$ for the Riesz-spectral operator representations.
\item[\dtwo] Fix $t\in(0,\bar t^\mu)$, where $\bar t^\mu\in\R_{>0}$ is as per \er{eq:t-bar-mu}.
\item[\dthree] Select a terminal payoff $\psi:\cX_\half\mapsinto\R\cup\{-\infty\}$ that encapsulates the terminal condition of interest, e.g. \er{eq:stationary-payoff}, and apply \er{eq:z-star} via the fundamental solution \er{eq:explicit-W-infty} to determine the corresponding terminal state $z^*\in\cX_\half$, see for example \er{eq:z-star-terminal-velocity}.
\item[\dfour] Truncate the Riesz-spectral operator representation for the optimal input $w^*(0)$ of \er{eq:optimal}, with
\begin{align}
	w^*(0)
	& = k(0,x) = \ts{\frac{1}{m}} \, \opAsqrt \, \op{I}_\mu\, \opAsqrt\, \left( \opbreve{P}^{\mu,\infty}(t)\, x + \opbreve{Q}^{\mu,\infty}(t)\, z^* \right)
	\nn\\
	& \approx \tilde w^*(0) \doteq \ts{\frac{1}{m}} \sum_{n=1}^N \frac{\lambda_n}{1 + \mu^2\, \lambda_n} \left( 
			p_n^{\mu,\infty}(t)\, \langle x,\, \tilde\varphi_n \rangle_\half
			+ q_n^{\mu,\infty}(t)\, \langle z^*,\, \tilde\varphi_n \rangle_\half
	\right)\,,
	\nn
\end{align}
where $p_n^{\mu,\infty}(t)$ and $q_n^{\mu,\infty}(t)$ are as per \er{eq:eig-pqr-infty}.
\item[\dfive] Propagate the solution of the wave equation \er{eq:wave} using $\pdtone{u}{s}(0,\cdot) = \tilde w^*(0)$. 
\end{itemize}
}
\vspace{3mm}
\end{center}

With particular reference to step {\dfour} in the case where a fixed final velocity $v\in\cX_\half$ is specified via $\psi \doteq \psi_v$ as per \er{eq:stationary-payoff}, substitution of the left-hand equality of \er{eq:z-star-terminal-velocity} in \er{eq:optimal} yields the required initial velocity as 
\begin{align}
	w^*(0)
	& = \ts{\frac{1}{m}}\, \opAsqrt\, \op{I}_\mu\, \opAsqrt \left( 
			\left[ \opbreve{P}^{\mu,\infty}(t) - \opbreve{Q}^{\mu,\infty}(t) \left( \opbreve{R}^{\mu,\infty}(t) \right)^{-1} \opbreve{Q}^{\mu,\infty}(t)' 
			\right] x 
			- m\, \opbreve{Q}^{\mu,\infty}(t) \left( \opbreve{R}^{\mu,\infty}(t) \right)^{-1} \, \op{J}\, \op{J}\, v \right).
	\label{eq:w-star-terminal-velocity}
\end{align}
A series form for $w^*(0)$ follows by substitution of the Riesz-spectral operator representations for $\op{P}^{\mu,\infty}(t)$, $\opbreve{Q}^{\mu,\infty}(t)$, $\opbreve{R}^{\mu,\infty}(t)$, $\opAsqrt$, $\op{I}_\mu$, and $\op{J}$ into \er{eq:w-star-terminal-velocity}, with the details omitted for brevity.

%% if-falsed old "other TPBVPs" text %%
\if{false}

?? 01Jan15 -- MOVE BACK TO OPTIMAL CONTROL SECTION ??

In the more general case where a terminal payoff $\psi$ is used to encapsulate alternative terminal data (for example, fixed terminal velocity), $W^{\mu,\infty}$ may be employed as the fundamental solution to the optimal control problem defined by $W^\mu$ of \er{eq:W} via Theorem \ref{thm:reach}. In particular, $z\in\cX_\half$ must be initialized via \er{eq:W-from-fund-conjecture}, with $z=z^*$ specified as per \er{eq:z-star}.
For example, 

??

Applying \er{eq:optimal}, 
\fi
%% end of if-falsed old "other TPBVPs" text %%

%%		Longer horizons.

\subsection{Longer horizons}
\label{sec:long-horizon}
As noted previously, the correspondence between stationary action and optimal control exploited for shorter horizons via \er{eq:W} may break down for longer time horizons due to loss of concavity of the associated payoff \er{eq:payoff}, see Theorem \ref{thm:second-difference}. Consequently, for longer time horizons, a modified approach is required. Two such approaches have been developed for finite dimensional problems, see \cite{MD-long-A:14,MD-stat:14}, based on replacing the supremum in the definition \er{eq:W} of the associated optimal control problem with a {\em stat} operation. This $\stat$ operation yields the stationary payoff (and hence the stationary action functional) without assuming that it is achieved at a maximum. In particular, in \cite{MD-long-A:14}, longer time horizons are accumulated via the concatenation of sufficiently many sufficiently short time horizons, with the $\stat$ operation used to characterize the intermediate states joining adjacent short time horizons. More generally, the supremum over inputs in \er{eq:W} may be completely replaced with the $\stat$ operation, see \cite{MD-stat:14}. Using either approach here requires a corresponding extension to infinite dimensions. For brevity, such an extension is postponed to later work. Instead, for the purpose of presenting an illustrative example in Section \ref{sec:example}, an outline of the development of the former (concatenation) approach is provided, in a formal setting only. This outline is as follows.

Given a fixed longer time horizon $t\in(\bar t^\mu,\infty)$ of interest, select a sufficiently large number $n_t\in\N$ of shorter horizons $\tau\doteq t/n_t$ such that $\tau\in(0,\bar t^\mu)$. By definition of $\tau$, Theorem \ref{thm:second-difference} implies that the payoff $J_{m,\psi}^\mu(\tau,x,\cdot)$ defined by \er{eq:payoff}, and hence the action functional of \er{eq:action-mu}, is concave for any $x\in\cX_\half$. That is, the action functional is concave on each of the subintervals $[(k-1)\, \tau, k\, \tau]$, $k\in [1,n_t]\cap\N$, with any loss of concavity occurring in the dependence on the intermediate states $\zeta_k\doteq\xi(k\,\tau)\in\cX_\half$. 

Motivated by this observation, a correspondence between stationary action and optimal control can be established for longer horizons for finite dimensional problems by relaxing the supremum in the associated optimal control problem, see \cite{MD-long-A:14}. In the infinite dimensional case considered here, it is conjectured that the fundamental solution $W^{\mu,\infty}(t,\cdot,\cdot)$ of the approximating optimal control problem defined by \er{eq:W}, as appearing in \er{eq:W-from-fund-conjecture}, is defined on longer time horizons by
\begin{align}
	W^{\mu,\infty}(t,x,z)
	& \doteq \stat_{\zeta\in(\cX_\half)^{n_t-1}} \left\{ \sum_{k=1}^{n_t} W^{\mu,\infty}(\tau,\zeta_{k-1},\zeta_k) \, 
		\biggl| \, \zeta_0 = x,\, \zeta_{n_t} = z
	\right\} 
	% \nn\\
	% & = \stat_{\zeta_1\in\cX_\half} \left\{ 
	%	W^{\mu,\infty}(\tau, x, \zeta_1) + \stat_{\zeta_2\in\cX_\half}  \biggl\{ 
	%		W^{\mu,\infty}(\tau, \zeta_1, \zeta_2) + \cdots \right.
	%		\nn\\
	% &		\qquad \left. \left. \cdots +
	%		\stat_{\zeta_{n_t-1}\in\cX_\half} \biggl\{ 
	%			W^{\mu,\infty}(\tau,\zeta_{n_t-2},\zeta_{n_t-1}) + W^{\mu,\infty}(\tau,\zeta_{n_t-1},z) 
	%		\biggr\} \cdots
	%	\right\}
	% \right\}
	\label{eq:W-longer}
\end{align}
for all $x,z\in\cX_\half$, in which the $\stat$ operation is defined generally by
\begin{align}
	\stat_{x\in\cX_\half} F(x) 
	& \doteq \left\{ F(\bar x) \, \left| \, \bar x\in\argstat_{x\in\cX_\half} F(x) \right. \right\},
	\quad
	\argstat_{x\in\cX_\half} F(x) 
	\doteq \left\{ x\in\cX_\half \, \left| \, 0 = \lim_{y\rightarrow x}  \frac{|F(y) - F(x)|}{\|y - x\|_\half} \right. \right\},
	\nn
\end{align}
for functional $F:\cX_\half\mapsinto\R$. Figure \ref{fig:subintervals} provides an illustration of the role of the intermediate states $\zeta_k\in\cX_\half$, $k\in[1,n_t]\cap\N$. (Note that replacing $\stat$ with $\sup$ in \er{eq:W-longer} recovers the original short horizon fundamental solution \er{eq:W-infty} as per \er{eq:W-from-fund-conjecture}, albeit applied to the longer horizon.)

\begin{figure}[h]
\begin{center}
\vspace{3mm}
\epsfig{file=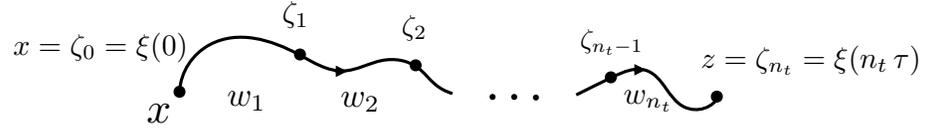,width=10cm}
\caption{Concatenations of trajectories to yield a longer time horizon.}
\label{fig:subintervals}
\end{center}
\end{figure}

In order to test the conjecture that \er{eq:W-longer} is a suitable generalization of the longer horizon fundamental solution, recall that $W^{\mu,\infty}(\tau,\cdot,\cdot)$ takes the form of the quadratic functional given by \er{eq:explicit-W-infty}, see Theorem \ref{thm:Wbreve-ok}. Combining \er{eq:explicit-W-infty} and \er{eq:W-longer},
\begin{align}
	W^{\mu,\infty}(t,x,z)
	& = \stat_{\zeta\in (\cX_\half)^{n_t-1}} \Theta^\mu(\tau,x,\zeta,z)\,,
	\qquad
	\Theta^\mu(\tau,x,\zeta,z)
	\doteq \demi \left\langle \left( \ba{c} x \\ \zeta \\ z \ea \right), \, 
	\Pi^\mu(\tau) \left( \ba{c} x \\ \zeta \\ z \ea \right)
	\right\rangle_{\star}
	\label{eq:W-longer-quad}
\end{align}
where $\langle \cdot, \, \cdot \rangle_{\star}$ denotes an inner product on $(\cX_\half)^{n_t+1}$, defined for all $\hat\zeta,\, \hat\xi\in(\cX_\half)^{n_t+1}$ by $\langle \hat\zeta, \hat\xi \rangle_{\star} \doteq \sum_{i=1}^{n_t+1} \langle \hat\zeta_i, \, \hat\xi_i \rangle_\half$, and $\Pi^\mu(\tau)\in(\bo(\cX_\half))^{(n_t+1)\times(n_t+1)}$ is a matrix of Riesz-spectral operators defined by
\begin{align}
	\Pi^\mu(\tau)
	& \doteq \left( \ba{c|cccc|c}
		\opbreve{P}^{\mu,\infty}(\tau) & \opbreve{Q}^{\mu,\infty}(\tau) & 0 % & 0 
		& \cdots & 0 & 0
		\\\hline
		\opbreve{Q}^{\mu,\infty}(\tau)' & \opbreve{P}^{\mu,\infty}(\tau) + \opbreve{R}^{\mu,\infty}(\tau) & \opbreve{Q}^{\mu,\infty}(\tau) % & 0 
		& \cdots & 0 & 0
		\\
		0 & \opbreve{Q}^{\mu,\infty}(\tau)' & \opbreve{P}^{\mu,\infty}(\tau) + \opbreve{R}^{\mu,\infty}(\tau) % & \opbreve{R}^{\mu,\infty}(\tau) 
			& \cdots & 0 & 0
		\\
		0 & 0 & \opbreve{Q}^{\mu,\infty}(\tau)' % & \opbreve{P}^{\mu,\infty}(\tau) + \opbreve{R}^{\mu,\infty}(\tau) 
			& \cdots & 0 & 0
		\\
		\vdots & \vdots & \vdots % & \vdots 
			& \ddots & \vdots & \vdots
		\\
		0 & 0 & 0 % & 0 
			& \cdots & \opbreve{Q}^{\mu,\infty}(\tau) & 0
		\\	
		0 & 0 & 0 % & 0 
			& \cdots & \opbreve{P}^{\mu,\infty}(\tau) + \opbreve{R}^{\mu,\infty}(\tau) & \opbreve{Q}^{\mu,\infty}(\tau)
		\\\hline
		0 & 0 & 0 % & 0 
			& \cdots & \opbreve{Q}^{\mu,\infty}(\tau)' & \opbreve{R}^{\mu,\infty}(\tau) 
	\ea \right)
	\nn\\
	& \equiv \left( \ba{c|c|c}
		\opbreve{P}^{\mu,\infty}(\tau) & \Pi_{1,2}^\mu(\tau) & 0
		\\\hline
		{\Pi}_{1,2}^\mu(\tau)' & \Pi_{2,2}^\mu(\tau) & \Pi_{2,3}^\mu(\tau)
		\\\hline
		0 & \Pi_{2,3}^\mu(\tau)' & \opbreve{R}^{\mu,\infty}(\tau)
	\ea \right)\,.
	\label{eq:matrix-of-Riesz}
\end{align}
The existence of a solution of a TPBVP such as \er{eq:TPBVP} on the long horizon $t\in\R_{>0}$ requires (by the action principle) that the $\stat$ in the definition \er{eq:W-longer-quad} of $W^{\mu,\infty}(t,x,z)$ exist. 
% This follows as there must exist a stationary point of the original action functional. 
Define $\zeta^* \in (\cX_{\frac{1}{2}})^{n_t-1}$ by
\begin{align}
	\zeta^{\mu,*} \doteq \argstat_{\zeta\in(\cX_\half)^{n_t-1}} \Theta^\mu(t,x,\zeta,z)\,,
	\nn
\end{align}
and note that $0 = \nabla_\zeta \Theta^\mu(t,x,\zeta^{\mu,*},z)$, where $\nabla_\zeta \Theta^\mu(t,x,\zeta,z)\in(\cX_\half)^{n_t-1}$ is the {\Frechet} derivative of $\Theta^\mu(t,x,\cdot,z)$ at $\zeta\in(\cX_\half)^{n_t-1}$. Applying \er{eq:matrix-of-Riesz} yields
\begin{align}
	\Pi_{2,2}^\mu(\tau)\, \zeta^{\mu,*}
	& = - \Pi_{1,2}^\mu(\tau)'\, x - \Pi_{2,3}^\mu(\tau)\, z\,.
	\label{eq:zeta-star}
\end{align}
% which is necessarily solvable for $\zeta^*\in(\cX_\half)^{n_t-1}$ by the stationary action principle. 
Hence, on the longer horizon $t\in[\bar t^\mu,2\, \bar t^{\mu})$ (for example), the fundamental solution \er{eq:W-infty} generalizes as per \er{eq:W-longer} to $W^{\mu,\infty}(t,x,\zeta^{\mu,*}) + W^{\mu,\infty}(t,\zeta^{\mu,*},z)$, where $\zeta^{\mu,*}$ solves \er{eq:zeta-star} for $n_t \doteq 2$, $\tau\doteq t/2$. This approach generalizes to any fixed longer horizon $t\in\R_{>0}$, and taking $\mu\rightarrow 0$ corresponds to sending $n_t\rightarrow\infty$ in \er{eq:W-longer}. Furthermore, in this limit, it may be shown that evaluating \er{eq:W-longer} via \er{eq:W-longer-quad} and \er{eq:zeta-star} yields the same explicit quadratic representation for the fundamental solution $W^{\mu,\infty}(t,x,z)$, but with $\mu=0$, as presented in Theorem \ref{thm:Wbreve-ok}.
% For brevity, the details are postponed to later work.

%%%%%%%%%%%%%%%%%%%%%%%%%%%%%%%%%%%%%%%%%%%%%%%%%%%%%%%%%%%%%%%%%%%%
%%
%%		Example
%%

\section{Example}
\label{sec:example}

For sufficiently short time horizons, as considered in Section \ref{sec:short-horizon}, the action principle corresponding to the wave equation \er{eq:wave} may be approximated by the optimal control problem \er{eq:W}. This approximation may be extended to longer horizons via the concatenation approach outlined in Section \ref{sec:long-horizon}, and becomes exact in the limit of the perturbation parameter $\mu\in\R$ tending to zero (see Theorem \ref{thm:trotter-kato}). 
Consequently, as the action principle describes all possible solutions to the wave equation \er{eq:wave}, including those constrained by any specific combination of boundary data, TPBVPs involving this wave equation may be solved via the optimal control problem \er{eq:W}. The initial velocity $w^*(0)$ that solves such a TPBVP may be found via the recipe enumerated in Section \ref{sec:short-horizon}, with $\mu=0$. In particular, step {\dfour} of that recipe yields the initial velocity that solves a TPBVP as 
\begin{align}
	v = w^*(0) & = \ts{\frac{1}{m}}\, \opAsqrt\, \op{I}_0\, \opAsqrt\, \ggrad_x W^{0,\infty}(t,x,z)
	= \ts{\frac{1}{m}}\, \op{A} \, \left[ \opbreve{P}^{0,\infty}(t)\, x + \opbreve{Q}^{0,\infty}(t)\, z ^*\right]
	\nn\\
	& = \ts{\frac{1}{m}} \sum_{n=1}^\infty \lambda_n^0 \left[ p_n^{0,\infty}(t)\, \langle x,\, \tilde\varphi_n \rangle_\half +
	q_n^{0,\infty}(t)\, \langle z^*,\, \tilde\varphi_n \rangle_\half \right] \tilde\varphi_n\,,
	\label{eq:initial-velocity}
\end{align}
where $p_n^{0,\infty}$, $q_n^{0,\infty}$ are as per \er{eq:eig-pqr-infty} and $z^*$ is as per \er{eq:z-star}, after sending $\mu\rightarrow 0$.

\begin{figure}[t!]
\vspace{-3mm}
\begin{center}
\subfigure[Terminal displacement $z(\lambda)$ vs $\lambda$.]{
	\psfrag{XX}{$\lambda$}
	\psfrag{YYYYY}{\hspace{-10mm} $z(\lambda) \  (\equiv u(t,\lambda))$}
	\epsfig{file=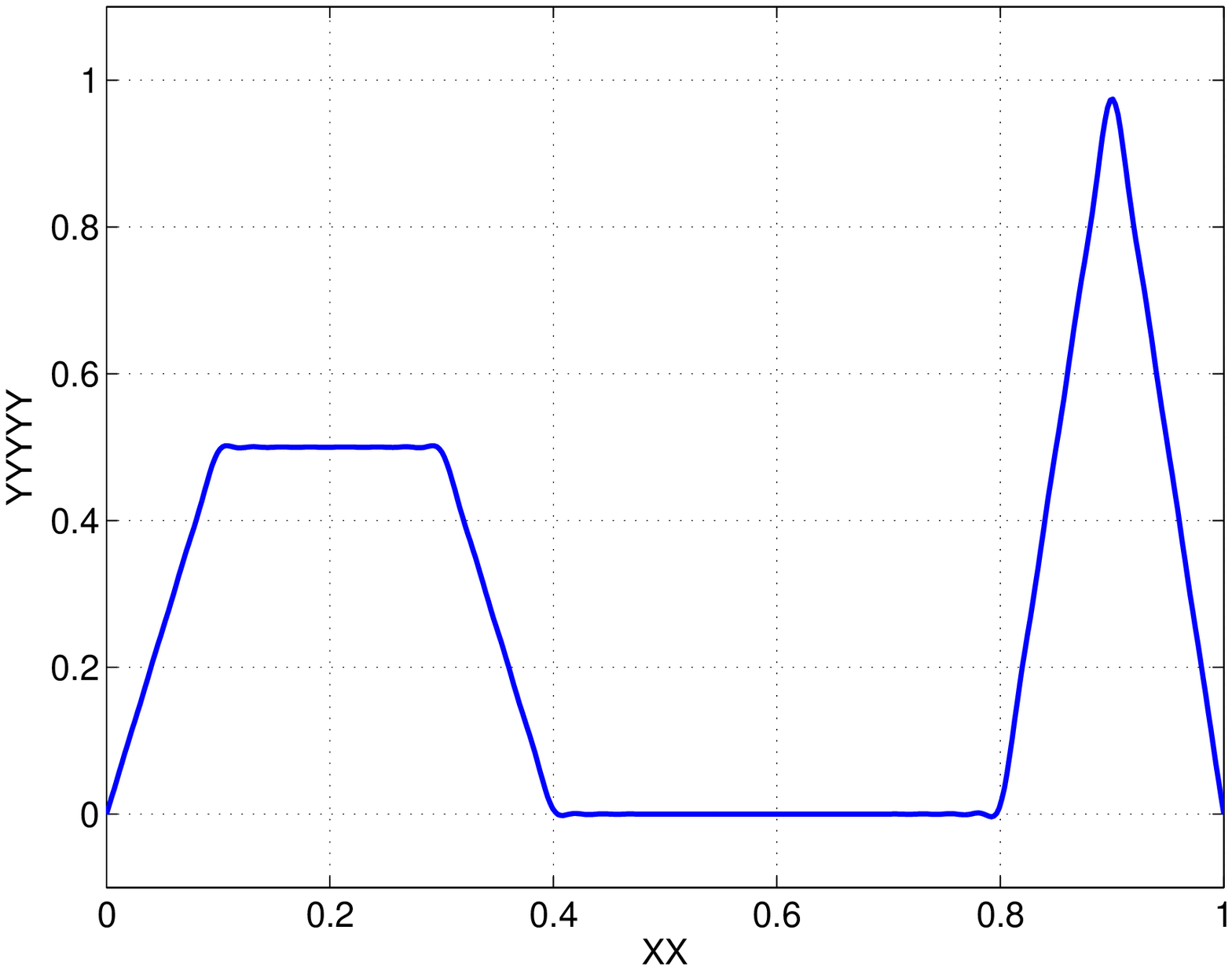,width=8cm}
}
\subfigure[Initial velocity $v(\lambda)$ vs $\lambda$.]{
	\psfrag{XX}{$\lambda$}
	\psfrag{YYYYY}{$\hspace{-10mm}v(\lambda) \  (\equiv \pdtone{u}{s}(0,\lambda))$}
	\epsfig{file=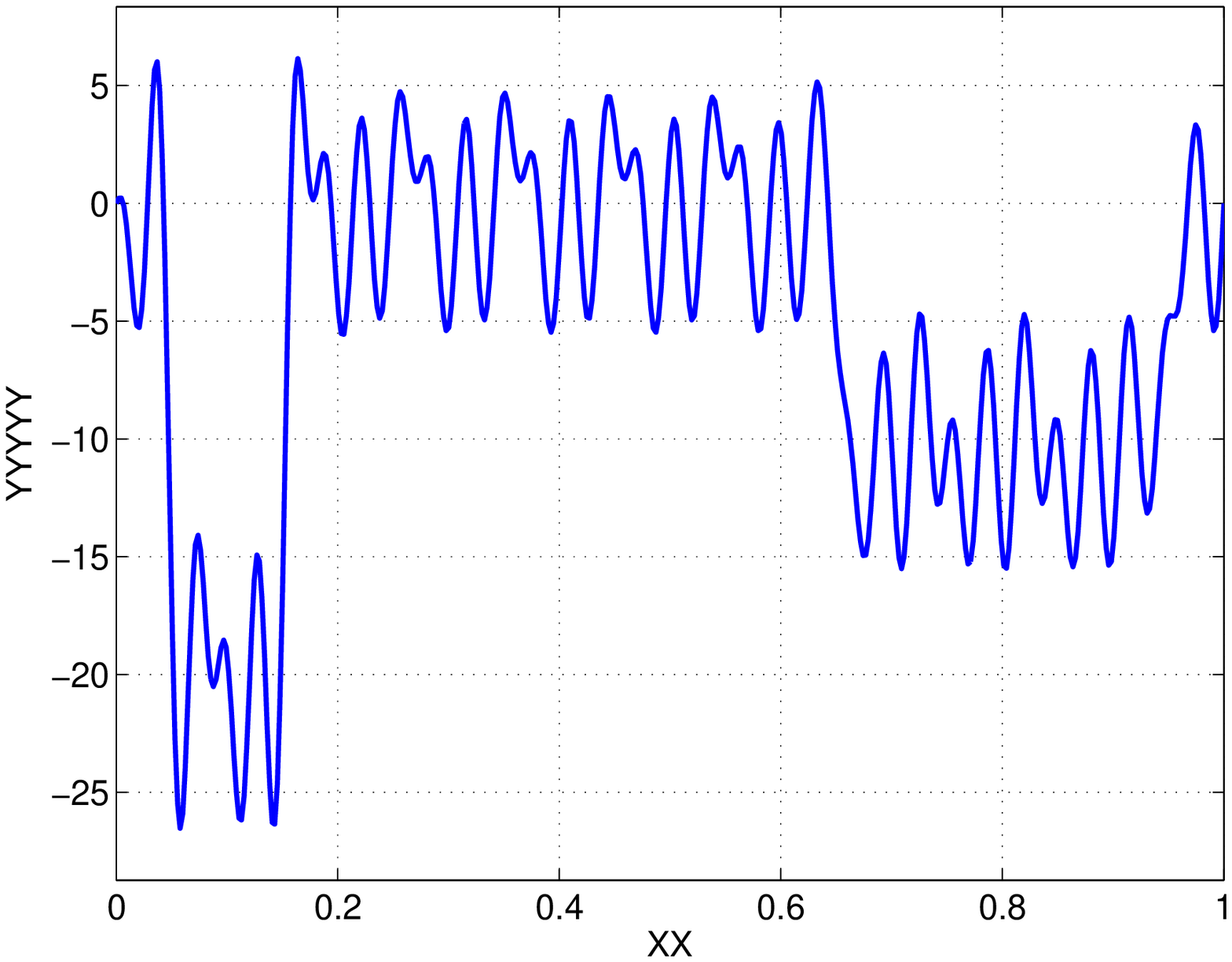,width=8cm}
}
\vspace{-4mm}
\caption{Initial velocity required to achieve terminal displacement.}
\label{fig:initial-terminal} 
\end{center}
% \end{figure}
%
% \begin{figure}[t]
\vspace{-3mm}
\begin{center}
\subfigure[Wave equation solution $u(s,\lambda)$, $s\in{[0,t]}$, $\lambda\in{[0,L]}$.]{
	\psfrag{XX}{$\lambda$}
	\psfrag{YY}{$s$}
	\psfrag{ZZZZ}{$u(s,\lambda)$}
	\epsfig{file=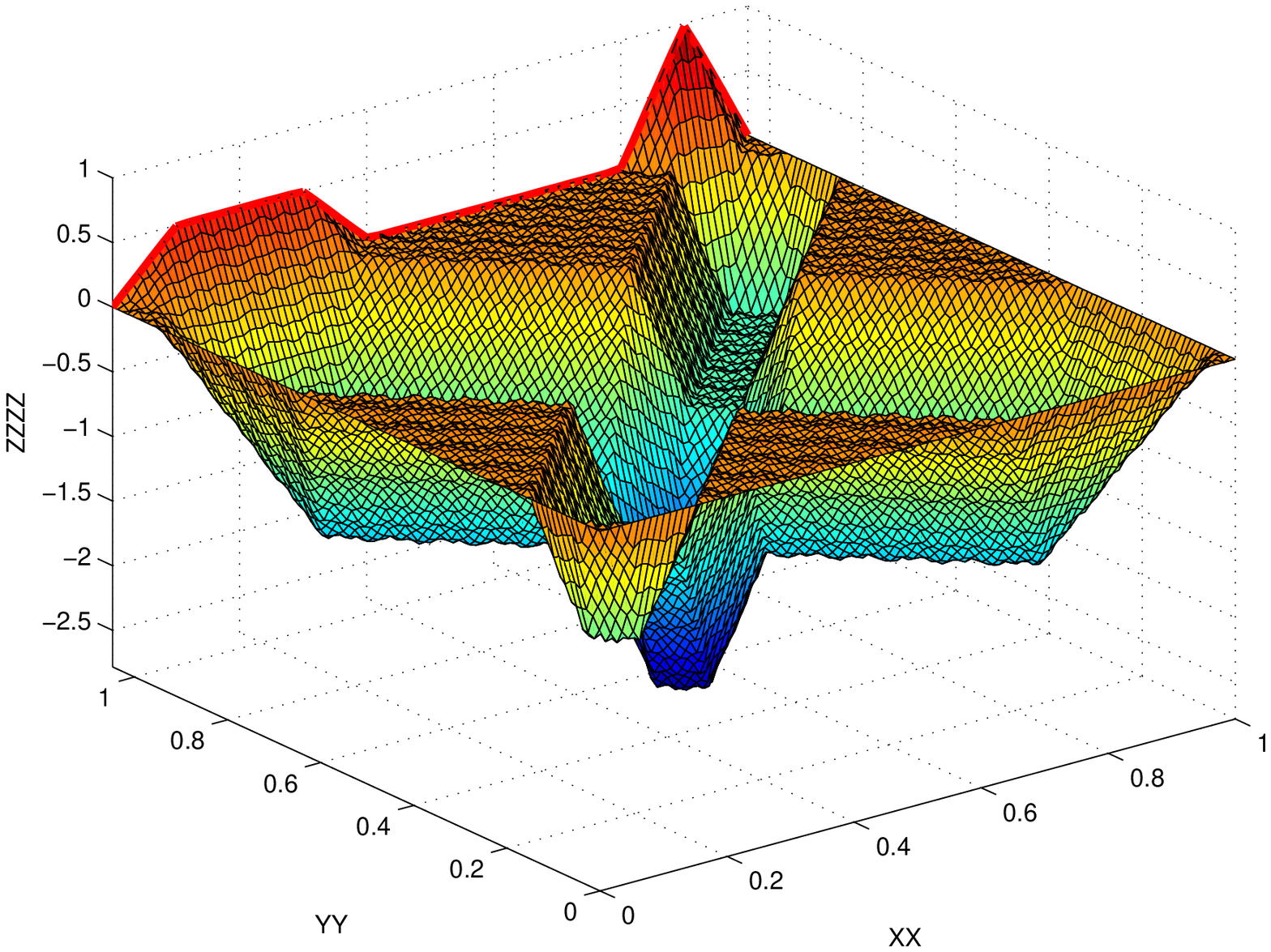,width=8cm}
}
\subfigure[Wave equation solution $u(s,\lambda)$, $s\in{[0,4 L (\ts{\frac{m}{\kappa}})^\half]}$, $\lambda\in{[0,L]}$.]{
	\psfrag{XX}{$\lambda$}
	\psfrag{YY}{$s$}
	\psfrag{ZZZZ}{$u(s,\lambda)$}
	\epsfig{file=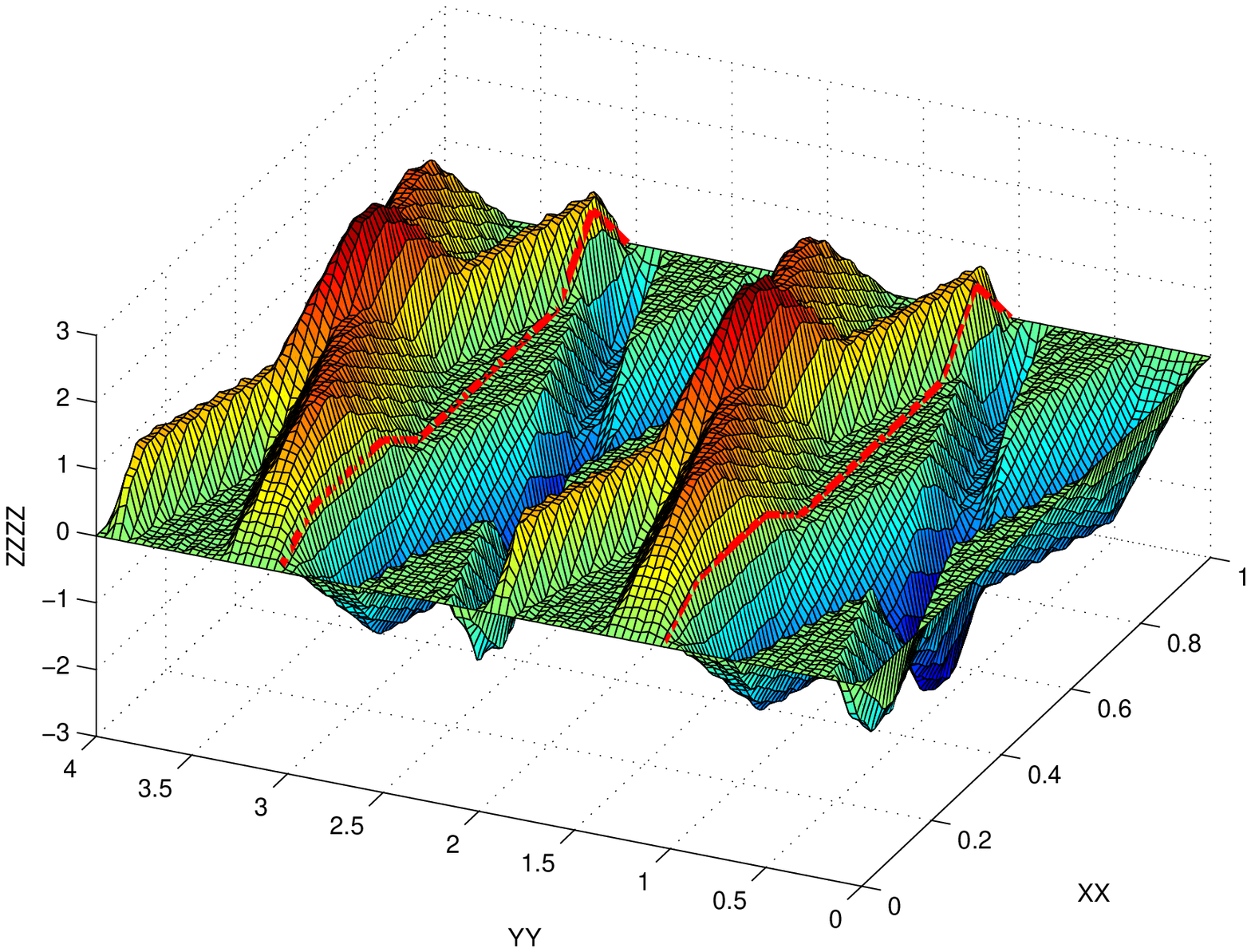,width=8cm}
}
\vspace{-4mm}
\caption{Solutions of wave equation \er{eq:wave}.}
\label{fig:wave-solutions}
\end{center}
\end{figure}

With $m=\kappa=L=1$ (with appropriate dimensions) in \er{eq:wave}, suppose that the specific problem $\TPBVP(t,x,z)$ is to be solved given the (arbitrary) initial displacement $x=0\in\cX_0$, terminal displacement $z\in\cX_0$ as per Figure \ref{fig:initial-terminal}(a), and horizon $t \doteq \ts{\frac{\pi\, L}{3}}\, ( \frac{m}{\kappa} )^\half$ ($\approx 1.05$). Recall in that case that the terminal payoff $\psi \doteq \psi^{\mu,\infty}(\cdot,z)$ encapsulates the required terminal displacement, and $z^* = z$ is required in \er{eq:initial-velocity}, as per discussion following \er{eq:z-star}. Applying \er{eq:initial-velocity} then yields the required initial velocity illustrated in Figure \ref{fig:initial-terminal}(b) that solves TPBVP \er{eq:TPBVP} via the fundamental solution \er{eq:W-infty}. This solution may be tested by propagating the initial displacement and velocity obtained forward to time $t$ by solving the wave equation \er{eq:wave} directly. (Here, the $C_0$-semigroup $\op{T}^\oplus(\cdot)$ of \er{eq:wave-solution} generated by $\op{A}^\oplus$ is applied to this end, see for example \cite{CZ:95}.) The resulting wave equation dynamics are illustrated in Figure \ref{fig:wave-solutions}(a), with the desired terminal displacement clearly achieved. Integration over a longer time period reveals (expected) periodic behaviour, see Figure \ref{fig:wave-solutions}(b).

%%%%%%%%%%%%%%%%%%%%%%%%%%%%%%%%%%%%%%%%%%%%%%%%%%%%%%%%%%%%%%%%%%%%%%%%%%
%%
%%		Conclusion
%%

\section{Conclusion}
A new fundamental solution based approach to solving a two point boundary value problem for a wave equation is considered. A value functional based characterization of this fundamental solution is formulated via the analysis of an optimal control problem that encapsulates the principle of stationary action. This value functional is shown to enjoy an explicit Riesz-spectral operator based representation via an associated infinite dimensional Hamilton Jacobi Bellman partial differential equation. Application of the fundamental solution obtained is illustrated via a simple example.

%%%%%%%%%%%%%%%%%%%%%%%%%%%%%%%%%%%%%%%%%%%%%%%%%%%%%%%%%%%%%%%%%%%%%%%%%%
%%
%%		Bibliography
%%

\bibliographystyle{siam}
\bibliography{spring}

%%%%%%%%%%%%%%%%%%%%%%%%%%%%%%%%%%%%%%%%%%%%%%%%%%%%%%%%%%%%%%%%%%%%%%%%%%
%%
%%		Appendix
%%

\appendix

%%
%%		Properties of the A operator and its square root.
%%

\section{Properties of operators $\op{A}$ and $\opAsqrt$}
\label{app:op-A-properties}
\label{app:half-basis}
Operators $\op{A}$ and $\opAsqrt$ are key to the application of the principle of stationary action to obtain the wave dynamics \er{eq:wave} via optimal control. The relevant properties of these operators are largely well-known \cite{CZ:95}, and are stated without proof unless otherwise indicated.

\begin{lemma}
\label{lem:op-A-properties}
The following properties hold on $\cX$:
\begin{enumerate}[(i)]
\item Operator $\op{A}$ is self-adjoint, positive, boundedly invertible, and closed, with 
\begin{align}
	& \op{A}' \, x = \op{A}\, x\,, 
	&&
	\forall \, x\in \dom(\op{A}') = \dom(\op{A}) = \cX_0\,,
	\nn\\
	&&&
	 \ran(\op{A}) = \ran(\op{A}') = \cX\,,
	\label{eq:op-A-ass-1a}
	\\
	& \langle x,\, \op{A}\, x \rangle = \| \partial x \|^2 > 0\,,
	&&
	\forall \ x\in\dom(\op{A}) = \cX_0\,, \, x\ne 0\,,
	\label{eq:op-A-ass-1b}
	\\
	& \op{A}^{-1}\, x = \int_\Lambda \hat A(\cdot,\zeta)\, x(\zeta)\, d\zeta\,,
	&&
	\hat A(\lambda,\zeta)
	\doteq  \ts{\frac{1}{L}} \left\{ \hspace{-1mm} \ba{cl}
		(L-\zeta)\, \lambda\,,	& 0\le \lambda\le \zeta \le L\,,
		\\
		(L-\lambda)\, \zeta\,,	& 0\le \zeta < \lambda \le L\,,
	\ea \right.
	\nn\\
	&&&
	 \forall \ x\in\dom(\op{A}^{-1}) = \cX\,,
	\nn\\
	&&&
	\ran(\op{A}^{-1}) = \dom(\op{A}) = \cX_0\,.
	\label{eq:op-A-ass-1c}
	% \label{eq:op-A-ass-3a}
\end{align}
\item Operator $\op{A}$ has a unique, positive, self-adjoint, boundedly invertible, and closed square root $\opAsqrt$, with 
\begin{align}
	& \opAsqrt\, x\in\dom(\opAsqrt) = \cX_\half
	&&
	\forall \ x\in\dom(\op{A}) = \cX_0\,,
	\label{eq:op-A-ass-2a}
	\\
	& \opAsqrt\, \opAsqrt\, x = \op{A} \, x\,,	
	&&
	\forall \ x\in\dom(\op{A}) = \cX_0\,,
	\label{eq:op-A-ass-2b}
	\\
	& \|\opAsqrt\, x\| = \| \partial x\| = \|x\|_\half
	&&
	\forall \ x\in\dom(\op{A}) = \cX_0\,,
	\label{eq:op-A-ass-2c}	
	\\
	& \op{J}\, x \doteq \left(\opAsqrt\right)^{-1}\, x = \left(\op{A}^{-1}\right)^\half\, x\,,
	&&
	\forall\, x\in\dom(\op{J}) = \cX\,,
	\nn\\
	&&&
	\ran(\op{J}) = \dom(\opAsqrt) = \cX_\half\,,
	\label{eq:op-A-ass-2d}
	\\
	& \left\| \op{J}\, x \right\|_\half = \|x\|\,,
	&&
	\forall\, x\in\dom(\op{J}) = \cX\,.
	\label{eq:op-A-ass-2e}
\end{align}
%
% \item Operator $-\op{A}$ is the generator of a $C_0$-semigroup of bounded linear operators on $\cX$. %, while $\op{A}$ is not.
\end{enumerate}
\end{lemma}

%%% if falsed proof %%%
\if{false} 
\begin{proof}
{\em (i)}
Fix $x,\, \xi\in\dom(\op{A}) = \cX_0$, and note that $\partial^2 x,\, \partial^2 \xi\in\cX$. By definition of the standard {\ltwo} inner-product on $\cX$,
\begin{align}
	\langle \xi,\, \op{A}\, x \rangle
	& = \int_\Lambda \xi(\zeta)\, (\op{A}\, x)(\zeta)\, d\zeta
	= - \int_\Lambda \xi(\zeta)\, \partial^2 x(\zeta)\, d\zeta
	\nn\\
	& = - \xi(\zeta)\, \partial x(\zeta) \biggl|_\Lambda + \int_\Lambda \partial \xi(\zeta) \, \partial x(\zeta)\, d\zeta
	\nn\\
	& = - \xi(\zeta)\, \partial x(\zeta) \biggl|_\Lambda + \partial \xi(\zeta)\, x(\zeta) \biggl|_\Lambda 
					- \int_\Lambda \partial^2 \xi(\zeta)\, x(\zeta)\, d\zeta
	\nn\\
	& = \partial \xi(L)\, x(L) - \partial\xi(0)\,  x(0) - \xi(L) \, \partial x(L) + \xi(0) \, \partial x(0)
					+ \int_\Lambda \left[ -\partial^2 \xi(\zeta) \right]\, x(\zeta)\, d\zeta
	\nn\\
	& = \langle \op{A}\, \xi,\, x \rangle\,,
	\label{eq:op-Lap-properties-1}
\end{align}
as $x(0) = 0 =  x(L)$ and $\xi(0) = 0 = \xi(L)$. That is, the adjoint of operator $\op{A}$ on $\cX$ is $\op{A}'\, x \doteq \op{A} \, x$ with $\quad\dom(\op{A}') \doteq \dom(\op{A})$, so that $\op{A}$ is self-adjoint on $\cX$ as per \er{eq:op-A-ass-1a}. In order to demonstrate positivity \er{eq:op-A-ass-1b} of $\op{A}$, recall that $\partial x\in\cX$ by inspection of \er{eq:cX-0}, so that $\|\partial x\|<\infty$. Suppose that $\|\partial x\| = 0$, i.e. $\partial x=0$ a.e. on $\Lambda$. By inspection of the boundary data in \er{eq:cX-0}, $x(\zeta) = x(0) + \int_0^\zeta \partial x(\rho)\, d\rho = 0$ for all $\zeta\in\Lambda$. That is, $\|x\| = 0$. So, in order to demonstrate positivity of $\op{A}$, it is sufficient to restrict attention to those $x\in\dom(\op{A})$ for which $\|\partial x\| > 0$. In that case, by inspection of the third equality in \er{eq:op-Lap-properties-1}, setting $\xi = x$ yields
\begin{align}
	\langle x,\, \op{A}\, x\rangle
	& = \int_\Lambda \partial x(\zeta)\, \partial x(\zeta) \, d\zeta
	= \| \partial x \|^2 > 0\,.
	\nn
\end{align}
That is, \er{eq:op-A-ass-1b} holds. Assertion \er{eq:op-A-ass-1c} follows by Example A.4.26 of \cite{CZ:95}. In particular, by inspection, $\op{A}^{-1}$ is an integral operator with kernel $\hat A\in\Ltwo(\Lambda\times\Lambda;\R)$, which implies that $\op{A}^{-1}\in\mathcal{L}(\cX)$. Consequently, $\op{A}$ is also closed. By definition, $\ran(\op{A}^{-1}) = \dom(\op{A}) = \cX_0$. Similarly, \er{eq:op-A-ass-1a} implies that $\dom(\op{A}^{-1}) = \ran(\op{A}) = \cX$. Hence, \er{eq:op-A-ass-1c} holds.

{\em (ii)} As $\op{A}$ defines a self-adjoint and positive operator on a real Hilbert space $\cX$ as per assertion (i), operator $\op{A}$ has a unique, self-adjoint, positive square root $\opAsqrt$, see (for example) Theorems 9-11 of \cite{B:68} or Lemma A.3.73 of \cite{CZ:95}. (Note that the self-adjoint property of $\opAsqrt$ follows immediately from the self-adjoint property of $\op{A}$ and uniqueness of $\opAsqrt$. In particular, $\op{A} = \op{A}' = (\opAsqrt\, \opAsqrt)' = (\opAsqrt)' \, (\opAsqrt)'$. Consequently, $(\opAsqrt)'$ is also a square-root of $\op{A}$, and so by uniqueness, $\opAsqrt = (\opAsqrt)'$ with the same domain.) The fact that $\opAsqrt$ is boundedly invertible follows by Example A.4.3 of \cite{CZ:95}. 
Assertions \er{eq:op-A-ass-2a} and \er{eq:op-A-ass-2b} follow by Lemma A.3.73 of \cite{CZ:95}. 
Assertion \er{eq:op-A-ass-2c} follows by applying the self-adjoint property of $\opAsqrt$ along with the definition \er{eq:cX-half} of $\|\cdot\|_\half$, and assertions \er{eq:op-A-ass-1b} and \er{eq:op-A-ass-2b}. In particular,
\begin{align}
	\|x\|_\half^2
	& % = \| \opAsqrt\, x \|^2
	= \langle \opAsqrt\, x,\, \opAsqrt\, x \rangle
	 % = \langle x,\, \opAsqrt\, \opAsqrt\, x \rangle
	= \langle x,\, \op{A}\, x\rangle
	= \|\partial x\|^2.
	\nn
\end{align}
Assertion \er{eq:op-A-ass-2d} follows by Example A.4.3 of \cite{CZ:95}. In particular, as $\op{A}^{-1}\in\mathcal{L}(\cX)$ is self-adjoint and positive, it has a well-defined square-root $(\op{A}^{-1})^\half = (\opAsqrt)^{-1} = \op{J}\in\mathcal{L}(\cX)$. This also implies that $\opAsqrt$ is closed. As $\op{J} \, \opAsqrt = \op{I}_{\half}$ and $\opAsqrt\, \op{J} = \op{I}$ (where $\op{I}_\half$ and $\op{I}$ denotes  the identity operator on $\cX_\half$ and $\cX$ respectively), $\cX = \dom(\op{J}) = \ran(\opAsqrt)$ and $\cX_\half = \ran(\op{J}) = \dom(\opAsqrt)$. (See also \cite{B:68}.)
Assertion \er{eq:op-A-ass-2e} follows by definition \er{eq:cX-half} of $\|\cdot\|_\half$, with
$\|x\|^2
	= \langle x,\, x\rangle
	% = \langle \op{I}\, x,\, \op{I}\, x\rangle
	= \langle \opAsqrt \op{J}\, x,\, \opAsqrt\op{J}\, x\rangle
	% = \langle \op{J}\, x,\, \op{J}\, x\rangle_\half
	= \|\op{J}\, x\|_\half
$.
\end{proof}

\fi
%%% end of if falsed proof %%%

%		Eigenvalues.

\begin{lemma}
\label{lem:eigenvalues}
Operator $\op{A}$ of \er{eq:op-A} has countably infinite simple eigenvalues given by $\{\lambda_n\}_{n=1}^\infty$, where eigenvalue $\lambda_n$ corresponds to eigenvector $\varphi_n\in\op{B}$ of \er{eq:basis-Riesz} (or equivalently $\tilde\varphi_n\in\optilde{B}$) and
\begin{align}
	& \lambda_n \doteq (\ts{\frac{n\, \pi}{L}})^2\,.
	\label{eq:eigenvalues}
\end{align}
Similarly, the square root $\opAsqrt$ of operator $\op{A}$ has countably infinite simple eigenvalues given by $\{\sqrt{\lambda_n}\}_{n=1}^\infty$, where eigenvalue $\sqrt{\lambda_n}$ corresponds to eigenvector $\varphi_n\in\op{B}$ (or equivalently in $\tilde\varphi_n\in\optilde{B}$) and $\lambda_n$ is as per \er{eq:eigenvalues}.
\end{lemma}

%%% if falsed proof %%%
\if{false}

\begin{proof}
The first assertion concerning the eigenvalues of $\op{A}$ is well known, c.f. Example A.4.26 of \cite{CZ:95}. In order to prove the second assertion, define $x_n\doteq \opAsqrt\, \varphi_n\in\cX_\half$ and note that $x_n\ne 0$ as $\opAsqrt$ is boundedly invertible by Lemma \ref{lem:op-A-properties} and $0\ne \varphi_n\in\cX_0$. Also note that as $x_n\in\cX_\half$, the first assertion implies that $\opAsqrt\, x_n = \op{A}\, \varphi_n = \lambda_n\, \varphi_n\in\cX_0$. Hence, $x_n\in\cX_0 = \dom(\op{A})$. So, applying $\op{A}$ to $x_n$ and noting that $\op{A}$ and $\opAsqrt$ commute,
\begin{align}
	\op{A}\, x_n
	& = \op{A} \, ( \opAsqrt\, \varphi_n ) = \opAsqrt\, \op{A}\, \varphi_n = \opAsqrt\, (\lambda_n \, \varphi_n) = \lambda_n\, x_n\,.
	\nn
\end{align}
That is, $x_n$ is an eigenvector of $\op{A}$ corresponding to eigenvalue $\lambda_n$. Hence, by the first assertion, there exists an $\alpha_n\in\R$, $\alpha_n\ne 0$, such that $(\opAsqrt\, \varphi_n =\, ) \, x_n = \alpha_n\, \varphi_n$. That is, $\varphi_n$ is an eigenvector of $\opAsqrt$ corresponding to eigenvalue $\alpha_n\in\R_{>0}$, where positivity follows from the fact that $\opAsqrt$ is self-adjoint and positive. Consequently,
\begin{align}
	\op{A}\, \varphi_n
	& = \opAsqrt\, \opAsqrt\, \varphi_n = \opAsqrt\, (\alpha_n\, \varphi_n) = \alpha_n\, \opAsqrt\, \varphi_n = \alpha_n^2 \, \varphi_n\,.
	\nn
\end{align}
That is, $\alpha_n^2$ is an eigenvalue of $\op{A}$ corresponding to eigenvector $\varphi_n$. Hence, by the first assertion, $\alpha_n^2 = \lambda_n$.
\end{proof}

\fi
%%% end of if falsed proof %%%

\begin{lemma} % [e.g. \cite{CZ:95}]
\label{lem:basis-Riesz}
$\op{B}$ and $\optilde{B}$ of \er{eq:basis-Riesz} are orthonormal Riesz bases for $\cX$ and $\cX_\half$ respectively.
\end{lemma}

%%% if falsed proof %%%
\if{false}

\begin{proof}
The fact that $\op{B}$ is a Riesz basis for $\cX$ is well-known, see for example \cite{CZ:95}. It remains to demonstrate that $\optilde{B}$ is a Riesz basis for $\cX_\half$. To this end, first note that $\optilde{B}\subset\cX_0\subset\cX_\half$. With $m,n\in\N$, Lemma \ref{lem:eigenvalues} implies that
\begin{align}
	& \langle \tilde\varphi_m,\, \tilde\varphi_n \rangle_\half
	= \langle \opAsqrt\, \tilde\varphi_m,\, \opAsqrt\, \tilde\varphi_n\rangle
	= \langle \tilde\varphi_m,\, \op{A}\, \tilde\varphi_n \rangle
	= \lambda_n \, \langle \tilde\varphi_m,\, \tilde\varphi_n \rangle
	= \lambda_n \, \ts{\frac{2\, L}{(n\, \pi)^2}} \, (\ts{\frac{L}{2}})\, \langle \, \varphi_m,\, \varphi_n \rangle
	= \delta_{mn}\,,
	\nn
\end{align}
where $\delta_{mn}$ denotes the Kronecker delta. Consequently, as $m,n\in\N$ are arbitrary, $\optilde{B}\subset \cX_\half$ is orthonormal. It remains to show that $\optilde{B}$ spans $\cX_\half$. To this end, fix any $x\in\cX_\half$ and define
$\tilde e_N,\, e_N\in\cX_\half$ for each $N\in\N$ by
\begin{align}
	\tilde e_N
	& \doteq x - \sum_{n=1}^N \langle x,\, \tilde\varphi_n \rangle_\half\, \tilde\varphi_n\,,
	\label{eq:basis-tilde-e}
	\\
	e_N 
	& \doteq x - \sum_{n=1}^N \langle x,\, \varphi_n \rangle\, \varphi_n\,.
	\label{eq:basis-breve-e}
\end{align}
Hence, for each $N,M\in\N$, $N>M$, \er{eq:basis-tilde-e} implies that
\begin{align}
	& \| \tilde e_N - \tilde e_M \|_\half^2
	= \left\| - \!\!\! \sum_{m=M+1}^N \langle x,\, \tilde\varphi_m\rangle_\half\, \tilde\varphi_m \right\|_\half^2
	= \left\langle \sum_{m=M+1}^N \langle x,\, \tilde\varphi_m\rangle_\half\, \tilde\varphi_m,\, 
				\sum_{n=M+1}^N \langle x,\, \tilde\varphi_n\rangle_\half\, \tilde\varphi_n \right\rangle_\half
	\nn\\
	& = \sum_{m=M+1}^N \langle x,\, \tilde\varphi_m\rangle_\half\, \left\langle \varphi_m,\, 
				\sum_{n=M+1}^N \langle x,\, \tilde\varphi_n\rangle_\half\, \tilde\varphi_n \right\rangle_\half
	= \sum_{m=M+1}^N \langle x,\, \tilde\varphi_m\rangle_\half^2\, 
						\langle \tilde\varphi_m,\, \tilde\varphi_m\rangle_\half
	= \sum_{m=M+1}^N \langle x,\, \tilde\varphi_m\rangle_\half^2
	\nn\\
	& = \sum_{m=M+1}^N \langle x,\, \op{A}\, \tilde\varphi_m \rangle^2
	= \sum_{m=M+1}^N \lambda_m^2 \, \langle x,\, \tilde\varphi_m \rangle^2
	\le \sum_{m=M+1}^\infty \lambda_m^2 \, \langle x,\, \tilde\varphi_m \rangle^2\,,
	\label{eq:basis-Cauchy}
\end{align}
in which the last equality follows by virtue of the fact that $\op{A}$ has distinct positive eigenvalues $\{\lambda_n\}_{n=1}^\infty$ corresponding to eigenvectors in $\optilde{B}$ (or equivalently $\op{B}$). In order to evaluate this last sum in \er{eq:basis-Cauchy}, note that Lemma \ref{lem:eigenvalues} implies that $\opAsqrt$ inherits distinct positive eigenvalues $\{\sqrt{\lambda_n}\}_{n=1}^\infty$  corresponding to the same eigenvectors $\optilde{B}$ (or $\op{B}$) of $\op{A}$. Hence, as $\opAsqrt\, x\in\cX$ (by definition of $x\in\cX_\half$),
\begin{align}
	\infty & > \sum_{m=1}^\infty \langle \opAsqrt\, x,\, \varphi_m \rangle^2
	= \sum_{m=1}^\infty \langle x,\, \opAsqrt\, \varphi_m \rangle^2
	= \sum_{m=1}^\infty \langle x,\, \sqrt{\lambda_m}\, \varphi_m \rangle^2
	= \sum_{m=1}^\infty \lambda_m\, \ts{\frac{(n\, \pi)^2}{2\, L}} \, \ts{\frac{2}{L}}\, \langle x,\, \tilde\varphi_m \rangle^2
	= \sum_{m=1}^\infty \lambda_m^2 \, \langle x,\, \tilde\varphi_m \rangle^2\,,
	\nn
\end{align}
where the last equality uses \er{eq:eigenvalues}. It immediately follows that
\begin{align}
	0 & = \lim_{M\rightarrow\infty} \sum_{m=M+1}^\infty \lambda_m^2 \, \langle x,\, \tilde\varphi_m \rangle^2\,.
	\label{eq:basis-sum}
\end{align}
Fix any $k\in\N$ and set $N\doteq M + k$ in \er{eq:basis-Cauchy}. Taking the limit as $M\rightarrow\infty$, and applying \er{eq:basis-sum} yields that 
$
	0 = \lim_{M\rightarrow\infty} \| \tilde e_{M+k} - \tilde e_M \|_\half^2
$.
As $k\in\N$ is arbitrary, it follows that $\{\tilde e_N \}_{N=1}^\infty$ defines a Cauchy sequence in $\cX_\half$. Completeness of (Hilbert space) $\cX_\half$ then implies that there exists a $\tilde e_\infty\in\cX_\half$ such that
\begin{align}
	\lim_{M\rightarrow\infty} \tilde e_M
	& = \tilde e_\infty\in\cX_\half\,.
	\label{eq:basis-tilde-e-infty}
\end{align}
Recall by definition \er{eq:basis-breve-e} that $e_N\in\cX_\half\subset\cX$ for each $N\in\N$. So, the fact that $\op{B}$ forms a Riesz basis for $\cX$ implies (by definition) that there exists a $e_\infty\in\cX$ such that
\begin{align}
	\lim_{N\rightarrow\infty} e_N
	& = e_\infty \doteq 0\in\cX\,.
	\label{eq:basis-e-infty}
\end{align}
By further inspection of \er{eq:basis-tilde-e} and \er{eq:basis-breve-e}, and application of \er{eq:basis-Riesz}, note that
\begin{align}
	\tilde e_N
	& = x - \sum_{n=1}^N \langle x,\ \op{A}\, \tilde\varphi_n\rangle\, \tilde\varphi_n
	=  x - \sum_{n=1}^N \lambda_n\, \langle x,\ \tilde\varphi_n\rangle\, \tilde\varphi_n
	= x - \sum_{n=1}^N \lambda_n^2\, \ts{\frac{2\, L}{(n\, \pi)^2}} \, \ts{\frac{L}{2}}\, \langle x,\ \varphi_n\rangle\, \varphi_n
	= e_N\,.
	\nn
\end{align}
So, applying the triangle inequality,
\begin{align}
	\| \tilde e_\infty - e_\infty \|
	& \le \| \tilde e_\infty - \tilde e_N \| + \| \tilde e_N - e_N\| + \| e_N - e_\infty\|
	= \| \op{J}\, \opAsqrt\, (\tilde e_\infty - \tilde e_N) \| + \| e_N - e_\infty\|
	\nn\\
	& \le  \| \op{J} \|\, \|\tilde e_\infty - \tilde e_N \|_\half +  \| e_N - e_\infty\|\,.
	\nn
\end{align}
Taking the limit as $N\rightarrow\infty$, applying \er{eq:basis-tilde-e-infty}, \er{eq:basis-e-infty}, and recalling \er{eq:basis-tilde-e}, yields that
\begin{align}
	0 & \ge \| \tilde e_\infty - e_\infty \|
	= \left\| \lim_{M\rightarrow\infty} \left\{  x - \sum_{m=1}^M \langle x,\, \tilde\varphi_m \rangle_\half\, \tilde\varphi_m \right\} - 0 \right\|
	= \left\| x - \sum_{m=1}^\infty \langle x,\, \tilde\varphi_m \rangle_\half\, \tilde\varphi_m \right\|\,.
	\nn
\end{align}
That is, $x = \sum_{m=1}^\infty \langle x,\, \tilde\varphi_m \rangle_\half\, \tilde\varphi_m$, as required. Finally, it is straightforward to show that for any $\alpha_n\in\R$, $n\in\N$,
$
	\left\| \sum_{n=1}^N \alpha_n\, \tilde\varphi_n \right\|_\half^2
	= \sum_{n=1}^N |\alpha_n|^2
$.
Hence, appealing to Definition 2.3.1 of \cite{CZ:95}, $\cB$ is an orthonormal Riesz basis for $\cX_\half$.
\end{proof}

\fi
%%% end of if falsed proof %%%

%		Properties of operator $\op{I}_\mu$.

\begin{lemma}
\label{lem:op-I-mu-properties}
The following properties hold on $\cX$ for any $\mu\in\R_{>0}$:
\begin{enumerate}[(i)]
\item Operator $\op{I}_\mu$ of \er{eq:op-I-mu} is bounded, linear, self-adjoint, positive, with
\begin{align}
	& \op{I}_\mu\, x\in\dom(\op{A}) = \cX_0
	&& \forall \ x\in\dom(\op{I}_\mu) = \cX\,,
	\label{eq:op-I-mu-ass-1a}
	\\
	& \op{I}_\mu\, x = \int_\Lambda I_\mu(\cdot,\zeta)\, x(\zeta)\, d\zeta\,, 
	&& I_\mu(\lambda,\zeta) 
	\doteq \frac{1}{\mu\, \sinh(\ts{\frac{L}{\mu}})} \left\{ \begin{aligned}
				\sinh(\ts{\frac{\lambda}{\mu}})\, \sinh( \ts{\frac{L-\zeta}{\mu}} )\,,
				& \quad 0\le \lambda \le \zeta \le L\,,
				\\
				\sinh(\ts{\frac{\zeta}{\mu}})\, \sinh( \ts{\frac{L-\lambda}{\mu}} )\,,
				& \quad 0\le \zeta \le \lambda \le L\,,
	\end{aligned} \right.
	\nn\\
	&&& \forall \ x\in\dom(\op{I}_\mu) = \cX\,.
	\label{eq:op-I-mu-ass-1b}
	% \\
	% & \op{I}_\mu\, x = \ts{\frac{1}{\mu^2}}\, \op{R}_{-\op{A}}(\ts{\frac{1}{\mu^2}})\, x
	% && \forall\ x\in\dom(\op{I}_\mu) = \cX\,,
	% \label{eq:op-I-mu-ass-1c}
\end{align}
% where $\op{R}_{-\op{A}}(\lambda)$, $\lambda\in\R_{\ge 0}$, is the resolvent of $-\op{A}$, see \cite{CZ:95,P:83}.

\item Operator $\op{I}_\mu$ has a unique, bounded, linear, self-adjoint, and positive square root $\op{I}_\mu^\half$, with
\begin{align}
	& \op{I}_\mu^\half\, x\in\dom(\opAsqrt) = \cX_\half
	&& \forall \ x\in\dom(\op{I}_\mu^\half) = \cX\,,
	\label{eq:op-I-mu-ass-2a}
	\\
	& \op{I}_\mu^\half\, \op{I}_\mu^\half\, x = \op{I}_\mu\, x
	&&\forall \ x\in\dom(\op{I}_\mu^\half) = \dom(\op{I}_\mu) = \cX\,,
	\label{eq:op-I-mu-ass-2b}
\end{align}

\item Operators $\op{I}_\mu$, $\op{I}_\mu^\half$, $\op{A}$, and $\op{A}^\half$ commute, with
\begin{align}
	& \op{I}_\mu^\half\, \op{A}^\half\, x = \op{A}^\half\, \op{I}_\mu^\half\, x\,, \quad
	\op{I}_\mu\, \opAsqrt\, x = \opAsqrt\, \op{I}_\mu\, x
	&& \forall \ x\in\dom(\opAsqrt) = \cX_\half\,.
	\label{eq:op-I-mu-ass-3a}
	\\
	& \op{I}_\mu\, \op{A}\, x = \op{A}\, \op{I}_\mu\, x\,, \quad
	\op{I}_\mu^\half\, \op{A}\, x = \op{A}\, \op{I}_\mu^\half\, x
	&& \forall \ x\in\dom(\op{A}) = \cX_0\,,
	\label{eq:op-I-mu-ass-3b}
\end{align}

\item Selected compositions of operators $\op{I}_\mu$, $\op{I}_\mu^\half$, $\op{A}$, and $\op{A}^\half$ define bounded linear operators, with
\begin{align}
	& \op{A}\, \op{I}_\mu, \, \opAsqrt\, \op{I}_\mu^\half \in\bo(\cX)\,,
	\label{eq:op-I-mu-ass-4a}
	\\
	& \op{A}\, \op{I}_\mu\in\bo(\cX_\half)\,,
	\label{eq:op-I-mu-ass-4b}
	\\
	& \opAsqrt\, \op{I}_\mu^\half\, \opAsqrt\in\bo(\cX_\half;\cX)\,.
	\label{eq:op-I-mu-ass-4c}
\end{align}
\end{enumerate}
\end{lemma}
\begin{proof}
{\em (i)} Fix any $x\in\dom(\op{A}) = \cX_0$. Consequently, $\op{A}\, x\in\cX$, and 
\begin{align}
	\langle x,\, (\op{I} + \mu^2\, \op{A})\, x \rangle
	& = \|x\|^2 + \mu^2\, \langle x,\, \op{A}\, x \rangle
	\ge \|x\|^2\,,
	% \label{eq:coercive}
	\nn
\end{align}
where the inequality follows by positivity of $\op{A}$, see assertion \er{eq:op-A-ass-1b} of Lemma \ref{lem:op-A-properties}. That is, $\op{I} + \mu^2\, \op{A}$ is both positive and coercive \cite[Definition A.3.71, p.606]{CZ:95}. It is also self-adjoint by \er{eq:op-A-ass-1a}. Hence, $\op{I} + \mu^2\, \op{A}$ is boundedly invertible, see for example \cite[Example A.4.2, p.609]{CZ:95} and \cite[Problem 10, p.535]{K:78}. In particular, $\op{I}_\mu \doteq (\op{I} + \mu^2\, \op{A})^{-1}\in\bo(\cX)$. In order to show that $\op{I}_\mu$ is also self-adjoint and positive (but not coercive), fix any $y,\, \eta\in\cX$, and define $x,\, \xi\in\cX_0$ by $x\doteq \op{I}_\mu\, y$ and $\xi = \op{I}_\mu\, \eta$. As $\op{I} + \mu^2\, \op{A}$ is self-adjoint, 
$
	\langle y,\, \op{I}_\mu\, \eta \rangle
	= \langle (\op{I} + \mu^2\, \op{A}) \, x,\, \xi \rangle
	= \langle x, \, (\op{I} + \mu^2\, \op{A}) \, \xi \rangle
	= \langle \op{I}_\mu\, y,\, \eta \rangle
$.
As $y,\, \eta\in\cX$ are arbitrary, it follows that $\op{I}_\mu$ is also self-adjoint. Furthermore, with $y = \eta$,
$
	\langle y,\, \op{I}_\mu\, y \rangle
	= \langle (\op{I} + \mu^2\, \op{A}) \, x,\, x \rangle \ge \|x\|^2 = \|\op{I}_\mu\, y\|^2
$.
As $I_\mu$ is invertible, the right-hand side is zero if and only if $y=0$. Hence, $\op{I}_\mu$ is positive. However, $\op{I}_\mu$ is not coercive, as it has eigenvalues arbitrarily close to zero. For example, select $y \doteq \psi_n$, with $\psi_n$ is as per \er{eq:basis-Riesz}. Note that $\|\psi_n\|=1$. Applying Lemma \ref{lem:eigenvalues}, it is straightforward to show that
$
	\langle \psi_n,\, \op{I}_\mu\, \psi_n \rangle
	= \frac{1}{1 + \mu^2\, \lambda_n} \, \|\psi_n\|^2 = \frac{1}{1 + \mu^2 \, (\ts{\frac{n\, \pi}{L}})^2}\,  \|\psi_n\|^2
$
for all $n\in\N$. Note in particular that the coefficient on the right-hand side may be made arbitrarily small for sufficiently large $n\in\N$. Hence, $\op{I}_\mu$ cannot be coercive.
% It is immediate that $\op{I}_\mu$ is also self-adjoint and positive (but not coercive).

It remains to be shown that \er{eq:op-I-mu-ass-1a} and \er{eq:op-I-mu-ass-1b} hold. Fix any $x\in\cX$. By noting that $\op{I} - \op{I}_\mu\in\bo(\cX)$, the definition \er{eq:op-I-mu} of $\op{I}_\mu$ implies that
\begin{align}
	\infty
	& > \ts{\frac{1}{\mu^2}}\, \|(\op{I} - \op{I}_\mu)\, x\|
	= \ts{\frac{1}{\mu^2}}\, \| ( [\op{I} + \mu^2\, \op{A}]\, \op{I}_\mu - \op{I}_\mu )\, x \| = \|\op{A}\, \op{I}_\mu \, x \|\,.
	\label{eq:op-A-I-mu-bound}
\end{align}
Hence, $\op{I}_\mu\, x\in\dom(\op{A}) = \cX_0$ for any $x\in\cX$, so that \er{eq:op-I-mu-ass-1a} holds.
Given the kernel $I_\mu$ as defined in \er{eq:op-I-mu-ass-1b}, define the operator $\ophat{I}_\mu$ by
\begin{align}
	\ophat{I}_\mu\, x
	& \doteq \int_\Lambda I_\mu(\cdot,\zeta)\, x(\zeta) \, d\zeta\,,
	\quad
	\dom(\ophat{I}_\mu) \doteq \cX\,,
	\label{eq:ophat-I-mu}
\end{align}
and note that $\ophat{I}_\mu\in\bo(\cX)$ by inspection. Fix any $x\in\cX$ and define $y\doteq \op{I}_\mu\, x \in \cX_0$. Hence, $\op{A}\, y\in\cX$, and
$
	\ophat{I}_\mu\, \op{A}\, y
	= \int_\Lambda I_\mu(\cdot,\zeta)\, [-\partial^2\, y(\zeta)]\, d\zeta
	= - \int_\Lambda [\partial_2^2 I_\mu(\cdot,\zeta)] \, y(\zeta)\, d\zeta
$, 
where $\partial_2^2 I_\mu(\lambda,\zeta) \doteq \ts{\frac{1}{\mu^2}} \, I_\mu(\lambda,\zeta) - \ts{\frac{1}{\mu^2}} \, \delta(\lambda-\zeta)$ is the second weak derivative of $I_\mu(\lambda,\cdot)$, $\lambda\in\Lambda$ fixed. Note in particular that the boundary conditions $y(0) = 0 = y(L)$ have been used here. Consequently, 
$
	\ophat{I}_\mu\, \op{A}\, y
	= - \ts{\frac{1}{\mu^2}} \int_\Lambda [I_\mu(\cdot,\zeta)  -  \delta(\cdot-\zeta) ]\,  y(\zeta) \, d\zeta
	= -\ts{\frac{1}{\mu^2}} ( \ophat{I}_\mu\, y - y)
$,
so that $\ophat{I}_\mu \, (\op{I} + \mu^2\, \op{A}) \, y = y$. Recalling the definition of $y$, it follows immediately that $\ophat{I}_\mu\, x = \op{I}_\mu\, x$. As $x\in\cX$ is arbitrary, assertion \er{eq:op-I-mu-ass-1b} follows. 

%% IFFALSED resolvent proof.
\if{false}

Recalling the definition of resolvent $\op{R}_{-\op{A}}(\ts{\frac{1}{\mu^2}})\in\bo(\cX)$ for $\mu\in\R_{>0}$, see for example \cite[Definition A.4.1,p.608]{CZ:95} or \cite[p.8]{P:83},
\begin{align}
	\ts{\frac{1}{\mu^2}} \, \op{R}_{-\op{A}} (\ts{\frac{1}{\mu^2}})
	& = \ts{\frac{1}{\mu^2}} \, (\ts{\frac{1}{\mu^2}} \, \op{I} - (-\op{A}))^{-1}
	= (\op{I} + \mu^2\, \op{A})^{-1} = \op{I}_\mu\,.
	\nn
\end{align}
That is, assertion \er{eq:op-I-mu-ass-1c} holds.

\fi
%% END OF IFFALSED resolvent proof.

{\em (ii)} The existence of a unique, bounded linear, self-adjoint, and positive square root $\op{I}_\mu^\half$ follows (for example) by \cite[Theorem 4]{B:68}. (Alternatively, see \cite[Lemma A.3.73, p.606]{CZ:95}.)

\if{false}
Operator $\op{I}+\mu^2\, \op{A}$ with domain $\cX_0$ is self-adjoint, positive, and closed by Lemma \ref{lem:op-A-properties}. Hence, Lemma A.3.73 of \cite{CZ:95} implies that it has a unique, self-adjoint, positive square root $(\op{I} + \mu^2\, \op{A})^\half \equiv \op{I}_\mu^{-\half}$ with domain $\cX_\half$. Consequently, $\ran(\op{I}_\mu^\half) = \cX_\half$, as per assertion \er{eq:op-I-mu-ass-2a}. Assertion \er{eq:op-I-mu-ass-2b} follows similarly.
\fi

{\em (iii)} Fix $x\in\cX_\half$. By definition, $\opAsqrt\, x\in\cX = \dom(\op{I}_\mu^\half)$, and
\begin{align}
	\op{I}_\mu\, \opAsqrt\, x
	& = (\op{I} + \mu^2\, \op{A})^{-1}\, \opAsqrt\, x = (\op{J} \, [ \op{I} + \mu^2 \, \op{A} ])^{-1} \, x
	= (\op{J} + \mu^2\, \op{J}\, \opAsqrt\, \opAsqrt)^{-1}\, x
	\nn\\
	& = (\op{J} + \mu^2 \, \opAsqrt\, \opAsqrt\, \op{J})^{-1}\, x
	= ([\op{I} + \mu^2\, \op{A}]\, \op{J})^{-1}\, x
	= \opAsqrt\, (\op{I} + \mu^2\, \op{A})^{-1}\, x
	= \opAsqrt\, \op{I}_\mu\, x\,,
	\nn
\end{align}
so that right-hand equality in assertion \er{eq:op-I-mu-ass-3a} holds. The remaining equalities follow similarly, with $x\in\cX_0$ yielding assertion \er{eq:op-I-mu-ass-3b}.

{\em (iv)} The first assertion in \er{eq:op-I-mu-ass-4a} follows from the proof of {\em (i)} above. In particular, $\op{I}-\op{I}_\mu\in\bo(\cX)$ and \er{eq:op-A-I-mu-bound} imply that $\|\op{A}\, \op{I}_\mu\| = \ts{\frac{1}{\mu^2}} \, \|\op{I} - \op{I}_\mu\|<\infty$, as required.
In order to prove the second assertion of \er{eq:op-I-mu-ass-4a}, note that for any $x\in\cX_\half$, \er{eq:op-I-mu-ass-3b} implies that 
\begin{align}
	\|\opAsqrt\, \op{I}_\mu^\half\, x \|^2
	% & = \langle \opAsqrt\, \op{I}_\mu^\half\, x,\, \opAsqrt\, \op{I}_\mu^\half\, x \rangle
	% \nn\\
	& =  \langle x,\, \op{I}_\mu^\half\,  \op{A} \, \op{I}_\mu^\half\, x \rangle
	% = \langle x,\, \opAsqrt\, \op{I}_\mu^\half\, \opAsqrt\, \op{I}_\mu^\half\, x \rangle
	= \langle x,\, \op{A}\, \op{I}_\mu\, x\rangle
	\le \|x\|\, \|\op{A}\, \op{I}_\mu\, x\|
	\le \|\op{A}\, \op{I}_\mu\| \|x\|^2\,.
	\nn
\end{align}
Hence, the restriction $\op{R}_\mu$ of $\opAsqrt\, \op{I}_\mu^\half:\cX\mapsinto\cX$ to the domain $\cX_\half\subset\cX$ is bounded and linear on that domain. However, as $\ol{\cX_\half} = \cX$, $\op{R}_\mu$ can be uniquely extended to an operator $\op{E}_\mu\in\bo(\cX)$ (see for example \cite[Theorem 2.7-11, p.100]{K:78}) that satisfies $\op{E}_\mu\, x = \op{R}_\mu\, x = \opAsqrt\, \op{I}_\mu^\half\, x$ for all $x\in\cX_\half$. Fix $y\in\cX$. Hence, for any $x\in\cX_\half$,
$
	\|\op{E}_\mu \, y - \opAsqrt\, \op{I}_\mu^\half\, y \|
	\le \|\op{E}_\mu\, x - \opAsqrt\, \op{I}_\mu^\half\, x \| + \|(\op{E}_\mu - \opAsqrt\, \op{I}_\mu^\half) \, (y-x) \|
	\le \|\op{E}_\mu - \opAsqrt\, \op{I}_\mu^\half \|\, \|y-x\|
$. 
Consequently, as $x\in\cX_\half$ is arbitrary and $\ol{\cX_\half} = \cX$, implies that
$
	\|\op{E}_\mu \, y - \opAsqrt\, \op{I}_\mu^\half\, y \|
	\le \|\op{E}_\mu - \opAsqrt\, \op{I}_\mu^\half \|\, \inf_{x\in\cX_\half}  \|y-x\| = 0
$.
As $y\in\cX$ is arbitrary, $\op{E}_\mu \equiv \opAsqrt\op{I}_\mu^\half$. Recalling that $\op{E}_\mu\in\bo(\cX)$ completes the proof of assertion \er{eq:op-I-mu-ass-4a}.

In order to prove assertion \er{eq:op-I-mu-ass-4b}, note that $\opAsqrt$ and $\op{I}_\mu$ commute on $\cX_\half$ by \er{eq:op-I-mu-ass-3a}. Hence, with $x\in\cX_\half$. 
$
	\|\op{A}\, \op{I}_\mu\, x\|_\half
	= \|\opAsqrt\, \op{I}_\mu\, \opAsqrt\, x\|_\half
	= \|\op{A}\, \op{I}_\mu\, (\opAsqrt\, x)\| 
	\le \|\op{A}\, \op{I}_\mu\| \, \|\opAsqrt\, x\|
	= \|\op{A}\, \op{I}_\mu\| \, \|x\|_\half
$.
As $\op{A}\, \op{I}_\mu\in\bo(\cX)$ by \er{eq:op-I-mu-ass-4a}, assertion \er{eq:op-I-mu-ass-4b} immediately follows.

Finally, in order to prove assertion \er{eq:op-I-mu-ass-4c}, note that \er{eq:op-I-mu-ass-4a} implies that $\| \opAsqrt\, \op{I}_\mu^\half\, \opAsqrt \, x\| \le \|\opAsqrt\, \op{I}_\mu^\half\| \, \|\opAsqrt \, x\| = \|\opAsqrt\, \op{I}_\mu^\half\| \, \|x\|_\half $ for any $x\in\cX_\half$. Consequently,
$
	\sup_{x\in\cX_\half, \|x\|_\half\ne 0} 
	\frac{\| \opAsqrt\, \op{I}_\mu^\half\, \opAsqrt \, x\|}{\|x\|_\half} \le \|\opAsqrt\, \op{I}_\mu^\half\| < \infty
$
as required.
%, as required.
\end{proof}

%%
%%		Properties of Riesz-spectral operators.
%%

\section{Riesz-spectral operators}
\label{app:Riesz}
It is useful to consider self-adjoint operators of the form
\begin{align}
	\opbreve{F}
	\, x
	& \doteq \sum_{n=1}^\infty f_n\, \langle x,\, \tilde\varphi_n \rangle_\half\, \tilde\varphi_n\,,
	\quad
	\dom(\opbreve{F})
	\doteq \left\{ x\in\cX_\half\, \biggl|\, \opbreve{F}\, x\in\cX_\half \right\},
	\label{eq:op-Riesz}
%	\\
%	\dom(\opbreve{F})
%	& \doteq \left\{ x\in\cX_\half \, \biggl| \, \sum_{n=1}^\infty |f_n|^2 \, |\langle x,\, \tilde\varphi_n\rangle_\half|^2 < \infty \right\}\,,
%	\label{eq:dom-Riesz}
\end{align}
where the set $\{f_n\}_{n\in\N}\subset\R$ of eigenvalues of $\opbreve{F}$ is simple and has a totally disconnected closure (i.e. no two elements of this closure can be joined by a segment lying entirely within it), and $\optilde{B} = \{\tilde\varphi_n\}_{n\in\N}$ (enumerating the corresponding eigenvectors of $\opbreve{F}$) is the orthonormal Riesz basis defined by \er{eq:basis-Riesz}. This type of operator is closed and densely defined on $\cX_\half$, see \cite[Example 2.1.13, p.29]{CZ:95}, and is referred to as a {\em Riesz-spectral operator} on $\cX_\half$, see \cite[Definition 2.3.4, p.41]{CZ:95}. Operators $\op{A}$ and $\opAsqrt$ are Riesz-spectral operators, and may be similarly represented, see \cite[Theorem 2.3.5]{CZ:95} and Lemma \ref{lem:specific-op-Riesz} below. The identity $\op{I}$ also takes the form \er{eq:op-Riesz}, with
\begin{align}
	x = \op{I}\, x
	& \doteq \sum_{n=1}^\infty \langle x,\, \tilde\varphi_n \rangle_\half \, \tilde\varphi_n\,,
	\quad \dom(\op{I}) \equiv \cX_\half\,.
	\label{eq:identity-Riesz}
	% \\
	% \dom(\opbreve{I})
	% & \doteq  \left\{ x\in\cX_\half \, \biggl| \, \sum_{n=1}^\infty |\langle x,\, \tilde\varphi_n\rangle_\half|^2 < \infty \right\} \equiv \cX_\half\,,
	% \label{eq:dom-identity}
\end{align}
However, $\op{I}$ is not a Riesz-spectral operator (its eigenvalues are repeated at $1$, and so are not simple). Nevertheless, $\|x\|_\half^2 = \sum_{n=1}^\infty |\langle x,\, \tilde\varphi_n \rangle_\half|^2$ for all $x\in\cX_\half$, see \cite[Corollary 2.3.3, p.40]{CZ:95}.

The remainder of this appendix documents some useful properties of Riesz-spectral operators that are applied in the main body of the paper. Unless otherwise indicated, proofs of these properties are considered standard and are omitted.

%%% if falsed lemma -- no longer cited without omitted proofs %%%
\if{false}

\begin{lemma}
\label{lem:inner-closed}
The functional $\langle \xi,\, \cdot\rangle_\half:\cX_\half\mapsinto\R$ is closed for any fixed $\xi\in\cX_\half$.
\end{lemma}

\begin{proof}
Fix $\xi\in\cX_\half$, $\xi\ne 0$, and define $\pi_\xi:\cX_\half\mapsinto\R$ by $\pi_\xi(\cdot) = \langle \xi,\, \cdot\rangle_\half$. Let $\{x_n\}_{n=1}^\infty$ denote a convergence sequence in $\cX_\half$, such that $\lim_{n\rightarrow\infty} x_n = \bar x\in\cX_\half$, and define the sequence $\{y_n\}_{n=1}^\infty$ in $\R$ by $y_n = \pi_\xi(x_n) = \langle \xi,\, x_n\rangle_\half$. By definition, note that for $n,m\in\N$, $y_n - y_m = \langle \xi,\, x_n - x_m\rangle_\half$, so that
\begin{align}
	|y_n - y_m|
	& \le | \langle \xi,\, x_n - x_m\rangle_\half | \le \|\xi\|_\half\, \|x_n - x_m\|_\half\,.
	\label{eq:closed-y-diff}
\end{align}
As $\{x_n\}_{n=1}^\infty$ is convergent, it is also a Cauchy sequence. Consequently, given $\eps\in\R_{>0}$ and $\xi$ as specified above, there exists a $N_\xi(\eps)\in\N$ such that $\|x_n - x_m\|_\half < \eps / \|\xi\|_\half$ for all $n,m>N_\xi(\eps)$. Combining this with \er{eq:closed-y-diff}, $|y_n - y_m | \le \|\xi\|_\half\, \|x_n - x_m\|_\half < \|\xi\|_\half\, (\eps / \|\xi\|_\half) = \eps$ for all $n,m>N_\xi(\eps)$. That is, $\{y_n\}_{n=1}^\infty$ is a Cauchy sequence in $\R$, and hence there exists a $\bar y\in\R$ such that $\lim_{n\rightarrow\infty} y_n = \bar y$. Define $\hat y \doteq \pi_\xi(\bar x) = \langle \xi,\, \bar x \rangle_\half$. Applying the triangle inequality,
\begin{align}
	|\hat y - \bar y|
	& = | \hat y - y_n + y_n - \bar y| 
	\le |\hat y - y_n| + |y_n - \bar y| 
	\nn\\
	& = | \langle \xi,\, \bar x - x_n \rangle_\half | + | y_n - \bar y |
	\nn\\
	& \le \|\xi\|_\half\, \|x_n - \bar x \|_\half + |y_n - \bar y|\,.
	\label{eq:y-hat-y-bar}
\end{align}
Again fix any $\eps\in\R_{>0}$. By definition, there exists $M_\xi(\eps)\in\N$ such that $\|x_n - \bar x\|_\half < \ts{\frac{\eps}{2\, \|\xi\|_\half}}$ and $|y_n - \bar y| < \ts{\frac{\eps}{2}}$ for all $n>M_\xi(\eps)$. Applying this in \er{eq:y-hat-y-bar} yields that $|\hat y - \bar y| \le \|\xi\|_\half \, \|x_n - \bar x\|_\half +  |y_n - \bar y| < \ts{\frac{\eps}{2}} + \ts{\frac{\eps}{2}} = \eps$ for all $n>M_\xi(\eps)$. As $\eps\in\R_{>0}$ is arbitrary, it follows that $\hat y = \bar y$. So, recalling the definitions of $\bar x$, $\bar y$ and $\hat y$,
\begin{align}
	\left. \begin{aligned}
		\bar x & = \lim_{n\rightarrow\infty} x_n
		\\
		\bar y & = \lim_{n\rightarrow\infty} \pi_\xi(x_n)
	\end{aligned} \right\}
	& \quad\Longrightarrow\quad
	\bar y = \pi_\xi(\bar x)\,.
	\nn
\end{align}
That is, $\pi_\xi$ is closed for $\xi\ne 0$. For the case where $\xi = 0$, note that $\pi_\xi(\cdot) = 0$ is trivially closed.
\end{proof}

\fi
%%% end of if falsed lemma %%%

\begin{lemma}
\label{lem:dom-Riesz}
The domain $\dom(\opbreve{F})$ of a Riesz-spectral operator $\opbreve{F}$ of the form \er{eq:op-Riesz} is equivalently given by
\begin{align}
	\dom(\opbreve{F})
	& = \left\{ x\in\cX_\half \, \left| \, \|\opbreve{F}\, x\|_\half < \infty \right. \right\}\,,
	\qquad
	\|\opbreve{F}\, x\|_\half^2 = 	
	\sum_{n=1}^\infty |f_n|^2 \, |\langle x,\, \tilde\varphi_n\rangle_\half|^2\,.
	\label{eq:dom-Riesz}
	% \label{eq:norm-Riesz}
\end{align}
\end{lemma}

%%% if falsed proof %%%
\if{false}

\begin{proof}
Applying the definition \er{eq:op-Riesz} of the Riesz-spectral operator $\opbreve{F}$,
\begin{align}
	\| \opbreve{F}\, x \|_\half^2
	% & = \langle \opbreve{F}\, x,\, \opbreve{F}\, x\rangle_\half
	& = \left\langle \sum_{m=1}^\infty f_m \, \langle x,\, \tilde\varphi_m \rangle_\half\, \tilde\varphi_m, \opbreve{F}\, x \right\rangle_\half
	% = \sum_{m=1}^\infty f_m \langle x,\, \tilde\varphi_m \rangle_\half \, \langle \tilde\varphi_m,\, \opbreve{F}\, x \rangle_\half
	% \nn\\
	% & 
	= \sum_{m=1}^\infty f_m \langle x,\, \tilde\varphi_m \rangle_\half \, 
	\left\langle \tilde\varphi_m,\, \sum_{n=1}^\infty f_n \, \langle x,\, \tilde\varphi_n \rangle_\half\, \tilde\varphi_n \right\rangle_\half
	\nn\\
	& = \sum_{m=1}^\infty \sum_{n=1}^\infty f_m \, f_n \, \langle x,\, \tilde\varphi_m \rangle_\half \, 
	\langle x,\, \tilde\varphi_n \rangle_\half \,
	\langle \tilde\varphi_m,\, \tilde\varphi_n\rangle_\half
	= \sum_{m=1}^\infty |f_m|^2 \, | \langle x,\, \tilde\varphi_m \rangle_\half |^2\,,
	% \label{eq:norm-Riesz}
\end{align}
in which Lemma \ref{lem:inner-closed} has been applied twice. Recalling the domain $\dom(\opbreve{F})$ of \er{eq:dom-Riesz} completes the proof.
\end{proof}

\fi
%%% end of if falsed proof %%%

%%% if false lemma -- no longer used due to omitted proofs %%%
\if{false}

\begin{lemma}
\label{lem:self-adjoint-Riesz}
Any Riesz-spectral operator $\opbreve{F}$ of the form \er{eq:op-Riesz} is self-adjoint.
\end{lemma}

\begin{proof}
Fix $x,\xi\in\dom(\opbreve{F})\subset\cX_\half$. Recall that the functional $\langle y,\, \cdot\rangle_\half:\cX_\half\mapsinto\R$ is closed by Lemma \ref{lem:inner-closed} for $y\in\{x,\xi\}$ fixed. Hence, by inspection of \er{eq:op-Riesz},
\begin{align}
	\langle \xi,\, \opbreve{F}\, x \rangle_\half
	& = \left\langle \xi,\, \sum_{n=1}^\infty f_n \, \langle x,\, \tilde\varphi_n \rangle \tilde\varphi_n \right\rangle_\half 
	= \sum_{n=1}^\infty f_n \langle x,\, \tilde\varphi_n \rangle_\half\, \langle \xi,\, \tilde\varphi_n \rangle_\half 
	= \left\langle \sum_{n=1}^\infty f_n \, \langle \xi,\, \tilde\varphi_n \rangle_\half\, \tilde\varphi_n,\, x \right\rangle_\half
	= \langle \opbreve{F}\, \xi,\, x \rangle_\half\,,
	\nn
\end{align}
as required.
\end{proof}

\fi
%%% end of if falsed lemma %%%

\begin{lemma}
\label{lem:compose-Riesz}
Let $\opbreve{F}$ and $\opbreve{G}$ denote two Riesz-spectral operators of the form \er{eq:op-Riesz}, with respective point spectra $\sigma_p(\opbreve{F}) = \{ f_n \}_{n=1}^\infty$ and $\sigma_p(\opbreve{G}) = \{ g_n \}_{n=1}^\infty$, and domains $\dom(\opbreve{F})$, $\dom(\opbreve{G})$ as per \er{eq:dom-Riesz}. Suppose additionally that $\{ f_n\, g_n \}_{n=1}^\infty$ is simple, and  its closure is totally disconnected. Then, the composition $\opbreve{F}\, \opbreve{G}$ is also a Riesz-spectral operator of the form \er{eq:op-Riesz}, with
\begin{align}
	\opbreve{F}\, \opbreve{G}\, x
	& = \sum_{n=1}^\infty f_n\, g_n\, \langle x,\, \tilde\varphi_n \rangle_\half \, \tilde\varphi_n\,,
	\quad
	\dom(\opbreve{F}\, \opbreve{G})
	= \left\{ x\in\dom(\opbreve{G})\subset\cX_\half\, \biggl| \, 
	\opbreve{G}\, x\in\dom(\opbreve{F}) \right\}.
	\label{eq:compose-Riesz}
\end{align}
\end{lemma}

%%% if falsed proof %%%
\if{false}

\begin{proof}
By definition, $x\in\dom(\opbreve{F}\, \opbreve{G})$ if and only if $x\in\dom(\opbreve{G})$ and $\opbreve{G}\, x\in\dom(\opbreve{F})$. That is, $\dom(\opbreve{F}\, \opbreve{G})$ is as per \er{eq:compose-Riesz}. By definition \er{eq:op-Riesz} of $\opbreve{F}$, note that $\opbreve{G}\, x\in\dom(\opbreve{F})$ if and only if $\opbreve{F}\, \opbreve{G}\, x\in\cX_\half$. Recall that the functional $\langle \cdot,\, \tilde\varphi_n \rangle_\half:\cX_\half\mapsinto\R$, $n\in\N$, is closed by Lemma \ref{lem:inner-closed}. Hence, for any $x\in\dom(\opbreve{F}\, \opbreve{G})$,
\begin{align}
	\langle \opbreve{G}\, x ,\, \tilde\varphi_n \rangle_\half
	& = \left\langle \lim_{N\rightarrow\infty} \sum_{k=1}^N g_k \, \langle x,\, 
				\tilde\varphi_k \rangle_\half \, \tilde\varphi_k,\, \tilde\varphi_n \right\rangle_\half
	= \lim_{N\rightarrow\infty} \sum_{k=1}^N g_k \, \langle x,\, \tilde\varphi_k \rangle_\half \, 
							\langle \tilde\varphi_k,\, \tilde\varphi_n \rangle_\half
	= g_n \, \langle x,\, \tilde\varphi_n \rangle_\half\,.
	\nn
\end{align}
Consequently, by definition \er{eq:op-Riesz} of $\opbreve{F}$ and $\opbreve{G}$,
\begin{align}
	\opbreve{F}\, \opbreve{G}\, x
	& = 	\opbreve{F}\, \sum_{n=1}^\infty g_n\, \langle x,\, \tilde\varphi_n \rangle_\half\, \tilde\varphi_n
	= \sum_{n=1}^\infty g_n \, \langle x,\, \tilde\varphi_n\rangle_\half \, \opbreve{F}\, \tilde\varphi_n
	= \sum_{n=1}^\infty f_n\, g_n\, \langle x,\, \tilde\varphi_n \rangle_\half \, \tilde\varphi_n\,,
	\nn
\end{align}
as $\opbreve{F}$ is closed and linear. That is, $\opbreve{F}\, \opbreve{G}\, x$ is of the form \er{eq:compose-Riesz}. 
% Furthermore, as $\opbreve{G}$ is also closed and linear, the composition $\opbreve{F}\, \opbreve{G}$ must also be closed and linear. 
By inspection, the point spectrum $\sigma_p(\opbreve{F}\, \opbreve{G})$ of this composition is $\{f_n\, g_n\}_{n=1}^\infty$, with each eigenvalue corresponding to an eigenvector similarly enumerated in the Riesz basis $\optilde{B}$. By hypothesis, as this point spectrum is simple, and its closure totally disconnected, $\opbreve{F}\, \opbreve{G}$ must also be a Riesz-spectral operator by definition. (Furthermore, as any Riesz-spectral operator is closed, also by definition, it follows that $\opbreve{F}\, \opbreve{G}$ must be closed.)
\end{proof}

\fi
%%% end of if falsed proof. %%%

\begin{corollary}
\label{cor:compose-dom-Riesz} Let $\opbreve{F}$ and $\opbreve{G}$ be Riesz-spectral operators as per Lemma \ref{lem:compose-Riesz}. Then, the domain of the composition $\opbreve{F}\, \opbreve{G}$ of  \er{eq:compose-Riesz} is equivalently given by
\begin{align}
	\dom(\opbreve{F}\, \opbreve{G})
	& = \left\{ x\in\cX_\half \, \left| \, 
			\sum_{n=1}^\infty (1 + |f_n|^2)\, |g_n|^2 \, |\langle x,\, \tilde\varphi_n \rangle_\half|^2 < \infty 
		\right. \right\}.
	\label{eq:compose-dom-Riesz-sum}
\end{align}
If additionally there exists $f_-\in\R_{>0}$ such that $|f_n| \ge f_-$ for all $n\in\N$, then the domain $\dom(\opbreve{F}\, \opbreve{G})$ specified via \er{eq:op-Riesz} or \er{eq:compose-dom-Riesz-sum} is equivalently given by
\begin{align}
	\dom(\opbreve{F}\, \opbreve{G})
	& = \left\{ x\in\cX_\half\, \left| \, \sum_{n=1}^\infty |f_n\, g_n|^2 \, |\langle x,\, \tilde\varphi_n\rangle_\half|^2 < \infty \right. \right\}.
	\label{eq:compose-dom-Riesz-sum-simple}
\end{align}
\end{corollary}
\begin{proof}
Recalling \er{eq:compose-Riesz}, $x\in\dom(\opbreve{F}\, \opbreve{G})$ if and only if $x\in\dom(\opbreve{G})$ and $\opbreve{G}\, x\in\dom(\opbreve{F})$. These respective properties hold if and only if
\begin{align}
	\infty
	& > \sum_{n=1}^\infty |g_n|^2\, |\langle x,\, \tilde\varphi_n \rangle_\half|^2\,,
	\quad
	\infty 
	> \sum_{n=1}^\infty |f_n|^2 \, |\langle \opbreve{G}\, x,\, \tilde\varphi_n \rangle_\half|^2
	= \sum_{n=1}^\infty |f_n \, g_n|^2 \, |\langle x,\, \tilde\varphi_n \rangle_\half|^2\,.
	\label{eq:compose-dom-sums}
\end{align}
So, the domain of $\opbreve{F}\, \opbreve{G}$ is given by
\begin{align}
	\dom(\opbreve{F}\, \opbreve{G}) 
	& = \left\{ x\in\dom(\opbreve{G})\subset\cX_\half\, \biggl| \, \opbreve{G}\, x\in\dom(\opbreve{F}) \right\}
	= \left\{ x\in\dom(\opbreve{G})\subset\cX_\half \, \left| \, 
			\sum_{n=1}^\infty |f_n\, g_n|^2 \, | \langle x,\, \tilde\varphi_n \rangle_\half|^2 < \infty
		\right. \right\}
	\nn\\
	& = \left\{ x\in\cX_\half\, \left| \,
			\sum_{n=1}^\infty (1 + |f_n|^2) \, |g_n|^2\, | \langle x,\, \tilde\varphi_n \rangle_\half|^2 < \infty
		\right. \right\},
	\nn
\end{align}
as specified by \er{eq:compose-dom-Riesz-sum}. Suppose additionally that there exists $f_-\in\R_{>0}$ such that $|f_n|\ge f_-$ for all $n\in\N$. Define the domain candidate $\funspace{D}$ as per \er{eq:compose-dom-Riesz-sum-simple}, that is
$
	\funspace{D}
	\doteq \left\{ x\in\cX_\half\, \left| \, \sum_{n=1}^\infty |f_n\, g_n|^2 \, |\langle x,\, \tilde\varphi_n\rangle_\half|^2 < \infty \right. \right\}
$.
Fix any $x\in\dom(\opbreve{F}\, \opbreve{G})$ via \er{eq:compose-dom-Riesz-sum}. By inspection, it immediately follows that $x\in\funspace{D}$. That is, $\dom(\opbreve{F}\, \opbreve{G})\subset\funspace{D}$. In order to prove the opposite direction, fix any $x\in\funspace{D}$, and note that the second inequality in \er{eq:compose-dom-sums} implies the first. In particular,
\begin{align}
	\infty 
	& > \sum_{n=1}^\infty |f_n \, g_n|^2 \, |\langle x,\, \tilde\varphi_n \rangle_\half|^2 
	> f_-^2 \, \sum_{n=1}^\infty |g_n|^2 \, |\langle x,\, \tilde\varphi_n \rangle_\half|^2
	\quad
	\stackrel{f_-\in\R_{>0}}{\Longrightarrow}
	\quad
	\infty >
	\sum_{n=1}^\infty |g_n|^2 \, |\langle x,\, \tilde\varphi_n \rangle_\half|^2\,,
\end{align}
which implies that
$
	\infty
	> \sum_{n=1}^\infty |f_n \, g_n|^2 \, |\langle x,\, \tilde\varphi_n \rangle_\half|^2
	+ \sum_{n=1}^\infty |g_n|^2 \, |\langle x,\, \tilde\varphi_n \rangle_\half|^2
	= \sum_{n=1}^\infty (1 + |f_n|^2)\, |g_n|^2 \, |\langle x,\, \tilde\varphi_n \rangle_\half|^2
$. 
Consequently, $x\in\funspace{D}$ implies that $x\in\dom(\opbreve{F}\, \opbreve{G})$, or $\funspace{D}\subset\dom(\opbreve{F}\, \opbreve{G})$. Combining this with the earlier conclusion that $\funspace{D}\supset\dom(\opbreve{F}\, \opbreve{G})$ yields that $\funspace{D} \equiv \dom(\opbreve{F}\, \opbreve{G})$. That is, \er{eq:compose-dom-Riesz-sum-simple} holds as required.
\end{proof}

\begin{lemma}
\label{lem:inverse-Riesz}
Let $\opbreve{F}$, $\opbreve{F}^\sharp$ denote a pair of Riesz-spectral operators of the form \er{eq:op-Riesz}, with point spectra $0\not\in\sigma_p(\opbreve{F}) = \{ f_n \}_{n=1}^\infty$ and $\sigma_p(\opbreve{F}^\sharp) = \{ \ts{\frac{1}{f_n}} \}_{n=1}^\infty$. Then,
\begin{align}
	\opbreve{F}\, \opbreve{F}^\sharp\, x & = x\,,
	\quad \forall \ x\in\dom(\opbreve{F}^\sharp)\,,
	&
	\opbreve{F}^\sharp\, \opbreve{F}\, x & = x\,,
	\quad \forall \ x\in\dom(\opbreve{F})\,.
	\label{eq:inverse-Riesz}
\end{align}
\end{lemma}

%%% if falsed proof %%%
\if{false}

\begin{proof}
Fix any $x\in\dom(\opbreve{F}\, \opbreve{F}^\sharp)$. Equivalently, $x\in\dom(\opbreve{F}^\sharp)$ and $\opbreve{F}^\sharp\, x\in\dom(\opbreve{F})$. Recalling that $\langle\cdot,\, \tilde\varphi_n\rangle_\half:\cX_\half\mapsinto\R$, $n\in\N$, is closed by Lemma \ref{lem:inner-closed}, it follows respectively that
\begin{align}
	\infty
	& > \sum_{n=1}^\infty |\ts{\frac{1}{f_n}}|^2 \, |\langle x,\, \tilde\varphi_n \rangle_\half|^2\,,
	% \label{eq:inv-R-1}
	\quad
	\infty
	> \sum_{n=1}^\infty |f_n|^2 \, |\langle \opbreve{F}^\sharp \, x,\, \tilde\varphi_n \rangle_\half|^2
	= \sum_{n=1}^\infty |f_n|^2\, |\ts{\frac{1}{f_n}}|^2 \, |\langle x,\, \tilde\varphi_n \rangle_\half|^2
	= \sum_{n=1}^\infty |\langle x,\, \tilde\varphi_n \rangle_\half|^2\,.
	\nn
	% \label{eq:inv-R-2}
\end{align}
Note that the second inequality above holds if and only if $x\in\cX_\half$, and so is automatically fulfilled. Hence,
\begin{align}
	x\in\dom(\opbreve{F}\, \opbreve{F}^\sharp)
	& \quad\Longleftrightarrow\quad 
	\infty
	> \sum_{n=1}^\infty |\ts{\frac{1}{f_n}}|^2 \, |\langle x,\, \tilde\varphi_n \rangle_\half|^2
	\quad\Longleftrightarrow\quad
	x\in\dom(\opbreve{F}^\sharp)\,.
	\nn
\end{align}
Consequently, $\dom(\opbreve{F}\, \opbreve{F}^\sharp) \equiv \dom(\opbreve{F}^\sharp)$. So, for any $x\in\dom(\opbreve{F}^\sharp)$, again applying Lemma \ref{lem:inner-closed} and the fact that $\opbreve{F}$ is closed,
$
	\opbreve{F}\, \opbreve{F}^\sharp\, x
	= \sum_{n=1}^\infty f_n \, \langle \opbreve{F}^\sharp\, x,\, \tilde\varphi_n \rangle_\half \, \tilde\varphi_n
	= \sum_{n=1}^\infty f_n \, \ts{\frac{1}{f_n}} \, \langle x,\, \tilde\varphi_n\rangle_\half\, \tilde\varphi_n
	= \sum_{n=1}^\infty \langle x,\, \tilde\varphi_n\rangle_\half\, \tilde\varphi_n
	= x
$,
as per the first identity in \er{eq:inverse-Riesz}. The second identity in \er{eq:inverse-Riesz} follows similarly.
\end{proof}

\fi
%%% end of if falsed proof %%%

\begin{corollary}
\label{cor:unique-inverse-Riesz}
A Riesz operator $\opbreve{F}$ on $\cX_\half$ with point spectrum satisfying $0\notin\sigma_p(\opbreve{F}) = \{f_n\}_{n=1}^\infty$ is invertible. Furthermore, its inverse $\opbreve{F}^{-1}$ is also a Riesz-spectral operator on $\cX_\half$, and is given by
\begin{align}
	\opbreve{F}^{-1}\, x
	& = \sum_{n=1}^\infty \ts{\frac{1}{f_n}}\, \langle x,\, \tilde\varphi_n\rangle_\half \tilde\varphi_n\,,
	\quad
	\dom(\opbreve{F}^{-1}) 
	= \left\{ x\in\cX_\half \, \biggl|\, \sum_{n=1}^\infty |\ts{\frac{1}{f_n}}|^2 \, |\langle x,\, \tilde\varphi_n \rangle_\half|^2 < \infty \right\}.
	\label{eq:unique-inverse-Riesz}
\end{align}
\end{corollary}

%%% if falsed proof %%%
\if{false}

\begin{proof}
The existence of an inverse $\opbreve{F}^\sharp$ of $\opbreve{F}$ follows by Lemma \ref{lem:inverse-Riesz}. Uniqueness of this inverse, and hence the representation \er{eq:unique-inverse-Riesz} follows by virtue of the fact that $\kernel(\opbreve{F}) = \{ 0 \} = \kernel(\opbreve{F}^\sharp)$.
\end{proof}

\fi
%%% end of if falsed proof %%%

% \begin{remark}
It is well known by Riesz's Lemma that the identity $\op{I}$ of \er{eq:identity-Riesz} is not a Riesz-spectral operator on $\cX_\half$.
Consequently, the composition of a Riesz-spectral operator $\opbreve{F}$ and its inverse $\opbreve{F}^{-1}$ (also a Riesz-spectral operator, by Corollary \ref{cor:unique-inverse-Riesz}) is not itself a Riesz-spectral operator. Indeed, in attempting to apply Lemma \ref{lem:compose-Riesz} to such a composition  $\opbreve{F}\, \opbreve{F}^{-1}$ reveals that its point spectrum $\sigma_p(\opbreve{F}\, \opbreve{F}^{-1}) = \{ 1 \}$ is not simple, thereby violating the definition of a Riesz-spectral operator.
% \end{remark}

% ?? REWRITE THE FOLLOWING LEMMA ?? i.e.

\begin{lemma}
\label{lem:specific-op-Riesz}
% \label{lem:op-A-and-opAsqrt-Riesz}
$\op{A}$, $\opAsqrt$, $\op{J}$, $\op{I}_\mu$, $\op{I}_\mu^\half$, $\op{I}_\mu^{-\half}$, and $\opAsqrt\, \op{I}_\mu\, \opAsqrt$ are Riesz-spectral operators of the form \er{eq:op-Riesz} on $\cX$ and $\cX_\half$, with respective eigenvalues given by $\lambda_n$, $\lambda_n^\half$, $\lambda_n^{-\half}$, $(1 + \mu^2\, \lambda_n)^{-1}$, $(1 + \mu^2\, \lambda_n)^{\half}$, $(1 + \mu^2\, \lambda_n)^{-\half}$, and $\lambda_n\, (1 + \mu^2\, \lambda_n)^{-1}$ for all $n\in\N$, where $\lambda_n$ is as per \er{eq:eigenvalues}.
\end{lemma}
\begin{proof}
Operator $\op{A}$ is closed and linear on $\cX$, with simple eigenvalues $\sigma_p(\op{A}) = \{\lambda_n\}_{n=1}^\infty\subset\R$ defined by \er{eq:eigenvalues} and corresponding to eigenvectors $\op{B} = \{\varphi_n\}_{n=1}^\infty\subset\cX$ as per \er{eq:basis-Riesz}. The closure of the point spectrum of $\op{A}$, denoted by $\ol{\sigma_p(\op{A})}$, is totally disconnected, and $\op{B}$ forms a Riesz basis for $\cX$, see Lemmas \ref{lem:op-A-properties}, \ref{lem:eigenvalues}, and \ref{lem:basis-Riesz}. Hence, operator $\op{A}$ is a Riesz-spectral operator on $\cX$ (see also\cite[Definition 2.3.4, p.41]{CZ:95}), and operator $\op{A}$ and its domain $\dom(\op{A}) = \cX_0$ may be represented as per \er{eq:op-Riesz} with the aforementioned eigenvalues. An analogous argument for $\op{A}$ defined in $\cX_\half$, with eigenvectors $\optilde{B}$ as per \er{eq:basis-Riesz} corresponding to the same eigenvalues $\{\lambda_n\}_{n\in\N}$ and forming a Riesz basis for $\cX_\half$, yields that $\op{A}$ is also a Riesz-spectral operator on $\cX_\half$. A similar argument yields that $\opAsqrt$ is a Riesz-spectral operator on $\cX$ and $\cX_\half$. As $\opAsqrt\, \op{J} = \op{I}$, Corollary \ref{cor:unique-inverse-Riesz} implies that $\op{J}$ is similarly a Riesz-spectral operator on $\cX$ and $\cX_\half$.

In order to show that the remaining operators are Riesz-spectral operators, first note that $\op{I}_\mu^{-1} = \op{I} + \mu^2\, \op{A}$ defined via \er{eq:op-I-mu} is a Riesz-spectral operator, with eigenvalues $\{ 1 + \mu^2 \, \lambda_n \}_{n\in\N}$. This follows by \er{eq:identity-Riesz} and the fact that $\op{A}$ is a Riesz-spectral operator. Consequently, Corollary \ref{cor:unique-inverse-Riesz} implies that $\op{I}_\mu$ is also a Riesz-spectral operator. Lemma \ref{lem:op-I-mu-properties} states that $\op{I}_\mu$ has a unique square-root $\op{I}_\mu^\half$. Subsequently applying Lemma \ref{lem:compose-Riesz} and Corollary \ref{cor:unique-inverse-Riesz} implies that both $\op{I}_\mu^{\half}$ and $\op{I}_\mu^{-\half}$ are Riesz-spectral operators. 
The fact that the composition $\opAsqrt\, \op{I}_\mu\, \opAsqrt$ is a Riesz-spectral operator follows by two further applications of Lemma \ref{lem:compose-Riesz}.
\if{false}

The domain $\dom(\op{A})$ is
\begin{align}
	\dom(\op{A})
	& = \left\{ x\in\cX \, \biggl| \, \op{A}\, x \in \cX \right\}
	= \left\{ x\in\cX \, \left| \, \sum_{n=1}^\infty |\lambda_n|^2 \, |\langle x,\, \varphi_n \rangle|^2 < \infty \right. \right\}
	= \cX_0\,,
	\nn\\
	\dom(\op{A})
	& = \left\{ x\in\cX_\half \, \biggl| \, \op{A}\, x\in\cX_\half \right\}
	= \left\{ x\in\cX_\half \, \left| \, \sum_{n=1}^\infty |\lambda_n|^2 \, |\langle x,\, \tilde\varphi_n \rangle_\half|^2 < \infty \right. \right\}
	\subset\cX_0\,.
	\nn
\end{align}

\fi
\end{proof}

\if{false}
 
\begin{lemma}
% \label{lem:op-A-and-opAsqrt-Riesz}
Operators $\op{A}$ and $\opAsqrt$ are Riesz-spectral operators on $\cX$, with
\begin{align}
	\op{A}\, x
	& = \sum_{n=1}^\infty \lambda_n\, \langle x,\, \varphi_n \rangle \, \varphi_n\,,
	& 
	& \dom(\op{A}) 
	= \left\{ x\in\cX \, \left| \, \sum_{n=1}^\infty |\lambda_n|^2 \, |\langle x,\, \varphi_n \rangle|^2 < \infty \right. \right\}
	= \cX_0\,,
	% \label{eq:op-A-Riesz-on-X}
	\\
	\opAsqrt\, x
	& = \sum_{n=1}^\infty \sqrt{\lambda_n}\, \langle x,\, \varphi_n \rangle \, \varphi_n\,,
	& 
	& \dom(\opAsqrt)
	= \left\{ x\in\cX \, \left| \, \sum_{n=1}^\infty |\lambda_n| \, |\langle x,\, \varphi_n \rangle|^2 < \infty \right. \right\} = \cX_\half\,,
	\label{eq:opAsqrt-Riesz-on-X}
\end{align}
or, equivalently, in terms of the $\langle \, , \rangle_\half$ inner product,
\begin{align}
	\op{A}\, x
	& = \sum_{n=1}^\infty \lambda_n\, \langle x,\, \tilde\varphi_n \rangle_\half \, \tilde\varphi_n\,,
	& 
	& \dom(\op{A}) 
	= \left\{ x\in\cX \, \left| \, \sum_{n=1}^\infty |\lambda_n| \, |\langle x,\, \tilde\varphi_n \rangle_\half|^2 < \infty \right. \right\} = \cX_0\,,
	\label{eq:op-A-Riesz-on-X-half}
	\\
	\opAsqrt\, x
	& = \sum_{n=1}^\infty \sqrt{\lambda_n}\, \langle x,\, \tilde\varphi_n \rangle_\half \, \tilde\varphi_n\,,
	& 
	& \dom(\opAsqrt)
	= \left\{ x\in\cX \, \left| \, \sum_{n=1}^\infty |\langle x,\, \tilde\varphi_n \rangle_\half|^2 < \infty \right. \right\} = \cX_\half\,.
	\label{eq:opAsqrt-Riesz-on-X-half}
\end{align}
\end{lemma}

\fi

%%% if falsed proof %%%
\if{false}

\begin{proof}
Operator $\op{A}$ is closed and linear on $\cX$, with simple eigenvalues $\sigma_p(\op{A}) = \{\lambda_n\}_{n=1}^\infty\subset\R$ as per \er{eq:eigenvalues} corresponding to eigenvectors $\op{B} = \{\varphi_n\}_{n=1}^\infty\subset\cX$ as per \er{eq:basis-Riesz}, $\ol{\sigma_p(\op{A})}$ is totally disconnected, and $\op{B}$ forms a Riesz basis for $\cX$, see Lemmas \ref{lem:op-A-properties}, \ref{lem:eigenvalues}, and \ref{lem:basis-Riesz}. Hence, operator $\op{A}$ is a Riesz-spectral operator on $\cX$ (see \cite[Definition 2.3.4, p.41]{CZ:95}), and operator $\op{A}$ and its domain $\cX_0 = \dom(\op{A})$ may be represented via analogues of \er{eq:op-Riesz} and \er{eq:dom-Riesz} defined with respect to $\cX$, see \cite[Theorem 2.3.5, p.41]{CZ:95}. This immediately yields \er{eq:op-A-Riesz-on-X}. Representation \er{eq:opAsqrt-Riesz-on-X} for $\opAsqrt$ on $\cX$ follows similarly.

As $\cX_0\subset\cX_\half\subset\cX$, operator $\op{A}$ and $\opAsqrt$ may be expressed in terms of the inner product $\langle \, , \rangle_\half$ on $\cX_\half$. In particular, with $x\in\cX_0$, note that $\lambda_n\, \langle x,\, \varphi_n \rangle\, \varphi_n = \langle x,\, \lambda_n\, \varphi_n\rangle\, \varphi_n = \langle x,\, \op{A}\, \varphi_n \rangle\, \varphi_n = \langle \opAsqrt\, x,\, \opAsqrt\, \varphi_n \rangle\, \varphi_n = \langle x,\, \varphi_n \rangle_\half\, \varphi_n = \lambda_n \, \langle x,\, \tilde\varphi_n \rangle_\half\, \tilde\varphi_n$, where the last equality follows by \er{eq:basis-Riesz} and \er{eq:eigenvalues}. Hence, applying \er{eq:op-A-Riesz-on-X},
$\op{A}\, x = \sum_{n=1}^\infty \lambda_n \, \langle x,\, \tilde\varphi_n \rangle_\half\, \tilde\varphi_n$. As $\dom(\op{A}) = \{ x\in\cX \, | \, \op{A}\, x \in\cX \}$, following steps analogous to \er{eq:norm-Riesz} implies that $x\in\dom(\op{A})$ if and only if
\begin{align}
	\infty & >
	\| \op{A}\, x\|^2 % = \langle \op{A}\, x,\, \op{A}\, x \rangle
	= \left\langle \sum_{m=1}^\infty \lambda_m \, \langle x,\, \tilde\varphi_m \rangle_\half\, \tilde\varphi_m,\,
			\sum_{n=1}^\infty \lambda_n \, \langle x,\, \tilde\varphi_n \rangle_\half\, \tilde\varphi_n \right\rangle
	= \sum_{m=1}^\infty \sum_{n=1}^\infty \lambda_m\, \lambda_n \, 
			\langle x,\, \tilde\varphi_m \rangle_\half \, \langle x,\, \tilde\varphi_n \rangle_\half \, 
			\langle \tilde\varphi_m,\, \tilde\varphi_n\rangle
	\nn\\
	& = \sum_{m=1}^\infty \sum_{n=1}^\infty \sqrt{ \lambda_m\, \lambda_n}\,
			\langle x,\, \tilde\varphi_m \rangle_\half \, \langle x,\, \tilde\varphi_n \rangle_\half \, 
			\langle \varphi_m,\, \varphi_n\rangle
	= \sum_{n=1}^\infty |\lambda_n|\, |\langle x,\, \tilde\varphi_n \rangle_\half|^2\,.
	\nn
\end{align}
That is, $\dom(\op{A})$ may be equivalently represented as per \er{eq:op-A-Riesz-on-X-half}. Similarly, $x\in\dom(\opAsqrt)$ if and only if
$
	\infty > \|\opAsqrt\, x\|^2 = \sum_{n=1}^\infty |\langle x,\, \tilde\varphi_n\rangle_\half|^2 = \|x\|_\half^2
$,
so that $\dom(\opAsqrt)$ may be equivalently represented as per \er{eq:opAsqrt-Riesz-on-X-half}.
\end{proof}

\fi
%%% end of if falsed proof %%%

\if{false}
% \begin{remark}
By inspection of representation \er{eq:op-A-Riesz-on-X-half} and Lemma \ref{lem:eigenvalues}, $x\in\dom(\op{A})$ if and only if
\begin{align}
	\infty & > \sum_{n=1}^\infty |\langle x,\, \sqrt{\lambda_n}\, \tilde\varphi_n \rangle_\half|^2
	=  \sum_{n=1}^\infty |\langle x,\, \opAsqrt\, \tilde\varphi_n \rangle_\half|^2 
	= \sum_{n=1}^\infty |\langle \opAsqrt\, x,\, \tilde\varphi_n \rangle_\half|^2\,,
	\nn
\end{align}
so that $x\in\cX_0 = \dom(\op{A})$ if and only if $\opAsqrt\, x\in\cX_\half = \dom(\opAsqrt)$.
% \end{remark}
\fi

%%
%%		Riesz-spectral operator-valued functions
%%		

\section{Riesz-spectral operator-valued functions}
\label{app:Riesz-functions}
A Riesz-spectral operator-valued function takes the form
\begin{align}
	\op{F}(t)
	\, x
	& \doteq \sum_{n=1}^\infty f_n(t)\, \langle x,\, \tilde\varphi_n \rangle_\half\, \tilde\varphi_n\,,
	\quad
	\dom(\op{F}(t))
	\doteq \left\{ x\in\cX_\half\, \biggl|\, \op{F}(t)\, x\in\cX_\half \right\},
	\label{eq:op-F}
%	\\
%	\dom(\opbreve{P}_t^c)
%	& \doteq \left\{ x\in\cX_\half \, \biggl| \, \sum_{n=1}^\infty |p_n^{\mu,c}(t)|^2 \, |\langle x,\, \tilde\varphi_n\rangle_\half|^2 < \infty \right\},
%	\label{eq:op-dom}
\end{align}
where $t\in\Omega$, $\tilde\varphi_n\in\optilde{B}$ is as per \er{eq:basis-Riesz}, and $f_n:\Omega\mapsinto\R$, for all $n\in\N$, for some interval $\Omega\subset\R_{\ge 0}$.

\begin{lemma}
\label{lem:op-F}
Suppose that the Riesz-spectral operator-valued function $\op{F}$ defined by \er{eq:op-F} satisfies the following properties with respect to a bounded open interval $\Omega\subset\R_{>0}$:
\begin{enumerate}[(i)]
\item $\{f_n(t)\}_{n\in\N}\subset\R$ is a strictly monotone sequence for every $t\in \Omega$; 
\item $f_n\in C^2(\Omega;\R)$ for all $n\in\N$; and 
\item there exists an $M_f\in\R_{\ge 0}$ such that $\max(|f_n(t)|, \, |\dot f_n(t)|,\, |\ddot f_n(t)|) \le M_f$ for all $n\in\N$ and $t\in\Omega$.
\end{enumerate}
Then, $\op{F}(t)\in\bo(\cX_\half)$ for all $t\in\Omega$, and $\op{F}:\Omega\mapsinto\bo(\cX_\half)$ is {\Frechet} differentiable with derivative $\opdot{F}:\Omega\mapsinto\bo(\cX_\half)$ of the form \er{eq:op-F} given for all $t\in\Omega$ and $x\in\cX_\half$ by
\begin{align}
	\opdot{F}(t)\, x
	& = \sum_{n=1}^\infty \dot f_n(t)\, \langle x,\, \tilde\varphi_n \rangle_\half\, \tilde\varphi_n\,.
	\label{eq:op-F-dot}
\end{align}
\end{lemma}
\begin{proof}
Fix $t\in\Omega$ and $x\in\cX_\half$, $x\ne 0$. By property {\em (i)}, as $\{f_n(t)\}_{n\in\N}\subset\R$ is a strictly monotone sequence, its closure is the union of itself and its supremum or infemum, where the latter is strictly less than or strictly greater than every element of $\{f_n(t)\}_{n\in\N}$. Hence, there always exists at least one open interval between any two distinct elements of $\ol{\{f_n(t)\}_{n\in\N}}$, so that any two such elements cannot be joined by a segment lying entirely within $\ol{\{f_n(t)\}_{n\in\N}}$. That is, $\ol{\{f_n(t)\}_{n\in\N}}$ is totally disconnected. As $\optilde{B} = \{ \tilde\varphi_n \}_{n\in\N}$ is an orthonormal Riesz basis for $\cX_\half$, it follows by \cite[Corollary 2.3.6, p.45]{CZ:95} that $\op{F}(t)$ is a Riesz-spectral operator.
\if{false}

By inspection of \er{eq:op-F}, 
\begin{align}
	\| \op{F}(t)\, x \|_\half^2
	= \langle \op{F}(t)\, x,\, \op{F}(t)\, x\rangle_\half
	& = \left\langle \sum_{m=1}^\infty f_m(t) \, \langle x,\, \tilde\varphi_m \rangle_\half\, \tilde\varphi_m, \op{F}(t)\, x \right\rangle_\half
	= \sum_{m=1}^\infty f_m(t) \langle x,\, \tilde\varphi_m \rangle_\half \, \langle \tilde\varphi_m,\, \op{F}(t)\, x \rangle_\half
	\nn\\
	& = \sum_{m=1}^\infty f_m(t) \langle x,\, \tilde\varphi_m \rangle_\half \, 
	\left\langle \tilde\varphi_m,\, \sum_{n=1}^\infty f_n(t) \, \langle x,\, \tilde\varphi_n \rangle_\half\, \tilde\varphi_n \right\rangle_\half
	\nn\\
	& = \sum_{m=1}^\infty \sum_{n=1}^\infty f_m(t) \, f_n(t) \, \langle x,\, \tilde\varphi_m \rangle_\half \, 
	\langle x,\, \tilde\varphi_n \rangle_\half \,
	\langle \tilde\varphi_m,\, \tilde\varphi_n\rangle_\half
	\nn\\
	& = \sum_{m=1}^\infty |f_m(t)|^2 \, | \langle x,\, \tilde\varphi_m \rangle_\half |^2\,,
	\label{eq:op-F-norm}
\end{align}
in which Lemma \ref{lem:inner-closed} has been applied twice. 

\fi
Applying \er{eq:dom-Riesz} to $\op{F}(t)$, and applying property {\em (iii)} and \er{eq:identity-Riesz}, yields that
\begin{align}
	\| \op{F}(t)\, x \|_\half^2
	& \le \sum_{m=1}^\infty |f_m(t)|^2 \, | \langle x,\, \tilde\varphi_m \rangle_\half |^2
	% \nn\\
% \end{align}
% Hence, property {\em (iii)} and \er{eq:identity-Riesz} imply that
% \begin{align}
%	\|\op{F}(t)\, x\|_\half^2
	% & 
	\le M_f^2\, \sum_{m=1}^\infty |\langle x,\, \tilde\varphi_m\rangle_\half|^2 = M_f^2\, \|x\|_\half^2\,,
	\label{eq:op-F-norm}
\end{align}
or $\op{F}(t)\in\bo(\cX_\half)$ with $\|\op{F}(t)\| \le M_f$. Define $\Omega(t) \doteq \left\{ s\in\R_{>0} \, | \, s+t\in\Omega \right\}$, and fix $\eps\in\Omega(t)$, $\eps\ne 0$. Define $\opdot{F}$ as per \er{eq:op-F-dot}, and note by property {\em (iii)} and \er{eq:op-F-norm} that $\opdot{F}(t)\in\bo(\cX_\half)$ with $\|\opdot{F}(t)\|_\half\le M_f$, where $\|\cdot\|_\half$ here denotes the induced operator norm on $\cX_\half$. Applying property {\em (ii)}, $f_n(t+\eps) = f_n(t) + \eps\, \dot f_n(t) + (\ts{\frac{\eps^2}{2}})\, \ddot f_n(\tau)$ for some $\tau\in(t,t+\eps)$, so that
\begin{align}
	& \left\| [\op{F}(t+\eps) - \op{F}(t) -  \eps\, \opdot{F}(t)]\, x \right\|_\half^2
	= \sum_{n=1}^\infty | f_n(t+\eps) - f_n(t) - \eps\, \dot f_n(t) |^2 \, |\langle x,\, \tilde\varphi_n\rangle_\half|^2
	\nn\\
	& \le (\ts{\frac{\eps^2}{2}})^2\, \sum_{n=1}^\infty \sup_{\tau\in\Omega(t)} | \ddot f_n(\tau) |^2 \, |\langle x,\, \tilde\varphi_n\rangle_\half|^2
	\le (\ts{\frac{\eps^2}{2}})^2\, M_f^2 \, \sum_{n=1}^\infty |\langle x,\, \tilde\varphi_n\rangle_\half|^2
	= (\ts{\frac{\eps^2}{2}})^2\, M_f^2 \, \|x\|_\half^2\,.
	\nn
\end{align}
Consequently, dividing through by $\eps\ne 0$ and $\|x\|_\half\ne 0$,
\begin{align}
	\lim_{\eps\rightarrow 0} \frac{\left\| \op{F}(t+\eps) - \op{F}(t) -  \eps\, \opdot{F}(t) \right\|_\half}{|\eps|}
	& = 
	\lim_{\eps\rightarrow 0} \sup_{\|x\|_\half\ne 0} \frac{\left\| [\op{F}(t+\eps) - \op{F}(t) -  \eps\, \opdot{F}(t)]\, x \right\|_\half}{|\eps|\, \|x\|_\half}
	\le \lim_{\eps\rightarrow 0} \ts{\frac{|\eps|}{2}}\, M_f = 0\,,
	\nn
\end{align}
in which the left-hand norm is again the induced operator norm on $\cX_\half$, thereby demonstrating that $\opdot{F}$ is indeed the {\Frechet} derivative of $\op{F}$.
\end{proof}

%%%%%%%%%%%%%%%%%%%%%%%%%%%%%%%%%%%%%%%%%%%%%%%%%%%%%%%%%%%%%%%%%%%%%%%%%%
%%
%%		End of document
%%

\end{document}